\documentclass[leqno,11pt,a4paper]{amsart}

\usepackage{jc}
\usepackage[mathcal]{eucal}
\usepackage{comment}

\begin{document}

\author{Julian Chaidez}
\address{Department of Mathematics\\University of Southern California\\Los Angeles, CA\\90007\\USA}
\email{julian.chaidez@usc.edu}

\author{Michael Huang}
\address{Department of Mathematics\\University of Southern California\\Los Angeles, CA\\90007\\USA}
\email{mshuang@usc.edu}

\title[Convex Hypersurfaces And Robust Heterodimensional Dynamics]{Convex Hypersurfaces And Robust Heterodimensional Dynamics}

\begin{abstract} We prove that any closed orientable hypersurface in a contact manifold of dimension five or greater is isotopic to a robustly non-convex hypersurface via an arbitrarily $C^0$-small isotopy. This strengthens a recent result of the first author and yields a strong counterpart to the groundbreaking density theorem of Honda-Huang and Giroux. This is proven by combining a new convexity obstruction via heteroclinics and recent advances in robust heterodimensional dynamics due to Li-Turaev to produce a robust deconvexifying plug, which is a local and robust convexity obstruction.
\end{abstract}

\vspace*{-15pt}

\maketitle

\vspace*{-20pt}

%\tableofcontents

\section{Introduction} \label{sec:introduction} 
A hypersurface in a contact manifold is \emph{convex} if there exists a contact vector field that is everywhere transverse to the hypersurface. Convex surface theory was originally pioneered by Giroux \cite{g1991}, who established several fundamental properties of convex surfaces and applied them to resolve (in dimension three) a question of Eliashberg-Gromov \cite{eliashberg1991convex} on the existence of contact Morse functions. Convex surfaces subsequently developed into a core tool in low-dimensional contact topology, with many fundamental applications (cf. \cite{h1999,h2000,hkm2002,hkm2003,hkm2007,e2001,eh2002,eh2005}). We refer the reader to the monograph of Etnyre \cite{etnyre} for a more comprehensive reference.

\vspace{3pt}

In higher dimensions, convex hypersurface theory has only recently gained traction due to the groundbreaking work of Honda-Huang \cite{hh2018,hh2019}. Applications have included the completion of the higher-dimensional Giroux correspondence by Breen-Honda-Huang \cite{bhh2023}. A central question in all of these works regards the abundance of convex hypersurfaces. More precisely, we have the following question about density and genericity \cite[Remark 1.2.4]{hh2019} or \cite{AIMproblems}.

\begin{question*}[Density/Genericity] \label{qu:density} Are convex hypersurfaces dense (or even generic) in the space of closed oriented hypersurfaces in any contact manifold in the $C^k$-topology?
\end{question*}

\noindent One of the fundamental results of Giroux \cite{g1991} states that convex surfaces are $C^\infty$-generic in dimension three, and the main results of Honda-Huang \cite{hh2018} established the $C^0$-density of convex hypersurfaces in any dimension. On the other hand, it was long conjectured that Question \ref{qu:density} should have a negative answer in the $C^\infty$-topology beyond dimension three. This longstanding problem was finally resolved by the first author \cite{jc2024} who recently proved the following result.

\begin{theorem*} \cite{jc2024} \label{thm:chaidez_nonconvex} There is a closed oriented hypersurface $\Sigma$ in any contact manifold $(Y,\xi)$ of dimension five or greater that cannot be approximated by convex hypersurfaces in the $C^2$-topology.  
\end{theorem*}

\noindent The work \cite{jc2024} applied a tool from partially hyperbolic dynamics, namely Bonatti-Diaz blenders \cite{whatisblender,bonatti2004dynamics,bonattidiaz1995,bonatti2008robust}, to construct a closed hypersurface in any contact manifold with $C^2$-robustly transitive characteristic foliations. Any such hypersurface is also $C^2$-robustly non-convex, since a convex hypersurface cannot have transitive characteristic foliation. These results revealed a connection between the topology of hypersurfaces in contact manifolds and partially hyperbolic dynamics.

\vspace{3pt}

In this paper, we develop these connections further by establishing a general relationship between robust non-convexity in contact topology and robust heterodimensional dynamics (cf. Bonatti-Diaz \cite{bonatti2008robust}). As an application, we use recent work of Li-Turaev \cite{lt2024} to construct robust and local obtructions to convexity, yielding the following dramatic strengthening of Theorem \ref{thm:chaidez_nonconvex}.

\begin{theorem*}[Main Result] \label{thm:main_result} Any closed orientable hypersurface $\Sigma$ in a contact manifold $(Y,\xi)$ of dimension five or greater is isotopic to a $C^2$-robustly non-convex hypersurface via a $C^0$-small isotopy.
\end{theorem*}

\noindent This resolves a conjecture \cite[Conjecture 20]{jc2024} of the first author, and proves that the failure of Question \ref{qu:density} is extreme. Moreover, this work points to a connection between the convex/non-convex dichotomy and the program of $C^1$-generic dynamics advocated by Bonatti \cite{bonattitowards}.

\subsection{Heteroclinic Cycles And Robust Heterodimensional Dynamics} \label{subsec:heteroclinic_cycles_and_robust_heterodimensional_dynamics} We start by discussing the key concepts involved in robust heterodimensional dynamics. Recall that a \emph{singular line field} 
\[
L \quad\text{on a closed manifold}\quad M
\] is an equivalence class of smooth vector field up to rescaling by a positive smooth function. A singular line field may be viewed as a 1-dimensional foliation with singularities, or as a flow modulo reparametrization. In this paper, we will formulate many concepts using singular line fields instead of vector fields or flows, since this is the dynamical structure associated to a hypersurface in a contact manifold. Fix the following terminology.

\begin{definition*}[Heteroclinic Cycle] A \emph{heteroclinic cycle} $C$ of a singular line field $L$ is a pair of hyperbolic basic sets $C_+$ and $C_-$ whose stable and unstable sets satisfy
\[ W^u(C_-) \cap W^s(C_+) \neq \emptyset \qquad\text{and}\qquad W^u(C_+) \cap W^s(C_-) \neq \emptyset \] 
The \emph{index} $\on{ind}(C)$ is the pair of indices of the hyperbolic invariant sets $C_+$ and $C_-$. A heteroclinic cycle is \emph{simple} if both basic sets are periodic orbits and \emph{heterodimensional} if $\on{ind}(C_+) \neq \on{ind}(C_-)$.\end{definition*}

\begin{remark*} We adopt the analogous terminology for heteroclinic cycles, indices and heterodimensional cycles in the settings of diffeomorphisms, vector fields and flows.
\end{remark*}

\begin{figure}[h]
    \centering
    \includegraphics[width=.9\linewidth]{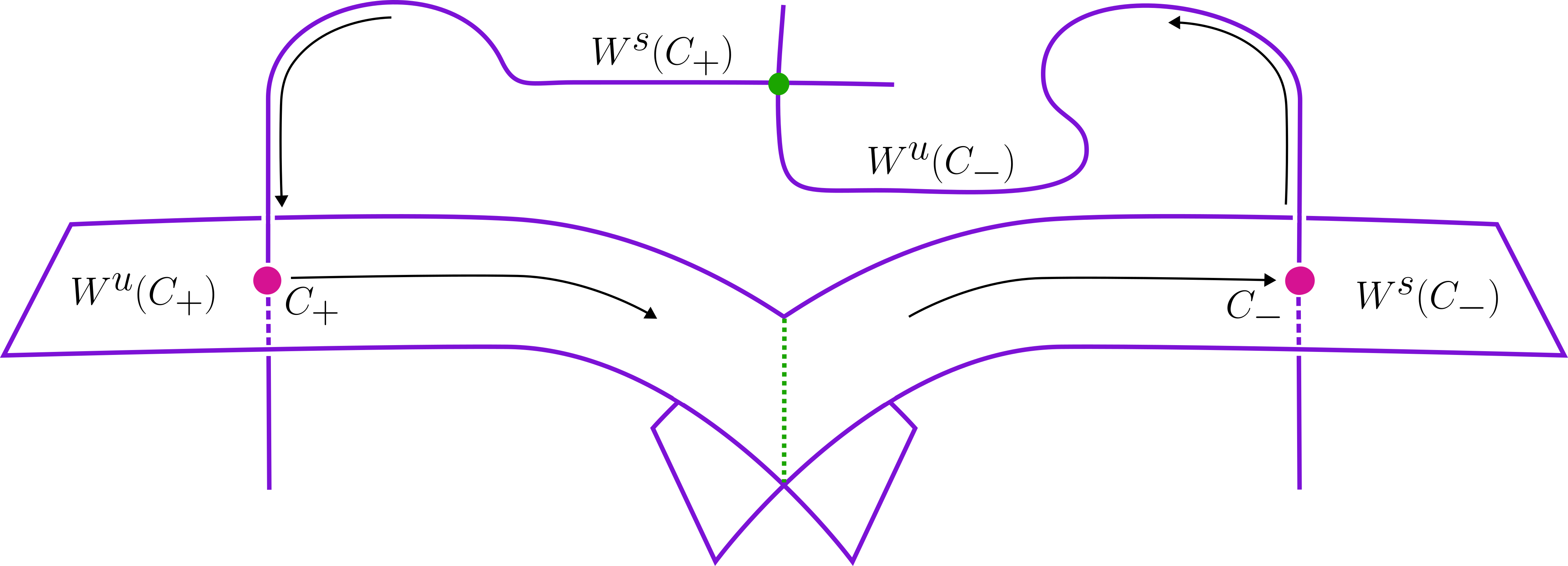}
    \caption{A simple heterodimensional cycle consisting of two hyperbolic fixed points of a 3-dimensional diffeomorphisms. }
    \label{fig:heterodimensional_cycle}
\end{figure}

\noindent Given a hyperbolic invariant set $\Lambda$ of a singular line field $L$ on $\Sigma$, any other singular line field $K$ sufficiently close to $L$ in the $C^1$-topology has a hyperbolic invariant set $\Lambda_K$ called the \emph{continuation} of  $\Lambda$. The continuation is Hausdorff close to $\Lambda$, with the same index and conjugate dynamics.

\begin{definition*}[Robustness] A heteroclinic cycle $C$ of a singular line field $L$ is \emph{robust} if it has  a $C^1$-open neighborhood $\mathcal{U}$ in the space of singular line fields such that the continuations
\[
C^+_K\quad\text{and}\quad C^-_K
\]
of the basic sets $C_+$ and $C_-$ form a heteroclinic cycle for any singular line field $K$ in $\mathcal{U}$. \end{definition*}

\begin{remark*}[Fragile And Robust Heteroclinics] \label{rmk:fragile_vs_robust} Let $C_+$ and $C_-$ be hyperbolic orbits such that $\on{ind}(C_+)$ is less than $\on{ind}(C_-)$. If the stable and unstable manifolds intersect transversely, then
\[
\on{dim}(W^s(C_+) \cap W^u(C_-)) \le 0 \qquad \on{dim}(W^u(C_-) \cap W^s(C_+)) > 0
\]
Therefore, in the transverse case $W^s(C_+) \cap W^u(C_-)$ must be empty while the intersections $W^s(C_-) \cap W^u(C_+)$ persist. Thus, the heteroclinic intersections in these two sets are called \emph{fragile} and  \emph{robust} respectively. Since a generic flow satisfies the Smale property, this analysis implies that a robust heterodimensional cycle cannot be simple.
\end{remark*}

Historically, robust heterodimensional dynamics originally appeared in the seminal work of Abraham-Smale \cite{abraham2000nongenericity} who constructed the first examples of $C^1$-open sets of diffeomorphisms that are not Axiom A. This disproved the conjecture that Axiom A diffeomorphisms and flows are $C^1$-generic, historically known as \emph{Smale's dream} \cite{smale1967differentiable}. Further examples were discovered in \cite{diaz1992nonconnected,diaz1992nonconnected,diaz1995persistence,gorodetski2000certain,mane1978contributions,shub2006topologically,simon19723} and a general approacht to robust heterodimensional cycles was eventually developed by Bonatti-Diaz \cite{bonatti2004dynamics,bonattidiaz1995} using blenders. Robust heterodimensional dynamics, along with the related phenomenon of robust homoclinic tangencies, subsequently developed into central phenomena in differentiable dynamics (cf. \cite{bonatti2008robust,hp2018,hertz2011partially} for related overviews). 

\vspace{3pt}

The centrality of robust heterodimensional dynamics is exhibited by the following famous conjecture of Palis, posed originally in \cite{palis2000global,palis2005global}, which can be viewed as a revision of Smale's dream that incorporates non-hyperbolic robust phenomena. 

\begin{conjecture*}[Palis] A $C^1$-generic diffeomorphism (or non-singular line field) either satisfies Axiom A, contains a robust heterodimensional cycle, or contains a robust homoclinic tangency.
\end{conjecture*}

\noindent In fact, there is a version of this question posed by Bonatti-Diaz \cite[Question 1.2]{bonatti2008robust} that removes the case of homoclinic tangencies, and thus conjectures a dichotomy between hyperbolicity and robust heterodimensional cycles. This variant of the Palis Conjecture is closely related to the results of this paper, and we will return to it shortly.

\subsection{Characteristic Foliations And Basic Sets} \label{subsec:intro_convexity_obstruction} We next revisit hypersurfaces in contact manifolds and formulate the key convexity obstruction. Consider a hypersurface
\[
\Sigma \qquad \text{in a contact manifold}\qquad (Y,\xi)
\]
Any such hypersurface has a natural singular line field called the \emph{characteristic foliation} that is induced by the contact structure. Examples of characteristic foliations include the suspension flow of a contactomorphism, the Liouville flow on a Liouville manifold and Hamiltonian flows with drag. We will review key facts about characteristic foliations in Section \ref{sec:contact_hamiltonian_manifolds}.

\begin{remark*}[Contact Hamiltonian Manifolds] In the body of the paper, we will formulate most of our results in terms of \emph{contact Hamiltonian manifolds} which are the natural intrinsic models for hypersurfaces in contact manifolds. This language is introduced in \cite{chaidez2026conformally} and reviewed in Section \ref{sec:contact_hamiltonian_manifolds}. For simplicity, we will formulate our results in the introduction using hypersurfaces.
\end{remark*}

 A subset $\Lambda$ in the hypersurface $\Sigma$ is called \emph{Liouville} if there is a contact form for $\xi$ whose restriction to a neighborhood of $\Lambda$ is a Liouville form. If $\Sigma$ is oriented, then a Liouville subset is \emph{positive} if the Liouville form can be chosen to induce the given orientation, and \emph{negative} if it can be chosen to induce the opposite orientation. Liouville subsets plays a key role in the theory of hypersurface in contact manifolds and has appeared, for example, in the work of Breen \cite{b2021,b2024}.

\begin{figure}[h]
    \centering
    \includegraphics[width=.9\linewidth]{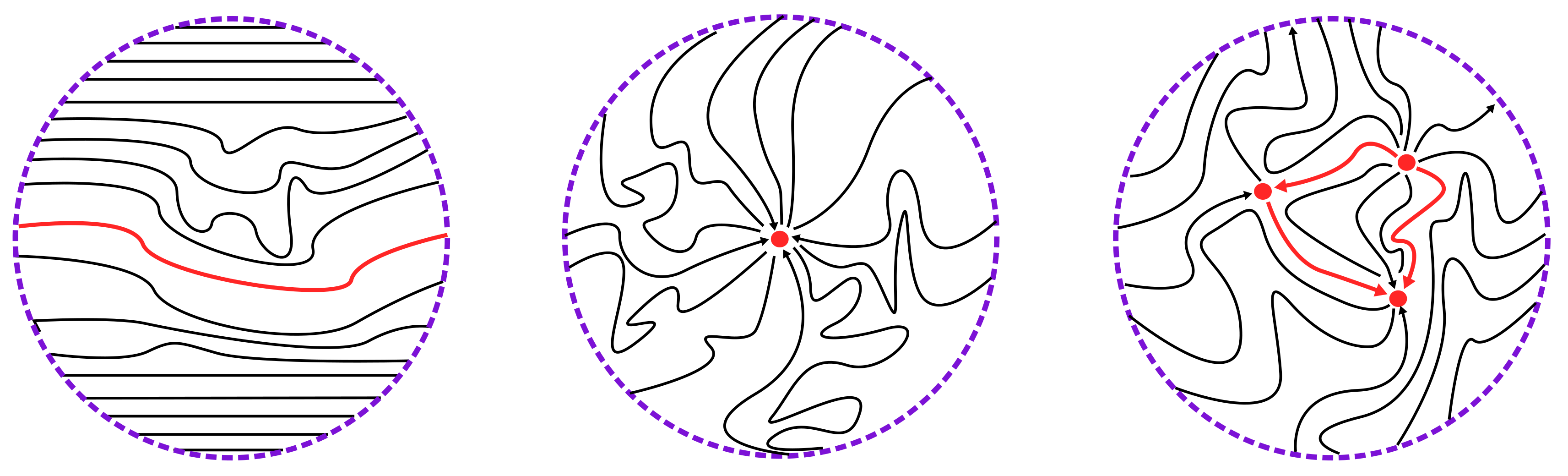}
    \caption{Examples of Liouville and non-Liouville subsets. The first subset, a trivial arc in a topologically trivial characteristic foliation, is both positive and negative Liouville. The second subset, a sink singularity, is a negative Liouville set. The final subset is a more complicated invariant set that is not Liouville.}
    \label{fig:liouville_invariant_sets}
\end{figure}
In general, a subset can be positive Liouville, negative Liouville, neither or even both. Our first main result states that hyperbolic basic sets are automatically Liouville, and the sign of the Liouville form is determined by the index of the basic set.

\begin{theorem*}[Basic Sets Are Liouville] \label{thm:intro_basic_sets_are_Liouville} Let $\Lambda \subset \Sigma$ be a hyperbolic basic set of the characteristic foliation of hypersurface $\Sigma$ in a contact $(2n+1)$-manifold. Then $\Lambda$ is Liouville with
\[
\Lambda 
\text{ positive} \qquad\text{if and only if}\qquad \text{$\on{ind}(\Lambda) \le n$}
\]\end{theorem*}

\noindent Our proof introduces a novel ergodic characterization of Liouville invariant subsets analogous to the well-known McDuff-Sullivan criterion for Reeb flows \cite{sullivan1976cycles,duff1987applications}. An interesting result of this analysis is that any Liouville invariant subset cannot be both positive and negative (Lemma \ref{lem:sign_exclusive}).

\vspace{3pt} Any convex hypersurface can be divided into a positive Liouville subset and a negative Liouville subset along a sub-manifold $\Gamma \subset \Sigma$ called the \emph{dividing set} that is transverse to the characterstic foliation. Using this splitting, one can deduce the following corollary of Theorem \ref{thm:intro_basic_sets_are_Liouville}.

\begin{corollary*}[Convexity Obstruction] \label{cor:convexity_obstruction} Let $C_+$ and $C_-$ be basic sets of the characteristic foliation of a convex hypersurface $\Sigma$ in a contact $(2n+1)$-manifold $(Y,\xi)$. Then
\[W^s(C_+) \cap W^u(C_-) = \emptyset \qquad\text{if the indices satisfy $\on{ind}(C_+) < n \le \on{ind}(C_-)$}\]
In particular, the pair $C = (C_+,C_-)$ does not form a heterodimensional cycle. \end{corollary*}

\begin{remark*}[Single Heteroclinic Versus Cycle] The reader will note that Corollary \ref{cor:convexity_obstruction} only requires one heteroclinic, instead of a full heterodimensional cycle, to obstruct convexity. However, for a generic Axiom A flow, intersections in $W^s(C_+) \cap W^u(C_-)$ can only exist when $\on{ind}(C_+) > \on{ind}(C_-)$. Therefore, the Palis Conjecture suggests that, for basic sets $C_+$ and $C_-$ with $\on{ind}(C_+) < n \le \on{ind}(C_-)$, the intersection $W^s(C_+) \cap W^u(C_-)$  can only be robustly non-empty when $C_+$ and $C_-$ form a heterodimensional cycle. In other words, it seems likely that the robust version of the obstruction in Corollary \ref{cor:convexity_obstruction} does require a heterodimensional cycle.
\end{remark*}

\begin{remark*} Theorem \ref{thm:intro_basic_sets_are_Liouville} is a generalization of a theorem of Breen \cite{b2021} stating that hyperbolic orbits are Liouville. The obstruction in Corollary \ref{cor:convexity_obstruction} is analogous to the obstruction to convexity coming from retrograde flowlines between critical points (cf. \cite{hh2019}).
\end{remark*}

\subsection{Creation Of Heterodimensional Dynamics} \label{subsec:creation_of_heterodimensional} Next we discuss the problem of realizing robust heterodimensional cycles in a given characteristic foliation. First, we have the following theorem that permits the creation of heterodimensional cycles by small perturbations.

\begin{theorem*}[Cycle Creation] \label{thm:cycle_creation} Let $\Sigma$ be a hypersurface in a contact manifold $(Y,\xi)$ of dimension $2n+1 \ge 5$ and let $U \subset Y$ be an open set intersecting $\Sigma$. Then there is an isotopy
\[
\iota_s:\Sigma \to Y \qquad\text{with}\qquad \iota_0(\Sigma) = \Sigma \qquad\text{and}\qquad \iota_s = \iota_0 \text{ outside of $U$}
\]
that is arbitrarily small in the $C^0$-topology and such that the characteristic foliation of $\Sigma_1$ contains a heteroclinic cycle $C$ with index
\[
\on{ind}(C_+) = n - 1 \qquad\text{and}\qquad \on{ind}(C_-) = n
\]\end{theorem*}

\noindent Theorem \ref{thm:cycle_creation} is proven by a explicit construction of a heterodimensional cycle in a standard region where the hypersurface resembles a certain standard local model, specifically the symplectization of a small contact disk. Precisely, one first perturbs the characteristic foliation to create a small closed orbit contained in a region modelled on the standard Liouville ball, then one modifies the monodromy of a Poincare section along this orbit to introduce a heterodimensional cycle between two closed hyperbolic orbits close to the original small orbit.

\vspace{3pt}

 Second, we show that any heterodimensional cycle of the correct index can be used to create a robust heterodimensional cycle. This can be viewed as a stabilization theorem \cite{lt2024} for cycles in the setting of characteristic foliations.

\begin{theorem*}[Fragile To Robust] \label{thm:robustification} Let $\Sigma$ be a closed hypersurface in a contact manifold $(Y,\xi)$ whose characteristic foliation contains a simple heteroclinic cycle $C$ of index $(n-1,n)$. Then there is an isotopy
\[
\iota_s:\Sigma \to Y \qquad\text{with}\qquad \iota_0(\Sigma) = \Sigma
\]
that is arbitrarily small in the $C^\infty$-topology and such that the characteristic foliation of $\Sigma_1$ contains a robust heterodimensional cycle of index $(n-1,n)$.\end{theorem*}

\noindent Note that the combination of Corollary \ref{cor:convexity_obstruction}, Theorem \ref{thm:cycle_creation} and Theorem \ref{thm:robustification} immediately imply the main result (Theorem \ref{thm:main_result}) stated at the beginning of the introduction. 

\vspace{3pt}

This is an application of the deep and very recent work of Li-Turaev \cite{lt2024} on robust heterodimensional dynamics. Precisely, we show (in Theorem \ref{thm:characteristic_unfoldings}) that, for any hypersurface $\Sigma$ whose characteristic foliation contains a heterodimensional cycle of index $(n-1,n)$, we can construct a family of embedded hypersurfaces
\[
\iota_\epsilon:\Sigma \to Y \qquad\text{parametrized by $\epsilon \in \R^2$ with $\iota_0(\Sigma) = \Sigma$}
\] such that the characteristic foliations of the surfaces $\iota_\epsilon(\Sigma)$ form a \emph{proper unfolding} of singular line fields on $\Sigma$, in the sense of \cite{lt2024}. Roughly speaking, a proper unfolding $\Phi_\epsilon$ in the sense of \cite{lt2024} is a family of flows (in the continuous time case) or diffeomorphisms (in the discrete time case) parametrized by $\epsilon \in \R^k$ for $k  \ge 2$. The systems $\Phi_\epsilon$ must have a continuously varying pair of hyperbolic orbits $C^+_\epsilon$ and $C^-_\epsilon$ that form a heterodimensional cycle $C$ when the parameters are zero. Moreover, as one of the parameters $\epsilon_1$ varies, the fragile heteroclinic intersection disappears in a generic manner. Similarly, as the other parameter $\epsilon_2$ varies, a certain distinguished multiplier of $C^-_\epsilon$ called the \emph{center-stable} multiplier must vary non-trivially (see Figure \ref{fig:unfolding}). We will give a very detailed review of the theory of unfoldings in Section \ref{sec:non_degenerate_cycles_and_unfoldings}.  

\begin{figure}[h]
    \centering
    \includegraphics[width=.9\linewidth]{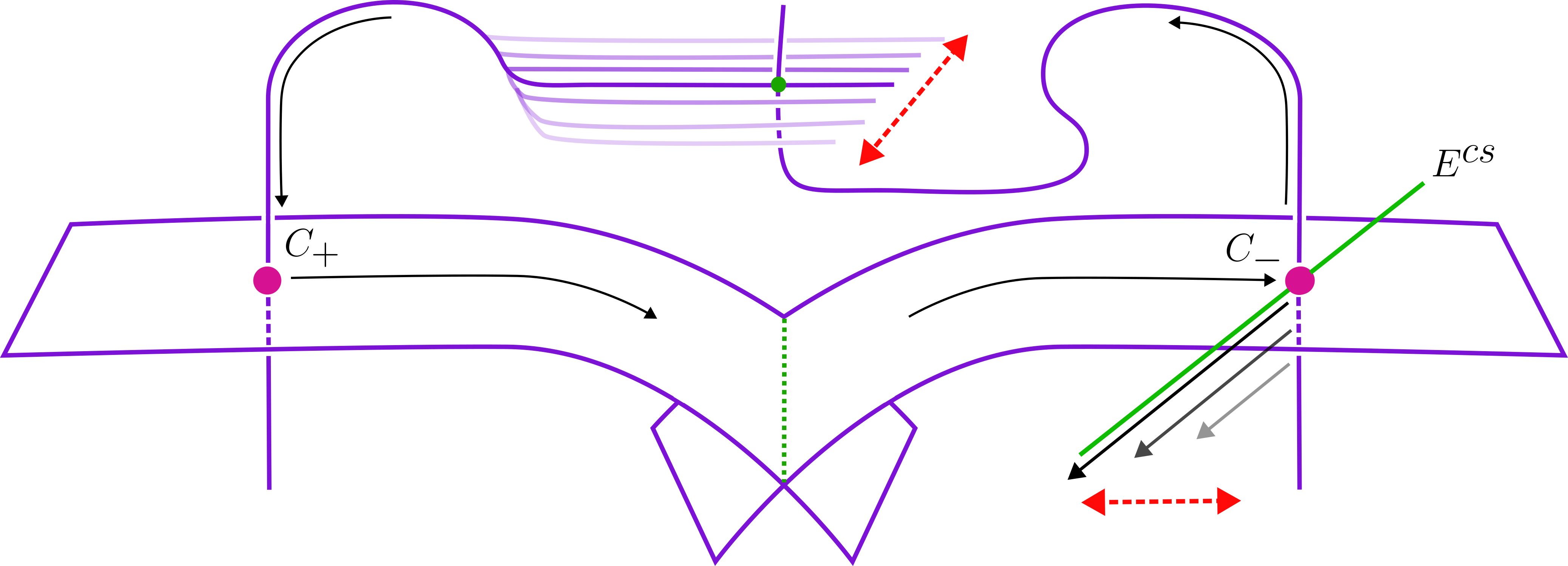}
    \caption{A proper unfolding based at the heterodimensional cycle in Figure \ref{fig:heterodimensional_cycle}. Varying one of the parameters in the family causes the fragile heteroclinic to disappear transversely, and varying another parameter causes the center-stable multiplier to vary, as depicted by the dashed arrows above.}
    \label{fig:unfolding}
\end{figure}

A main result of Li-Turaev \cite[Thm B]{lt2024} states that, for any proper unfolding such that $\Phi_0$ has a heterodimensional cycle $C$, the system $\Phi_\epsilon$ exhibits a robust heterodimensional cycle $C'$ of the same index as $C$ for a generic small parameter $\epsilon$. Thus, a generic hypersurfaces $\iota_\epsilon(\Sigma)$ in the family constructed in Theorem \ref{thm:characteristic_unfoldings} will contain a robust heterodimensional cycle of the desired index for any small generic parameter $\epsilon$.

\begin{remark*}[Blenders] Blenders \cite{bonattidiaz1995,whatisblender} are the key mechanism producing robust heterodimensional cycles in the results of Li-Turaev (cf. \cite[Thm D]{lt2024}). They are also the key ingredient in the construction of robustly non-convex hypersurfaces in \cite{jc2024}, where the first author constructed the first examples of blenders in contact dynamics. In fact, the hypersurfaces constructed in \cite{jc2024} do contain heterodimensional cycles induced by the presence of a blender. Thus, this paper identifies more precisely the mechanism causing robust non-convexity in \cite{jc2024}.
\end{remark*}

\subsection{Generic Dynamics For Characteristic Foliations} We conclude this introduction with a perspective on our results related to the program on $C^1$-generic dynamics outlined by Bonatti. 

\vspace{3pt}

In \cite{bonattitowards}, Bonatti articulates a general philosophy for understanding $C^1$-generic diffeomorphisms, stating several conjectures similar to the Palis Conjecture. Roughly, each conjecture in \cite{bonattitowards} posits that the space of diffeomorphisms divides into two $C^1$-open sets consisting of (a) a set of diffeomorphisms satisfying a global structural property and (b) a set of diffeomorphisms exhibiting a robust local structure that obstructs the global one. The specific global structural property and robust local structure vary from conjecture to conjecture.

\vspace{3pt}

The results of this paper fit neatly into this program as follows. We consider space of closed oriented hypersurfaces in contact manifolds equipped with the $C^2$-topology, which corresponds to the $C^1$-topology on the characteristic foliations. On one hand, convexity of the hypersurface provides a global structural property of the characteristic dynamics. On the other hand, the results of this paper suggest that the complementary robust local structure should be a certain type of heterodimensional cycle. More precisely, we fix the following terminology.

\begin{definition*}[Positive-Negative] A heterodimensional cycle $C$ in the characteristic foliation of a hypersurface $\Sigma$ in a contact $2n+1$-manifold $(Y,\xi)$ is \emph{positive-negative} if the indices satisfy
\[
\on{ind}(C_+) \le n-1 \qquad\text{and}\qquad \on{ind}(C_-) \ge n
\]
\end{definition*}

We can divide the space of closed oriented hypersurfaces in a contact manifold into two $C^2$-open sets: the set of convex hypersurfaces and the set of hypersurfaces containing a robust positive-negative heterodimensional cycle. Corollary \ref{cor:convexity_obstruction} states that these sets are disjoint; Theorems \ref{thm:cycle_creation} and \ref{thm:robustification} state that the latter set is non-empty and $C^0$-dense in higher dimensions; and the result of Honda-Huang \cite{hh2019} states that the former set is non-empty and $C^0$-dense. We thus arrive at the following Palis-type conjecture in the spirit of Bonatti's program.

\begin{conjecture*} \label{con:convex_palis} The space of closed oriented hypersurfaces in a contact manifold $(Y,\xi)$ decomposes into two open and $C^2$-dense sets: the set of convex hypersurfaces and the set of hypersurfaces containing a robust positive-negative heterodimensional cycle $C$.
\end{conjecture*}

\noindent Phrased differently, Conjecture \ref{con:convex_palis} states that positive-negative cycles are generically the only obstruction to convexity.

\vspace{3pt}

In this paper, all of the positive-negative heterodimensional cycles $C$ involved in the convexity obstructions are coindex one. While they are not very well understood, robust heterodimensional cycles of higher codimension are known to exist \cite{barrientos2014symbolic,barrientos2025robust,asaoka2022stable}. It is interesting to ask if there are examples of robust positive-negative cycles of a different type. 

\begin{question*} Are there examples of robust, positive-negative heterodimensional cycles of coindex greater than one in the characteristic foliations of hypersurfaces in contact manifolds?
\end{question*}

\subsection*{Outline} This concludes Section \ref{sec:introduction}, the introduction. In Section \ref{sec:contact_hamiltonian_manifolds}, we provide  background on contact Hamiltonian manifolds (or equivalently, hypersurfaces in contact manifolds). In Sections \ref{sec:liouville_and_hyperbolic_invariant_sets} and \ref{sec:construction_of_heteroclinic_cycle}, we prove Theorems \ref{thm:intro_basic_sets_are_Liouville} and \ref{thm:cycle_creation} respectively. In Section \ref{sec:non_degenerate_cycles_and_unfoldings}, we provide key background on non-degeneracy and unfoldings from \cite{lt2024}. Finally, in Section \ref{sec:contact_unfoldings_and_robustifications} we prove Theorem \ref{thm:robustification}. 

\subsection*{Acknowledgements} AI tools were used solely for proofreading this manuscript. JC and MH were partially supported by NSF Award DMS-2446019 and US-Israel BSF Award 2024157. 

\newpage

\section{Contact Hamiltonian Manifolds And Convexity} \label{sec:contact_hamiltonian_manifolds}

In this section, we discuss the theory of contact Hamiltonian manifolds and their characteristic foliations.  These spaces are the natural intrinsic models for hypersurfaces in contact manifolds. This terminology is introduced at this level of generality in the survey \cite{chaidez2026conformally}.

\subsection{Contact Hamiltonian Manifolds} We start by introducing contact Hamiltonian manifolds and the basic structures associated to them.

\begin{definition}[Singular Hyperplane Field] A \emph{singular hyperplane field} $\eta$ on a manifold $\Sigma$ is an equivalence class of 1-forms up to multiplication by a nowhere zero smooth function. We let
\[
\on{ker}(\lambda) \qquad \text{denote the equivalence class of a 1-form $\lambda$}
\]
A \emph{singularity} of a singular hyperplane field $\eta$ is a point where the defining 1-form $\lambda$ vanishes. \end{definition}
 
\begin{definition}[Contact Hamiltonian Structure] \label{def:contact_Hamiltonian_manifold} A \emph{contact Hamiltonian structure} $\eta$ on a $2n$-manifold $\Sigma$ is a singular hyperplane field $\eta$ defined by 1-form $\lambda$ satisfying
\begin{equation} \label{eq:contact_Hamiltonian_form}
d\lambda^n \neq 0 \text{ at any point where $\lambda \wedge d\lambda^{n-1}$ vanishes.}
\end{equation}
A 1-form $\lambda$ satisfying this formula is called a \emph{contact Hamiltonian form}.\end{definition}

Any contact Hamiltonian manifold has a naturally associated dynamical system, well-defined up to reparametrization. More precisely, we have the following definition.

\begin{definition}[Characteristic Flow] The \emph{characteristic vector field} $Z$ determined by a contact Hamiltonian form $\lambda$ and a volume form $\mu$ is the unique vector field satisfying
\begin{equation} \label{eq:characteristic_vectorfield}
\iota_Z\mu = \lambda \wedge d\lambda^{n-1}
\end{equation}
A flow acquired by integrating a characteristic vector field is called a \emph{characteristic flow}. \end{definition}

\noindent The characteristic vector field depends on the choice of contact Hamiltonian form and volume form, but it determines a well-defined (singular) line field, in the following sense.

\begin{definition}[Singular Line Field] A \emph{singular line field} $L$ on a manifold $\Sigma$ is an equivalence class of vector fields up to multiplication by a nowhere zero smooth function. We let
\[
\on{span}(V) \qquad \text{denote the equivalence class of a vector field $V$}
\]
A \emph{singularity} of a singuar line field $L$ is a point where the defining vector field vanishes.\end{definition}

\begin{definition}[Characteristic Foliation] The \emph{characteristic foliation} $\Sigma_\eta$ of a contact Hamiltonian manifold $(\Sigma,\eta)$ is the unique singular line field spanned by any characteristic vector field. \end{definition}

\begin{remark} It is straightforward to check that the characteristic line field is well-defined (i.e. independent of the volume form and contact Hamiltonian form) \cite[Lem 2.17]{chaidez2026conformally}. \end{remark}

\begin{remark} \label{lem:rmk_alt_char_field} The characteristic vector field of $\eta$ can also be characterized as the unique vector field tangent to $\eta$ and that preserves $\eta$ in the sense that $\mathcal{L}_Z\lambda = f\lambda$ for any contact Hamiltonian form $\lambda$ and some smooth function $f$ that is non-zero along the singularities (cf. \cite[Lem 2.20]{chaidez2026conformally}).
\end{remark}

There is an important class of subsets of contact Hamiltonian manifold, where the contact Hamiltonian manifold structure is given by the kernel of a Liouville form. 

\begin{definition}[Liouville Subsets] A subset $\Lambda \subset \Sigma$ in a contact Hamiltonian manifold $(\Sigma,\eta)$ is called \emph{Liouville} if there is a contact Hamiltonian form $\lambda$ such that
\[
\text{$\lambda|_U$ is a Liouville form on a neighborhood $U$ of $\Lambda$}
\]
Given an orientation on $\Sigma$, a Liouville subset $\Lambda \subset \Sigma$ is called \emph{positive} if $\lambda|_U$ induces the orientation of $\Sigma$ and \emph{negative} if $\lambda|_U$ induces the opposite orientation of $\Sigma$. \end{definition}

There are a handful of examples of contact Hamiltonian manifolds that we will consider: hypersurfaces in contact manifolds, Liouville manifolds, symplectizations and suspensions. 

\begin{example}[Hypersurfaces] \label{ex:hypersurfaces} Any hypersurface $\Sigma$ in a contact manifold $(Y,\xi)$ with contact form $\alpha$ is a contact Hamiltonian manifold with contact Hamiltonian structure
\[
\eta = \xi \cap T\Sigma \qquad\text{defined by the contact Hamiltonian form}\qquad \lambda = \alpha|_\Sigma
\]
The characteristic foliation $\Sigma_\eta$ is simply the standard characteristic foliation $\Sigma_\xi$ with respect to $\xi$. \end{example}

\begin{example}[Exact Symplectic Manifolds] Any exact symplectic manifold $\Sigma$ with Liouville form $\lambda$ is a contact Hamiltonian manifold, with contact Hamiltonian structure
\[
\eta = \on{ker}(\lambda) \qquad\text{with contact Hamiltonian form $\lambda$}
\]
The characteristic foliation $\Sigma_\eta$ is simply the span of the Liouville vector field $Z$ in this case.\end{example}

\begin{example}[Symplectization] \label{ex:symplectization} The symplectization $\R_t \times Y$ of a contact manifold $(Y,\xi)$ is naturally a contact Hamiltonian manifold with contact Hamiltonian structure
\[
\eta = \on{span}(\partial_t) \oplus \xi
\]
In this case, the characteristic foliation is simply the span of the vector field $\partial_t$. \end{example}

\subsection{Contactization} \label{subsec:contactization} Every contact Hamiltonian manifold admits a contact structure on the product of the manifold with an interval, called the contactization. This requires a choice of framing.

\begin{definition}[Framing] \label{def:framing} A \emph{framing} $(u,\theta)$ for a contact Hamiltonian form $\lambda$ on a contact Hamiltonian manifold $(\Sigma,\eta)$ is a pair of a smooth function $u$ and a smooth 1-form $\theta$ on $\Sigma$ where
\begin{equation} \label{eq:framing_volume_form}
u \cdot d\lambda^n + n\theta \wedge \lambda \wedge d\lambda^{n-1} \quad\text{is a volume form on $\Sigma$}
\end{equation}\end{definition}

\begin{definition}[Contactization] The \emph{contactization} $C\Sigma$ of a contact Hamiltonian manifold $(\Sigma,\eta)$ is the product $[-\epsilon,\epsilon]_s \times \Sigma$ equipped with the contact structure $\xi$ with the contact form
\[
\alpha = uds + s(\theta + du) + \lambda \qquad\text{for any contact Hamiltonian form $\lambda$ and framing $(u,\theta)$}
\]\end{definition}

The contactization is indeed a contact manifold containing $\Sigma = 0 \times \Sigma$ as a hypersurface \cite[Lem 2.31]{chaidez2026conformally}. Moreover, it is independent of the contact Hamiltonian form and framing as a germ of a contact structure near $0 \times \Sigma$ due to the following lemma.

\begin{lemma} \cite{g2008intro} \label{lem:stability_for_hypersurfaces} Let $\xi$ and $\xi'$ be contact structures on $Y$ and let $\Sigma \subset Y$ be a closed hypersurface with $\xi \cap T\Sigma = \xi' \cap T\Sigma$. Then there is an isotopy $\Phi$ of $Y$ such that $\Phi_0 = \on{Id}$, $\Phi_1^*\xi = \xi'$ near $\Sigma$ and $\Phi_t(\Sigma) = \Sigma$.\end{lemma}

\begin{corollary}[Contactization Neighborhood] \label{cor:contactization_nbhds} Let $\Sigma \subset Y$ be a closed hypersurface in a contact manifold $(Y,\xi)$ with contact Hamtiltonian structure $\eta = \xi \cap T\Sigma$. Then there is a contact embedding
\[
\iota:C\Sigma \to Y \qquad\text{with}\qquad \iota(\Sigma) = \Sigma
\]
\end{corollary}

\noindent The set of framings is convex and also non-empty in the following strong sense.

\begin{lemma}[Framing Extension] Let $(\Sigma,\eta)$ be a contact Hamiltonian manifold. Fix an open subset $U$ and a 1-form $\phi$ on $\Sigma$ such that
\[\phi(Z) > 0\qquad \text{on $U$ for some characteristic vector field $Z$}\]
Then there is a framing $(u,\theta)$ with $u|_U = 0$ and $\theta|_U = \phi$.
\end{lemma}

\begin{proof} Choose a volume form $\mu$ and a contact Hamiltonian form $\lambda$ with characteristic vector field $Z$. There is a 1-form $\theta$ such that $\theta(Z) \ge 0$, such that $\theta|_U = \phi$ and such that $d\lambda^n$ is a volume form on the set $S$ where $\theta(Z) = 0$. Note that $S \cap U = \emptyset$ and $S = S_+ \sqcup S_-$ where $S_+$ is positive Liouville and $S_-$ is negative Liouville. There is then a smooth function $u$ such that
\[u|_U = 0 \qquad\text{and}\qquad u d\lambda^n = f\mu\text{ for a non-negative smooth function $f$ with $f|_S > 0$}\] 
It is simple to check that $(u,\theta)$ is framing that satisfies the desired properties.\end{proof}

\subsection{Convexity} A contact Hamiltonian manifold is convex if the contactization admits a translation invariant contact structure. More precisely, we adopt the following definition.

\begin{definition}[Convex Structure] \label{def:convex_structure} A contact Hamiltonian manifold $(\Sigma,\eta)$ is \emph{convex} if there is a smooth function $u$ such that $(u,-du)$ is a framing, or equivalently
\[
u ds + \lambda \qquad\text{is a contact form on $C\Sigma = [-\epsilon,\epsilon] \times \Sigma$}
\]
We refer to the function $u$ as a \emph{framing function} for the (convex) contact Hamiltonian manifold.\end{definition}

\noindent Note that Definition \ref{def:convex_structure} is simply the intrinsic formulation of convexity for hypersurfaces in contact manifolds, in the sense of Giroux \cite{g1991}.

\begin{definition}[Convex Hypersurface] \label{def:convex_hypersurface} A hypersurface $\Sigma$ in a contact manifold $(Y,\xi)$ is \emph{convex} if there exists a contact vector field $V$ that is transverse to $\Sigma$.
\end{definition}

The key property of convex hypersurfaces is the existence of a natural splitting into a positive Liouville region and a negative Liouville regions along a contact sub-manifold. 

\begin{definition}[Dividing Set] The \emph{dividing set} $\Gamma$ of a convex contact Hamiltonian manifold $(\Sigma,\eta)$ with respect to a framing function $u$ is the hypersurface
\[
\Gamma = \{u = 0\} \subset \Sigma
\]
The regions $\Sigma_+$ and $\Sigma_-$ where $u$ is non-negative and non-positive are respectively called the \emph{positive} and \emph{negative} regions of $\Sigma$. There is a natural splitting
\[
\Sigma = \Sigma_+ \cup \Sigma_- \qquad\text{with}\qquad \Sigma_+ \cap \Sigma_- = \Gamma
\]\end{definition}

\begin{lemma}[Dividing Set Is Contact] The dividing set $\Gamma \subset \Sigma$ is a contact sub-manifold transverse to the characteristic foliation that is independent of the framing function up to isotopy.
\end{lemma}

\begin{lemma}[Halves Are Liouville] \label{lem:liouville_halves} The halves $\Sigma_+$ and $\Sigma_-$ of a convex contact Hamiltonian manifold $(\Sigma,\eta)$ are positive and negative Liouville respectively, and $\Sigma_\eta$ points out from $\Sigma_+$ along $\Gamma$.
\end{lemma} 

\noindent For proofs of the above facts, we refer the reader to \cite[Lem 3.7 and 3.8]{chaidez2026conformally} or \cite{b2024}.

\subsection{Deformations} \label{subsec:deformations} There is a natural type of deformation for contact Hamiltonian structures given by embeddings into the contactization..

\begin{definition}[Ambient Deformation] An \emph{ambient deformation} $\iota$ of a contact Hamiltonian structure $\eta$ on $\Sigma$ is a 1-parameter family of embeddings into the contactization
\[
\iota:[0,1] \times \Sigma \to C\Sigma \qquad\text{such that $\iota_0$ is inclusion of $\Sigma = 0 \times \Sigma$}
\]
There is a naturally associated family of contact Hamiltonian structures given by $\iota_t^*\xi$ where $\xi$ is the contact structure on the contactization. We denote the endpoing of this family by
\[
\eta_\iota = \iota^*_1\xi
\]\end{definition}

\begin{definition}[Size] The \emph{$C^k$-size} of an ambient deformation with respect to a Riemannian metric $g$ is defined to be the $C^k$-distance
\[
\|\iota\|_{C^k} = \on{dist}_{C^k}(\iota,\iota_{\on{triv}})
\]
where $\iota_{\on{triv}}$ is the trivial isotopy given by the inclusion $\Sigma \to C\Sigma$ for all parameters. \end{definition}

\begin{remark}[Concatenation] \label{rmk:concatenation} If $\iota$ is an ambient deformation of $(\Sigma,\eta)$ and $\jmath$ is an ambient deformation of $\eta_\iota$ that is sufficiently $C^0$-small, then we can form a concatenation
\[\jmath \star \iota:[0,1] \times \Sigma \to C\Sigma\]
This is the concatenation of the isotopy $\iota$ and the isotopy $\phi \circ \jmath$, where $\phi:U \to C\Sigma$ is an embedding of the a neighborhood $U$ of $\Sigma$ in the contactization of $\Sigma$ with respect to $\eta_\iota$ into the contactization $C\Sigma$ with respect to $\eta$. Such an embedding exists by Corollary \ref{cor:contactization_nbhds}. Note that for any choice of metric $g$ and any integer $k$, there is a constant $A$ depending on $g,k$ and $\phi$ such that
\[
\|\jmath \star \iota\|_{C^k} \le A\|\jmath\|_{C^k} + \|\iota\|_{C^k} 
\]\end{remark}

There are a few key examples of ambient deformations that will appear frequently. In particular, one can ambiently deform by taking the graph of a Hamiltonian within the contactization.

\begin{example}[Graphical Deformation] Fix a sufficiently small smooth function $H$, a contact Hamiltonian form $\lambda$ and a framing $(\theta,u)$  on $\Sigma$. There is an associated \emph{graphical deformation}
\begin{equation}
\iota_H:[0,1]_r \times \Sigma \to [-\epsilon,\epsilon]_s \times \Sigma = C\Sigma \qquad\text{where}\qquad \iota_H(r,x) = (-rH(x),x)
\end{equation}
Here we use the explicit form of the contactization $C\Sigma$ given by
\[
C\Sigma = [-\epsilon,\epsilon] \times \Sigma \quad\text{with contact form}\quad uds + s(du + \theta) + \lambda
\]
The endpoing $\eta_{\theta,H}$ of the corresponding family of contact Hamiltonian structures is given by
\begin{equation}
\eta_{\theta,H} = \on{ker}(\lambda_{\theta,H}) \qquad\text{where}\qquad \lambda_{\theta,H} = \iota_1^*\alpha = -u \cdot dH - H(\theta + du) + \lambda
\end{equation}
\end{example}

\begin{remark}[Sympelectization Case] Given a contact manifold $U$, the symplectization $\R_t \times U$ has a natural framing $(0,dt)$. For this particular framing, we adopt the simplified notation
\[
\lambda_H = -H dt + \lambda \qquad\text{and}\qquad \eta_H = \on{ker}(\lambda_H)
\]
\end{remark}

\begin{remark}[Size] \label{rmk:size} For any contact Hamiltonian manifold $\Sigma$, any choice of Riemannian metric $g$ on the contactization and any integer $k$, there is a constant $A$ depending on $g$ and $k$ with
\[
\|\iota_H\|_{C^k} \le A \cdot \|H\|_{C^k}
\]
\end{remark}

\begin{lemma}[Hamiltonian Flow] \label{lem:Ham_flow} Fix a contact Hamiltonian $H:[0,1]_t \times U \to \R$ on a contact manifold $U$ with contact form $\alpha$. Then the deformed structure $\eta_H = -Hdt + \alpha$ has a characteristic vector field
\[
Z_H = \partial_t + V_H \qquad\text{where $V_H$ is the contact Hamiltonian vector field generated by $H$} 
\]\end{lemma}

\begin{proof} Recall that $V = V_H$ is uniquely defined by the property that $\alpha(V) = H$ and $\mathcal{L}_V\alpha = f\alpha$ for a smooth function $f$ on $(0,1)_t \times U$. Thus $Z$ is tangent to the contact Hamiltonian structure since $\lambda_H(V) = -H + \alpha(V) = 0$. Moreover, we can compute the Lie derivative of $\lambda_H$ along $Z$ to see that
\[
\mathcal{L}_Z\lambda_H = -\mathcal{L}_VH \cdot dt + \mathcal{L}_V \alpha = -\iota_V(\mathcal{L}_V\alpha) \cdot dt + \mathcal{L}_V \alpha = f(-H dt + \alpha) = f\lambda_H
\]
Thus the vector field $Z$ is a characteristic vector field by Remark \ref{lem:rmk_alt_char_field}. \end{proof}

\subsection{Plugs} \label{subsec:plugs} We next introduce the key notion of a plug. Let $D_{\on{std}}$ be the contact disk of dimension $2n-1$ with standard contact form $\alpha_{\on{std}}$. Consider the contact Hamiltonian manifold
\[
U_{\on{std}} = [0,1]_t \times D_{\on{std}} \qquad\text{with contact Hamiltonian structure}\qquad \eta_{\on{std}} = \on{ker}(\alpha_{\on{std}})
\]
There is a standard characteristic vector field $Z_{\on{std}} = \partial_t$. We adopt the following notation for the regions of the boundary where this vector field points inward and outward.
\[
D_{\on{in}} = 0 \times D_{\on{std}} \subset U_{\on{std}} \qquad\text{and}\qquad D_{\on{out}} = 0 \times D_{\on{std}} \subset U_{\on{std}} 
\]
The contactization of this contact Hamiltonian manifold can be written as 
\[
\R_s \times U_{\on{std}} = T^*[0,1] \times D_{\on{std}} \qquad \text{with contact structure}\qquad \xi_{\on{std}} = \on{ker}(sdt + \alpha_{\on{std}})
\]
A plug is essentially an ambient deformation of the contact Hamiltonian structure on $U_{\on{std}}$ that is supported away from the boundary. The precise definition is the following.

\begin{definition}[Plug] \label{def:contact_Hamiltonian_plug} A \emph{contact Hamiltonian plug} is an isotopy of embeddings
\[\iota:[0,1]_r \times U_{\on{std}} \to CU_{\on{std}} \qquad\text{such that $\iota_0$ is the inclusion $U_{\on{std}} \subseteq CU_{\on{std}}$}\]
with the property that $\iota_r$ agrees with the inclusion $U_{\on{std}} \subset CU_{\on{std}}$ near the boundary of $U_{\on{std}}$.  \end{definition}

\begin{example}[Graphical Plug] Any contact Hamiltonian $H:[0,1] \times D_{\on{std}} \to \R$ on the standard disk $D_{\on{std}}$ supported away from the boundary of $[0,1] \times D_{\on{std}}$ determines a \emph{graphical plug} $\iota_H$ with
\[
\iota_H(r,t,x) = (rH(t,x),t,x)
\]
\end{example}

\begin{definition}[Plug Insertion] \label{def:plug_insection} A \emph{plugging domain} $U$ in a contact Hamiltonian manifold $(\Sigma,\eta)$ is a domain $U \subset \Sigma$ equipped with an isomorphism of contact Hamiltonian manifolds
\[
(U_{\on{std}},\eta_{\on{std}}) \simeq (U,\eta|_U)
\]
The \emph{insertion} $\iota_U$ of a contact Hamiltonian plug $\iota$ along a plugging domain $U$ is the ambient deformation of $\Sigma$ that agrees with $\iota$ on $U \simeq U_{\on{std}}$ and that is the inclusion $\Sigma \to \Sigma$ outside of $U$. \end{definition}

\begin{remark} Note that the insertion of a contact Hamiltonian plug is only defined for a plug $\iota$ that is sufficiently small in the $C^0$-topology. 
\end{remark}

\begin{remark} \label{rmk:insertion_along_domain_of_H} We denote the insertion of a graphical plug $\iota_H$ along a plugging domain $U$ in $\Sigma$ by
\[
\iota_{U,H}:[0,1] \times \Sigma \to C\Sigma
\]
We denote the corresponding deformed contact Hamitonian structure by
\[\eta_{U,H}\]
Note that for any choice of metric $g$ on $\Sigma$ and any integer $k$, there is a constant $A$ depending on the plugging domain $U$, $g$ and $k$ such that
\[
\|\iota_{U,H}\|_{C^k} \le A \cdot \|H\|_{C^k}
\]\end{remark}

%\[(P,\eta_{\on{std}}) \qquad\text{given by $P = [0,1]_t \times D$ for the standard contact disk $(D,\xi_{\on{std}})$}\]
%Let $\lambda_{\on{std}}$ denote the standard contact Hamiltonian form on $P = [0,1] \times D$ coming from the standard contact form on $D$ and consider the contactization
%\[CP = T^*[0,1] \times D = \R_s \times P \qquad\text{with contact structure }\xi= \on{ker}(s dt + \lambda_{\on{std}})\]

\subsection{Transition And Return Maps} \label{subsec:transition_and_return_maps} We next review the notion of transition maps between local sections of a singular line field, and specifically the characteristic foliation.

\begin{definition}[Transition Map] \label{def:transition_map} Let $D$ and $D'$ be local sections of a singular line field $L$ on $M$ and let $\Gamma$ be a trajectory of $L$ connecting a point $P$ in $D$ to a point $P'$ in $D'$. The \emph{transition map}
\[
\on{Tr}\Gamma:D \to D' \qquad\text{along the trajectory $\Gamma$}
\]
is the germ at $P$ defined as follows. Fix a flow $\Phi$ generating $L$ and let $T$ be the time such that the map $[0,T] \to M$ given by $t \mapsto \Phi_t(P)$ parametrizes $\Gamma$. Then $\on{Tr}\Gamma = \Phi \circ \sigma$ where $\sigma$ is a smooth map
\[\sigma:U \to \R \times M \qquad\text{where $U \subset D$ is a neighborhood of $P$, $\Phi \circ \sigma(D) \subset D'$ and $\sigma(P) = (T,P)$}\]
The transition map is \emph{well-defined} on $U \subseteq D$ if a smooth map $\sigma$ satisfying these properties exists. \end{definition}

We remark on several standard composition and continuity properties of transition maps.

\begin{remark}[Definedness] \label{rmk:definedness} The transition map $\on{Tr}\Gamma:D \to D'$ along $\Gamma$ is unique on any open subset $U \subseteq M$ where it is well-defined, and may be viewed as a germ near the start point $P$ in $D$.  
\end{remark}

\begin{remark}[Composition] \label{rmk:composition}The transition map $\on{Tr}\Gamma$ associated to the concatenation of a segment $\Gamma_1$ connecting $D_0$ to $D_1$ and a segment $\Gamma_2$ connecting $D_1$ to $D_2$ satisfies
\[
\on{Tr}\Gamma = \on{Tr}\Gamma_2 \circ \on{Tr}\Gamma_1 
\]
on any open set $U \subseteq D$ where the transition map of $\Gamma$ is well-defined. \end{remark}

\begin{remark}[Continuity] \label{rmk:continuity} For any local sections $D,D'$ and segment $\Gamma$ from $P$ in $D$ and $P'$ in $D'$, and any singular line field $K$ sufficiently $C^1$-close to $L$, there is an associated segment $\Gamma_K$ starting at $P$ in $D$ and ending at a point in $D'$ and the transition map
\[
\on{Tr}\Gamma_K:D \to D'
\]
is well-defined on an open set $U \subset D$ independent of $K$ and continuous in $K$ in the $C^1$-topology. \end{remark}

\begin{definition}[Poincare Return Map] \label{def:poincare_return} Fix a closed orbit $\Gamma$ of $L$ and a local section $D$ intersecting $\Gamma$ at a point $P$.  The \emph{Poincare return map}
\[\on{Ret}\Gamma: D \to D\]
is the transition map associated to $\Gamma$ as a flowline segment connecting $P$ to $P$. \end{definition}

Next we turn to the transition maps and Poincare return maps associated to characteristic foliations. Fix a contact Hamiltonian manifold
\[\Sigma \qquad\text{with contact Hamiltonian structure}\qquad \eta\]
We note that the local sections and transition maps for characteristic foliations are naturally contact manifolds and contact maps respectively \cite[Lem 2.21 and Rmk 2.23]{chaidez2026conformally}. 

\begin{lemma}[Contact Section] Any local section $D$ of the characteristic foliation $\Sigma_\eta$ is a contact manifold with contact structure $\xi = \eta \cap TD$.
\end{lemma}

\begin{lemma}[Contact Transition] The transition map $\on{Tr}\Gamma$ along a segment $\Gamma$ of the characteristic foliation $\Sigma_\eta$ connecting local sections $D$ and $D'$ is a contact map on any $U \subseteq D$ where it is well-defined. 
\end{lemma}

There is a key construction for modifying transition maps using plugs, which will be used frequently. Fix a local sections $D$ and $D'$ of the characteristic foliation and a segment of flowline
\[\Gamma \qquad\text{from a point $P$ in $D$ to a point $P'$ in $D'$}\]

\begin{definition}[Well-Positioned] \label{def:well_positioned} A plugging domain $U$ is \emph{well-positioned} with the trajectory $\Gamma$ if 
\[
U \cap \Gamma = [0,1] \times O \subset U_{\on{std}} \qquad\text{where $O \in D_{\on{std}}$ is the origin} 
\]
\end{definition}

\begin{remark} \label{rmk:existence_of_well_positioned} Note that, for any point $O$ along $\Gamma$ and any neighborhood $W$ of $O$ intersecting $O$ in an embedded arc, there is a plugging domain $U \subset W$ well-positioned with $\Gamma$ such that $O$ is identified with the origin in $D_{\on{in}} = D_{\on{std}}$ is the origin.
\end{remark}

Given a plugging domain $U$ that is well-positioned with $\Gamma$, we can decompose $\Gamma$ into the concatenation of the intersection $U \cap \Gamma$ connecting $D_{\on{in}}$ and $D_{\on{out}}$ with the two segments
\[
\Gamma_{\on{in}} \text{ connecting $D$ to $D_{\on{in}} \subset U$} \qquad\text{and}\qquad \Gamma_{\on{out}}\text{ connecting $D_{\on{out}}\subset U$ to $D'$}
\]
Note that the transition map along $U \cap \Gamma$ is simply the identity map $D_{\on{in}} = D_{\on{out}} = D_{\on{std}}$. By Remark \ref{rmk:composition}, we can therefore write the transition map along $\Gamma$ as the composition
\[
\on{Tr}\Gamma = \on{Tr}\Gamma_{\on{out}} \circ \on{Tr}\Gamma_{\on{in}}
\]

\begin{definition}[Insertion Along Trajectory] \label{def:insertion_along_trajectory} Fix a plugging domain $U$ well-positioned with $\Gamma$ and a sufficiently $C^0$-small contact Hamiltonian
\[H:U_{\on{std}} = [0,1] \times D_{\on{std}} \to \R \]
such that $H$ vanishes near the boundary of $U_{\on{std}}$ and such that the time 1 contactomorphism $\Phi_H$ generated by $H$ preserves the origin. There is an associated insertion
\[
\iota_{U,H}:[0,1] \times \Sigma \to C\Sigma \qquad\text{of the plug $\iota_H$ associated to $H$}
\]
will be referred to in short as the \emph{insertion of the plug associated to $H$ along $\Gamma$ at $U$}. \end{definition}

By Lemma \ref{lem:Ham_flow}, the characteristic foliation of $\eta_{U,H}$ of the deformed contact Hamiltonian structure associated to the insertion above is given, on the plugging domain $U$, by
\[
Z_H = \partial_t + V_H \qquad\text{under the identification $U \simeq U_{\on{std}} = [0,1] \times D_{\on{std}}$}
\]
In particular, there is a trajectory of the deformed characteristic foliation connecting $0 \times O$ in $D_{\on{in}}$ to $1 \times O$ in $D_{\on{out}}$ and the transfer map along this trajectory is identified with
\[\Phi_H:D_{\on{std}} \to D_{\on{std}}\]
By concatenating this trajectory with the trajectories $\Gamma_{\on{in}}$ and $\Gamma_{\on{out}}$, we acquire a continuation of the original trajectory $\Gamma$ that we denote as follows.
\[
\Gamma_{U,H} \qquad\text{connecting $D$ to $D'$}
\]
The following lemma is immediate from this discussion and the composition property of transition maps (Remark \ref{rmk:composition}).

\begin{lemma}[Transition Map Deformation] \label{lem:transition_map_deformation} Let $\iota_{U,H}$ be the insertion of a plug associated to a contact Hamiltonian $H$ on $[0,1] \times D_{\on{std}}$ along a trajectory $\Gamma$ at a well-positioned plugging domain $U$. Then
\[
\on{Tr}\Gamma_{U,H}:D \to D' \qquad\text{is given by}\qquad \on{Tr}\Gamma_{U,H} = \on{Tr}\Gamma_{\on{out}} \circ\Phi_H \circ \on{Tr}\Gamma_{\on{in}}
\]
\end{lemma}

\section{Liouville And Hyperbolic Invariant Sets} \label{sec:liouville_and_hyperbolic_invariant_sets}

In this section, we study hyperbolic invariant sets of characteristic foliations. In particular, we prove that basic sets of characteristic foliations are automatically Liouville (Theorem \ref{thm:intro_basic_sets_are_Liouville}) and then deduce the key convexity obstruction in this paper (Corollary \ref{cor:convexity_obstruction}).

\subsection{Invariant Measures And Cohomology} \label{subsec:invariant_measures} We start by reviewing the theory of invariant measures and cohomology for continuous flows and singular line fields.  Fix a continuous flow
\[
\Phi:\R \times X \to X \qquad\text{on a compact metric space}\qquad X
\]
Recall that the space of \emph{coboundaries} and the \emph{cohomology} of the flow are given respectively by
\[
B(X,\Phi) = \on{close}\big(F \circ \Phi_s - F \; : \; s \in \R \text{ and }F \in C^0(X;\R)\big) \quad \text{and}\quad H(X,\Phi) = C^0(X;\R)/B(X,\Phi)
\]
where the closure is taken within the space of continuous functions $C^0(X;\R)$. Similarly, the space of \emph{signed invariant measures} is the dual space to the space of the cohomology.
\[
H(X,\Phi)^\vee
\]
Finally, an \emph{invariant measure} is a signed invariant measure that integrate positively against every positive continuous function. We denote this subset by
\[M(X,\Phi) \subset H(X,\Phi)^\vee\]

In the case of a differentiable flow on a smooth manifold, we have the following alternative description of the space of coboundaries.

\begin{lemma} Let $X$ be a compact manifold with a $C^1$ flow $\Phi$ generated by a $C^1$ vector field $V$. Then
\[
B(X,\Phi) = \on{close}\big(VF \; : \; F \in C^1(X;\R)\big)
\]
\end{lemma}

\begin{proof} Let $B = B(X,\Phi)$ denote the space of coboundaries and $B'$ denote the closure of the space of Lie derivatives $VF$. If $G = VF$, then $G$ is given by the limit $s^{-1} \cdot (F \circ \Phi_s - F)$ as $s \to 0$ and therefore $G$ is $B$. Since $B'$ is the closure of the space of all such $G$ and $B$ is closed, this shows that $B' \subseteq B$. Conversely if $G = F \circ \Phi_s - F$ then we can write
\[
G = \int_0^s V(F \circ \Phi_r) dr
\]
This implies that $G$ is the limit of linear combinations of functions of the form $VH$, from which it follows that $G$ is in $B'$. Thus $B' \subseteq B$ and the two spaces must coincide. \end{proof}

We can associate cohomology space and a space of invariant measures to a singular line field in a natural way. Fix a closed smooth manifold
\[
X \qquad\text{with a singular line field}\qquad L
\]
Given a pair of spanning vector fields $U$ and $V$ of the singular line field $L$, there is a positive smooth function $F$ such that $V = FU$. There are canonical isomorphisms of Banach spaces
\[
B(X,U) \to B(X,V) \quad\text{and}\quad H(X,U) \to H(X,V)\qquad\text{given by}\qquad G \mapsto FG
\]
This induces a dual identification of the invariant (signed) measures denoted by
\[
H(X,U)^\vee \to H(X,V)^\vee \quad\text{restricting to}\quad M(X,U) \to M(X,V)
\]

\begin{definition}[Cohomology/Measures For Line Fields] The \emph{cohomology} of a singular line field $L$ on a closed manifold $X$ are the Banach spaces
\[
H(X,L) = \underset{Z}{\text{colim}} \; H(X,V) 
\]
where the colimit is taken over the space of spanning vector fields $V$.  We analogously define the spaces of signed invariant measures and invariant measures
\[
H(X,L)^\vee \qquad\text{and}\qquad M(X,L)
\]\end{definition}

\subsection{Divergence Cohomology Class} Next we discuss the divergence of volume form along a singular line field as a cohomology class. Recall that the divergence of a smooth vector field $V$ with respect to a volume form $\mu$ is the function
\begin{equation}
\on{div}(V,\mu) \qquad\text{defined implicitly by}\qquad d(\iota_V\mu) = \on{div}(V,\mu) \cdot \mu
\end{equation}

\begin{definition}[Divergence] The \emph{divergence class} of a singular line field $L$ on a smooth manifold $X$ is the cohomology class
\[
\on{div}(L) \in H(X,L)
\]
given by $\on{div}(V,\mu)$ in $H(X,V)$ for any spanning vector field $V$ and any smooth volume form $\mu$.
\end{definition}

\begin{lemma}[Well-Defined] The divergence class $\on{div}(L)$ is well-defined as a cohomology class.
\end{lemma}

\begin{proof} Given a different spanning vector field $V$ and a smooth volume form $\mu$. Given a positive smooth function $F$, we can compute that
\[
\on{div}(FV,\mu) = F \cdot \on{div}(V,\mu) + VF = F \cdot \on{div}(V,\mu) + FV(\on{log}(F)) 
\]
This formula states that the divergence $\on{div}(FV,\mu)$ is identified with $\on{div}(V,\mu)$ under the identification of cohomology classes between $V$ and $FV$. Similarly, we can compute that
\[
\on{div}(V,F\mu) = \on{div}(V,\mu) + V(\on{log}(F))
\]
This shows that the cocycle is independent of the choice of $\mu$. \end{proof}

\begin{example}[Closed Orbit] \label{ex:closed_orbit} Suppose that $\Gamma$ is a closed orbit of a singular line field $L$ and consider the natural invariant Lebesgue probability measure
\[
\mu_\Gamma \in M(X,L)
\]
supported on the closed orbit $\Gamma$. The divergence evaluates against this measure as
\begin{equation}
\int_X \on{div}(L) \cdot \mu_\Gamma = \on{log}|\on{det}(T_O\on{Ret}_\Gamma)|
\end{equation}
Here $\on{Ret}_\Gamma:D \to D$ is the Poincare return map of $\Gamma$ with respect to any local section $D$ along $\Gamma$ and $O \in D$ is the fixed point of $\on{Ret}_\Gamma$ corresponding to the orbit $\Gamma$.
\end{example}

The divergence cocycle can be used to detect the existence of a spanning vector field that has positive (or negative) divergence along a given invariant set. More precisely, we have the following divergence croterion for convexity, which may be viewed as an analogue of the McDuff-Sullivan criterion for Reeb vector fields \cite{duff1987applications}.

\begin{proposition}[Div Criterion] \label{prop:sullivan} Let $\Lambda$ be a compact invariant set of a singular line field $L$ on a closed manifold $X$ and  of $L$. Then there is a spanning vector field $V$ with $\on{div}(V,\mu) > 0$ along $\Lambda$ if and only if
\[
\int_X \on{div}(L) \cdot \nu > 0 \qquad\text{for every invariant measure $\nu \in M(X,L)$} 
\]
Similarly, there is a spanning vector field $V$ with $\on{div}(V,\mu) < 0$ along $\Lambda$ if and only if
\[
\int_X \on{div}(L) \cdot \nu < 0 \qquad\text{for every invariant measure $\nu \in M(X,L)$} 
\]\end{proposition}

\begin{proof} Fix a spanning vector field $V$ generating a flow $\Phi$ on $X$. The set of invariant measures in $M(X,V)$ with support in $\Lambda$ is equivalent to the set of invariant measures $M(\Lambda,\Phi)$ of $\Phi|_\Lambda$. The existence of a spanning vector field $V$ with $\on{div}(V,\mu) > 0$ along $\Lambda$ evidently implies that
\[
\int_X \on{div}(L) \cdot \nu = \int_X \on{div}(V,\mu) \cdot \nu > 0 \qquad\text{for every invariant measure $\nu \in M(X,L) = M(X,V)$} 
\]
Conversely, suppose that the divergence $\on{div}(L) = \on{div}(V,\mu)$ pairs positively with any invariant measure $\nu$ in $M(\Lambda,\Phi)$. Then we claim that there is a coboundary
\[
H:\Lambda \to \R \qquad\text{such that}\qquad \on{div}(V,\mu) + H > 0
\]
Indeed, suppose otherwise. Then by the Hahn-Banach separation theorem, there is a signed invariant measure $\nu \in H(\Lambda,\Phi)^\vee$ in the dual of $H(\Lambda,\Phi)$ such that
\begin{equation} \label{eq:inequality_Hahn_Banach}
\int_\Lambda \on{div}(V,\mu) \cdot \nu \le 0 < \Big(\int_\Lambda F\nu \; : \; F \in C^0(\Lambda,\R_+)\Big) 
\end{equation}
This inequality implies that $\nu$ is a measure since it is posiitve on any positive function. Thus (\ref{eq:inequality_Hahn_Banach}) contradicts the hypothesis on the divergence. We can extend the coboundary $H$ on $\Lambda$ to a coboundary on $X$. Moreover, by the density of the functions $ZG$ in the space of boundaries, we may assume that $H = -U(\on{log}(F))$ where $F$ is a positive smooth function on $X$. Then we have
\[
F \cdot \on{div}(V,\mu) = \on{div}(U,\mu) - U(\on{log}(F)) > 0
\]
where $V$ is a spanning vector field defined by $FV = U$. An analogous proof applies for the negative divergence case. \end{proof}

\subsection{Liouville Criterion Via Divergence} \label{subsec:liouville_subsets} 

The following result (cf. \cite[Lem 2.36]{chaidez2026conformally}) characterizes Liouville subsets using the divergence.

\begin{lemma}[Divergence Criterion] \label{lem:div_criterion} Fix a subset $\Lambda \subset \Sigma$ of a contact Hamiltonian manifold and a volume form $\mu$. Then $\Lambda$ is positive Liouville if and only if there is a characteristic vector field $Z$ such that
\[
\on{div}(Z,\mu) > 0 \qquad\text{along $\Lambda$}
\]
Similarly, $\Lambda$ is negative Liouville if and only if there is a characteristic vector field with negative divergence.\end{lemma}

\noindent In general, it is possible for a subset of a contact Hamiltonian manifold can be both positive and negative Liouville. In contrast, we have the following lemma.

\begin{lemma}[Sign Exclusivity] \label{lem:sign_exclusive} Let $\Lambda \subset \Sigma$ be a compact invariant set the characteristic foliation of a contact Hamiltonian manifold $(\Sigma,\eta)$ that is Liouville. Then $\Lambda$ is not both positive and negative.
\end{lemma}

\begin{proof} Suppose that this were the case. Choose a characteristic vector field $Z$, a volume form $\mu$ and a positive smooth function $F$ such that
\[
\on{div}(Z,\mu) > 0 \qquad\text{and}\qquad \on{div}(Z,F\mu) = F \cdot \on{div}(Z,\mu) + ZF < 0
\]
Take any invariant probability measure $\nu$ supported on the invariant set $\Lambda$. Then we have
\[
\int_Y (F \cdot \on{div}(Z,\mu) \cdot \nu  = \int_Y (F \cdot \on{div}(Z,\mu) + ZF) \cdot \nu = \int_Y \on{div}(Z,F\mu) \cdot \nu 
\]
The left hand side is positive and the right hand side is negative. This is a contradiction.
\end{proof}

\subsection{Hyperbolic Basic Sets Are Liouville} We now prove that hyperbolic basic sets are automatically Liouville, generalizing the result of Honda-Huang \cite{hh2019} for hyperbolic fixed points and Breen for hyperbolic periodic orbits \cite{b2021}.  We recall the basic definition of hyperbolicity.

\begin{definition}[Hyperbolic Set] A compact invariant set $\Lambda$ of a flow $\Phi$ on a manifold $X$ is \emph{hyperbolic} if there are $\Phi$-invariant, continuous sub-bundles of the tangent bundle $TX$ called
\[
\text{the \emph{stable bundle} $E^s(\Phi,\Lambda)$} \qquad\text{and} \qquad \text{the \emph{unstable bundle} $E^u(\Phi,\Lambda)$} 
\]
such that the tangent bundle of the manifold $X$ splits as
\[
TX|_\Lambda =  \on{span}(Z) \oplus E^s(\Phi,\Lambda) \oplus E^u(\Phi,\Lambda)
\]
and such that there exist $A,B > 1$ such that,  for any $U \in E^u(\Phi,\Lambda)$, $V \in E^s(\Phi,\Lambda)$ and $s \in \R$
\begin{equation} \label{eq:hyperbolic_constants}
|T\Phi_s(U)| \ge \exp(As + B) \cdot |U| \qquad\text{and}\qquad |T\Phi_s(V)| \le \exp(-As + B) \cdot |V|
\end{equation}
 \end{definition}

\begin{definition}[Index] The \emph{index} $\on{ind}(\Lambda) $ of a hyperbolic invariant set $\Lambda$ of a flow $\Phi$ is the rank of the stable bundle $E^s(\Phi,\Lambda)$.
\end{definition}

\begin{definition}[Basic] A hyperbolic invariant set $\Lambda$ is \emph{basic} if it is transitive and locally maximal.
\end{definition}

We recall a few basic results about hyperbolic invariant sets that we will use later. First, we have the following reparametrization property (cf. \cite[Thm 5.1.11]{fh2019}).

\begin{lemma}[Reparametrization] \label{lem:reparametrization} Let $\Lambda$ be a hyperbolic invariant set of a flow $\Phi$ be a flow on a manifold $X$. Then $\Lambda$ is a hyperbolic invariant set of any smooth time reparametrization $\Psi$ of $\Phi$.
\end{lemma}

\noindent In particular, Lemma \ref{lem:reparametrization} implies that hyperbolicity is a property of invariant sets of singular line fields. Second, we have the following orbit density theorem.

\begin{theorem}[Orbit Density] \label{thm:orbit_density} Let $\Lambda$ be a hyperbolic basic set of a flow $\Phi$. Then the orbit measures
\[
\mu_\gamma \qquad\text{where $\gamma$ is a periodic orbit in $\Lambda$}
\]
are dense in the set of all probability measures supported on $\mu$. \end{theorem}

We also require the following lemma for hyperbolic fixed points of contactomorphisms. Our proof uses Lemma \ref{lem:det_bound_map} from Section \ref{subsec:invariant_manifolds_of_contactomorphisms} below, where many related results will also be proven.

\begin{lemma}[Determinant Bound For Flows] \label{lem:det_bound_flows} Let $\Gamma$ be a hyperbolic closed orbit of index $n-1$ or less of the characteristic foliation of a contact Hamiltonian manifold $(\Sigma,\eta)$ . Then
\[
\on{det}(T_O\on{Ret}_\Gamma) \ge \sigma_{cu}(T_O\on{Ret}_\Gamma)^n
\]
where $\sigma_{cu}(T_O\on{Ret}_\Gamma)$ is the smallest absolute value of an eigenvalue of the differential of $\on{Ret}_\Gamma$ at $O$ that is greater than one. \end{lemma}

\begin{proof} Let $O$ be the hyperbolic fixed point of the contact embedding $\Phi = \on{Ret}\Gamma$ as a partially defined contactomorphism of a local section $D$. Note that $\on{ind}(O) = \on{ind}(\Gamma) \le n-1$. Lemma \ref{lem:det_bound_map} in Section \ref{subsec:invariant_manifolds_of_contactomorphisms} then states that the differential $T_O\Phi$ satisfies
\[
\on{det}(T_O\Phi) \ge \sigma_{cu}(T_O\Phi)^n \qedhere
\]
\end{proof}

We are now ready to prove the main result of this section. This is precisely a restatement of Theorem \ref{thm:intro_basic_sets_are_Liouville} in intrinsic terms via contact Hamiltonian manifolds.

\begin{theorem}[Theorem \ref{thm:intro_basic_sets_are_Liouville}] \label{thm:hyperbolic_implies_Liouville} Let $\Lambda \subset \Sigma$ be a hyperbolic basic set of the characteristic foliation of a contact Hamiltonian $2n$-manifold $(\Sigma,\eta)$. Then $\Lambda$ is Liouville with
\[
\Lambda 
\text{ positive} \qquad\text{if and only if}\qquad \text{$\on{ind}(\Lambda) \le n-1$}
\]
\end{theorem}

\begin{proof} Fix a hyperbolic basic set $\Lambda$ of the characteristic foliation $\Sigma_\eta$ with
\[
\on{ind}(\Lambda) \le n-1
\]
Fix a characteristic vector field $Z$ generating a characteristic flow $\Phi$. By scaling $Z$, we may assume that every closed orbit of $\Phi$ in a small neighborhood of $\Lambda$ has period $1$ or greater. Fix constants $A,B$ satisfying (\ref{eq:hyperbolic_constants}). We claim that
\begin{equation} \label{eq:div_lower_bound}
\int_\Sigma \on{div}(\Sigma_\eta) \cdot \nu \ge nA \cdot \int_\Sigma \nu > 0 \qquad\text{for all $\Sigma_\eta$-invariant measures $\nu$ supported on $\Lambda$}
\end{equation}
By Proposition \ref{prop:sullivan} and Lemma \ref{lem:div_criterion}, this will prove that $\Sigma_\eta$ is positive Liouville. Moreover, reversing the orientation of $\Sigma$, and thus of the characteristic foliation, swaps the stable and unstable bundles of $\Lambda$. Thus, the case where $\on{ind}(\Lambda) \ge n$ follows by switching orientations. To demonstrate (\ref{eq:div_lower_bound}), we consider two cases. 

\vspace{3pt}

\noindent {\bf Case 1: Orbit Measures.} We first consider the case where $\nu = \mu_\Gamma$ is the invariant probability measure supported on a closed hyperbolic orbit $\Gamma \subset \Lambda$. Fix a local section $D$ of $\Sigma_\eta$ centered at a point $O \in \Gamma$ and let $T$ be the period of $\Gamma$. By Example \ref{ex:closed_orbit} and Lemma \ref{lem:det_bound_flows}, we see that
\begin{equation} \label{eq:liouville_hyperbolic_1}
\int_\Sigma \on{div}(\Sigma_\eta) \cdot \mu_\Gamma \ge \on{log}|\on{det}(T_O\on{Ret}_\Gamma)| \ge n \cdot \on{log}\big(\sigma_{cu}(T_O\on{Ret}_\Gamma)\big)
\end{equation}
where $T_O\on{Ret}_\Gamma$ is the linearized return map of $\Gamma$ and $\sigma_{cu}(T_O\on{Ret}_\Gamma)$ is the smallest absolute value of an eigenvalue with norm greater than one. On the other hand, we can write
\[
T_O\on{Ret}\Gamma = T\Phi_T|_E \qquad\text{where}\qquad E = E^s_O(\Phi,\Lambda) \oplus E^u(\Phi,\Lambda)_O \subset T_O\Sigma
\]
Moreover, $E^s_O(\Phi,\Lambda)$ and $E^u_O(\Phi,\Lambda)$ are precisely the direct sum of the real invariant subspaces of $T_O\on{Ret}_\Gamma$ corresponding to eigenvalues of modulus less than $1$ and greater than $1$, respectively. In particular, there is a non-zero vector $V \in E^u_O(\Phi,\Lambda) \subset T_O\Sigma$ such that
\[
\sigma_{cu}(T_O\on{Ret}_\Gamma)^k \cdot |V|  = |T\Phi_T(V)| \ge \exp(kAT + B) \cdot |V| \qquad\text{for any $k \in \N$}
\]
By dividing by $|V|$, taking a logarithm and passing to the limit as $k \to \infty$, we find that
\begin{equation} \label{eq:liouville_hyperbolic_2}
\on{log}(\sigma_{cu}(T_O\on{Ret}_\Gamma)) \ge AT \ge A
\end{equation}
Combining (\ref{eq:liouville_hyperbolic_1}) and (\ref{eq:liouville_hyperbolic_2}) yields the desired inequality in this case.

\vspace{3pt}

\noindent {\bf Case 2: General Measures.} Now consider any invariant measure $\nu$ supported on $\Lambda$. By Theorem \ref{thm:orbit_density}, there is a sequence of invariant measures $\mu_i$ supported on closed orbits $\gamma_i \subset \Lambda$ such that $\mu_i \to \nu$ weakly in the space of invariant measures. It follows that
\[
\int_\Sigma \on{div}(\Sigma_\eta) \cdot \nu = \lim_{i \to \infty} \int_\Sigma \on{div}(\Sigma_\eta) \cdot \mu_i \ge nA \cdot  \lim_{i \to \infty} \int_\Sigma \nu_i = nA \cdot \int_\Sigma \nu \qedhere
\]\end{proof}

\vspace{3pt}

We are now prepared to prove Corollary \ref{cor:convexity_obstruction}, following the proof that was sketched in the introduction. We also state this result in the language of contact Hamiltonian manifolds.

\begin{corollary}[Corollary \ref{cor:convexity_obstruction}] Let $C_+$ and $C_-$ be basic sets of the characteristic foliation of a contact Hamiltonian $2n$-manifold $(\Sigma,\eta)$. Then
\[W^s(C_+) \cap W^u(C_-) = \emptyset \qquad\text{if the indices satisfy $\on{ind}(C_+) < n \le \on{ind}(C_-)$}\]
In particular, the pair $C = (C_+,C_-)$ does not form a heterodimensional cycle. 
\end{corollary}

\begin{proof} Any convex contact Hamiltonian manifold $(\Sigma,\eta)$ can be split along a dividing set $\Gamma \subset \Sigma$ into a positive Liouville domain $\Sigma_+$ and a negative Liouville domain $\Sigma_-$.
\[
\Sigma = \Sigma_+ \cup \Sigma_- \qquad\text{with}\qquad \Gamma = \Sigma_+ \cap \Sigma_-
\]
Here $\Gamma$ is transverse to $\Sigma_\eta$. Moreover, the characteristic foliation $\Sigma_\eta$ points outward from $\Sigma_+$ and inward to $\Sigma_-$ along $\Gamma$. It follows from Theorem \ref{thm:hyperbolic_implies_Liouville} and Lemma \ref{lem:sign_exclusive} that the hyperbolic basic sets $C_-$ and $C_+$ must satisfy
\[
C_+ \subset \Sigma_+ \qquad\text{and}\qquad C_- \subset \Sigma_-
\]
from which it follows that there cannot be a heteroclinic from $C_-$ to $C_+$.  \end{proof}

\section{Creation Of Simple Heteroclinic Cycles} \label{sec:construction_of_heteroclinic_cycle}

We next show that any contact Hamiltonian manifold can be deformed to contain a simple heteroclinic cycle. We formulate our result via contact Hamiltonian plugs (Definition \ref{def:contact_Hamiltonian_plug}).

\begin{theorem}[Local Cycle Creation] \label{thm:local_cycle_creation} For any $\epsilon$ and any $n \ge 2$, there is a contact Hamiltonian $2n$-plug
\[\iota:[0,1] \times U_{\on{std}} \to CU_{\on{std}}\]
with $C^0$-size less than $\epsilon$ whose characteristic foliation contains a simple heteroclinic cycle of index $(n-1,n)$. \end{theorem}

\begin{proof}[Proof Sketch] First, we construct small plug containing a periodic orbit (Lemma \ref{lem:orbit_creation}). We then show that any contact embedding with a fixed point can be modified by a small Hamiltonian perturbation so that it is the identity near the fixed point (Corollary \ref{cor:local_small_Ham}) and that the identity map can be further perturbed to contain a heteroclinic cycle of the desired index (Corollary \ref{cor:cycle_creation_for_Id}). We can thus insert a plug associated to a well-chosen, small contact Hamiltonian along the periodic orbit from the first step, in order to modify the transition map to possess a heteroclinic cycle. \end{proof}

\begin{figure}[h]
    \centering
    \includegraphics[width=.9\linewidth]{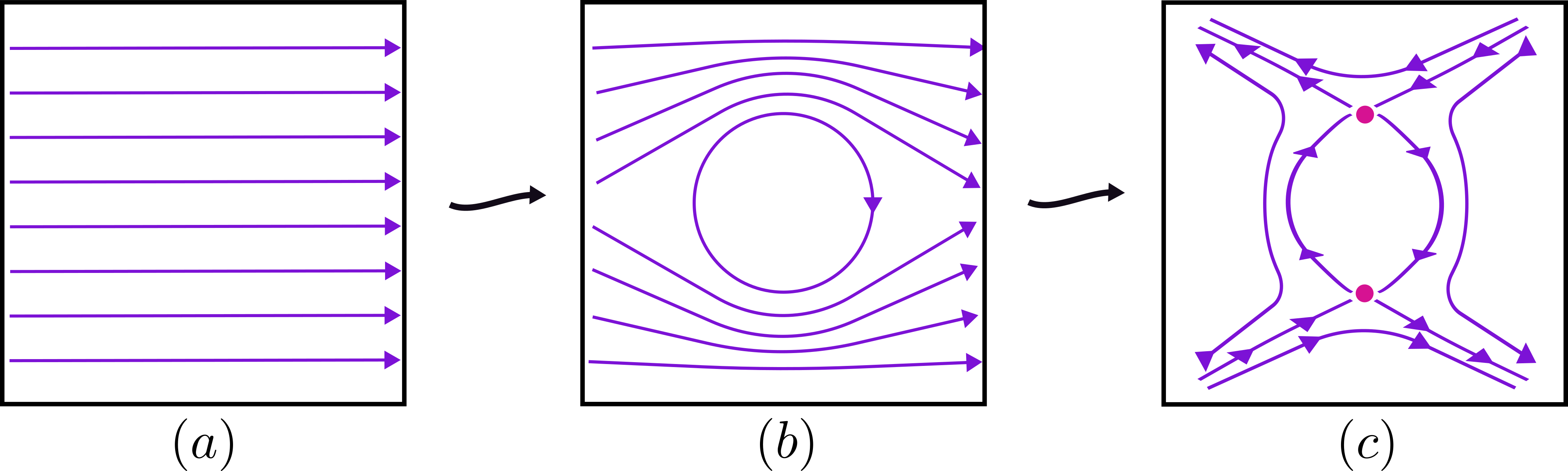}
    \caption{A cartoon of the steps in the proof of Theorem \ref{thm:local_cycle_creation} depicting:
    (a) a local model of the unperturbed characteristic foliation; (b) the introduction of a periodic orbit; and (c) the introduction of a heteroclinic cycle along the orbit. }
    \label{fig:4.1_proof_sketch}
\end{figure}

\begin{remark}[Theorem \ref{thm:cycle_creation}] Given an arbitrary hypersurface $\Sigma$ in a contact manifold $(Y,\xi)$ and an open set $U \subset Y$ intersecting $\Sigma$, we can insert the plug in Theorem \ref{thm:local_cycle_creation} in a plugging domain contained in $U$ to produce an $C^0$-small isotopy supported in $U$ connecting $\Sigma$ to a hypersurface $\Sigma$ containing a heterodimensional cycle $C$ with index $(n-1,n)$. Thus Theorem \ref{thm:local_cycle_creation} immediately implies Theorem \ref{thm:cycle_creation} from the introduction.
\end{remark}

 In the rest of this section, we develop the necessary lemmas used in the proof sketch above (Sections \ref{subsec:orbit_creation}-\ref{subsec:cycle_creation_for_maps}). We provide a full version of the proof at the end of the section (Section \ref{subsec:proof_of_local_cycle_creation}).

\subsection{Closed Orbit Creation} \label{subsec:orbit_creation} We start by showing that we can find a plug whose characteristic foliation contains a periodic orbit.

\begin{lemma}[Orbit Creation] \label{lem:orbit_creation} For any $\epsilon$ and any $n \ge 2$, there is a contact Hamiltonian $2n$-plug
\[\iota:[0,1] \times U_{\on{std}} \to CU_{\on{std}} \qquad\text{ with size less than $\epsilon$}\]
such that the characteristic foliation of the deformed contact Hamiltonian structure $\eta_\iota$ has a closed orbit $\Gamma$.
\end{lemma}

\begin{proof} We consider the standard Liouville form $\lambda$ on $U_{\on{std}} = [0,1]_t \times D_{\on{std}}$ given by $e^t\alpha_{\on{std}}$ where $\alpha_{\on{std}}$ is the standard contact form. The contactization $CU_{\on{std}}$ with respect to this Liouville form is
\[
C U_{\on{std}} = \R_s \times U_{\on{std}} \qquad\text{with contact form}\qquad ds + \lambda
\]
Given any smooth function $H:U_{\on{std}} \to \R$, we can form the graphical deformation $\lambda_H = dH + \lambda$. The Liouville vector field $Z_H$ of $\lambda_H$ is a characteristic vector field for $\eta_H = \on{ker}(\lambda_H)$ which satisfies
\[
Z_H = Z - X_H \qquad\text{for the Hamiltonian vector field $X_H$ of $H$}
\]
We must show that for each $\epsilon$, there is a smooth function $H:U_{\on{std}} \to [-\epsilon,\epsilon]$ that is the identity near $\partial U_{\on{std}}$ such that $Z_H$ has a closed orbit. We construct this in steps.

\vspace{3pt}

\noindent {\bf Step 1.} Take a symplectic embedding $W \to U_{\on{std}}$ where $W \subset T^*B^n$ is a small ball in the cotangent bundle of the $n$-ball intersecting the zero section. Let $\Lambda \subset U_{\on{std}}$ be a disk in this chart contained in the image of the zero section $T^*B^n$ and let $\lambda'$ be the pushforward of the standard Liouville form on $T^*B^n$ to $U$. Note that this Liouville form satisfies
\[
\lambda' = 0 \qquad\text{along the disk $\Lambda$}
\]
Moreover, since $\lambda' - \lambda$ is closed and $U$ is a ball, there is a smooth function $F$ on $U$ such that $dF = \lambda' - \lambda$. By adding a constant, we may assume that $F(x) = 0$ at a point $x \in \Lambda$. We then let
\[
G = \chi F \qquad\text{for a smooth bump function $\chi:U \to [0,1]$ with $\varphi = 1$ near $x$}
\]
By choosing the support of $\chi$ sufficiently small around $x$, and by shrinking the disk $\Lambda$ and the Darboux chart $U$, we can assume that
\[
|G| \le \epsilon/2 \qquad\text{and}\qquad \lambda_G = dG + \lambda = \lambda' \text{ in a Darboux chart $U$ containing $\Lambda$ and $p$}
\]

\vspace{3pt}

\noindent {\bf Step 2.} Next, we may choose coordinates $(r,\theta,q_1 \dots q_{n-2})$ on $B^n$ where $(r,\theta)$ are radial coordinates on $B^2 = B^2 \times 0 \subset B^n$ and such that $\Lambda$ is the 2-disk of radius $\le \delta$ contained in $B^2 \times 0$. Let $\phi$ denote the momentum coordinate corresponding to $\theta$ and consider a smooth Hamiltonian
\[
F:T^*B^n \to \R \qquad\text{with}\qquad F = \phi \text{ near }\partial\Lambda = \{(\delta,\theta) \in B^2 \; : \; \theta \in S^1\}
\]
It is simple to compute that the Hamiltonian vector field of $F$ is given by
\[
X_F = \partial_\theta \qquad\text{along}\qquad \partial\Lambda
\]
In particular, $X_F$ has $\partial\Lambda$ as a closed loop. Moreover, since $\phi$ vanishes near the zero section of $T^*B^n$ and in particular along $\partial\Lambda$, we can cutoff $F$ so that it satisfies
\[
|F| \le \epsilon/2 \qquad\text{and}\qquad \text{$F$ compactly supported in the chart $U$}
\]
We now pullback $F$ through the chart map to a function on $U$, which we will also denote as $F$. 

\vspace{3pt}

\noindent {\bf Step 3.} Finally, we let $H = F + G$ and let $\lambda_H = dF + \lambda_G = dH + \lambda$. Then the Liouville vector field is given by $Z_H = Z_G - X_F = -X_F$ along $\partial\Lambda$ and thus has $\partial\Lambda$ as a periodic orbit. The size of the deformation is given by $|F + H| \le |F| + |H| \le \epsilon$. \end{proof}

\subsection{Conformally Symplectic Linear Maps} Next, we briefly review conformally symplectic linear maps and their contact lifts. Recall that the conformal symplectic group
\[
\CSp(2n) \subset \on{GL}(2n;\R)
\]
is the group of orientation preserving invertible linear maps that preserve the standard linear symplectic form up to a positive conformal factor
\[
\on{CF}(A) \qquad\text{defined by}\qquad A^*\Omega_{\on{std}} = \on{CF}(A) \cdot \Omega_{\on{std}} 
\]

There is an important non-degeneracy property for conformally symplectic maps that we will need to use later in this section.

\begin{definition}[Conformally Non-Singular] \label{def:conformally_non_singular} A linear conformal symplectomorphism $A$ in $\CSp(2n)$ is \emph{conformally non-singular} if the conformal factor is not an eigenvalue of $A$.
\end{definition}

\begin{lemma} \label{lem:CSp_nonsing} The conformally non-singular elements are open and dense in $\CSp(2n)$.
\end{lemma}

\begin{proof} Openess is clear since the spectrum and conformal factor vary continuously. To prove density, fix $A \in \CSp(2n)$. For any constant $C \in \R_+$, we have $\on{CF}(CA) = C^2 \cdot \on{CF}(A)$ and the eigenvalues of $CA$ are given by $C\lambda$ for eigenvalues $\lambda$ of $A$. Thus $CA$ is conformally non-singular if $C \cdot \on{CF}(A)$ is not an eigenvalue of $A$. This holds for a generic constant in $\R_+$. 
\end{proof}

A linear conformal sympectomorphism naturally lifts to a contactomorphism as follows. Consider the standard contact Euclidean space
\[
\R^{2n+1} = \R \times T^*\R^n  \qquad\text{with contact form}\qquad \alpha_{\on{std}} = dz + \lambda_{\on{std}}
\]

\begin{definition}[Lifted Contactomorphism] \label{def:lifted_contactomorphisms} The \emph{lifted contactomorphism} $\widetilde{A}$ corresponding to a linear conformal symplectomorphism $A \in \CSp(2n)$ is the contactomorphism
\[
\widetilde{A}:\R^{2n+1} \to \R^{2n+1} \qquad\text{given by}\qquad \widetilde{A}(z,x,y) = \big(\on{CF}(A) \cdot z + f_A(x,y),A(x,y)\big)
\]
Here $f_A$ is the unique function on $\R^{2n} = T^*\R^n$ that vanishes at the origin and that satisfies
\[
df_A = A^*\lambda_{\on{std}} - \on{CF}(A) \cdot \lambda_{\on{std}}
\]\end{definition}

We will require a few elementary properties of lifted contactomorphisms. First, we have the following computation of the differential.

\begin{lemma}[Differential] \label{lem:CSp_differential} The lifted contactomorphism $\widetilde{A}$ of a linear conformal symplectomorphism $A$ fixes the origin and has differential at the origin given by
\[
T\widetilde{A} = \on{CF}(A) \oplus A \in \on{End}(\R \oplus \R^{2n})
\]
\end{lemma}

\begin{proof} Simply note that $f_A(0) = 0$ and that $\lambda_{\on{std}}$ vanishes at $0$, so that $df_A = 0$ at $0$.
\end{proof} 

\noindent Second, we have the following useful linearization theorem near fixed points of contact embeddings. Fix an open subset $U \subset Y$ of a contact manifold $(Y,\xi)$ and a contact embedding
\[
\Phi:U \to Y \qquad \text{with a fixed point}\qquad x \in U
\]
The differential of the contact embedding restricts to a conformally symplectic bundle map
\[
T\Phi|_\xi:\xi \to \xi \qquad\text{defined over $U$}
\]

\begin{proposition}[Linearization] \label{prop:linearization} Let $\Phi:U \to Y$ be a contact embedding with a fixed point $x$ such that
\[
T\Phi|_\xi \text{ is conformally non-singular at $x$}
\]
Then there are Darboux coordinates on a neighborhood $W$ centered at $x$ and a conformally symplectic linear map $A$ such that
\[
T\Phi = T\widetilde{A} \qquad\text{ at the point $x = 0$ in Darboux coordinates on $W$}
\]\end{proposition}

\begin{proof} In any Darboux chart sending $x$ to $0 \in \R^{2n+1}$, the contact structure at $x$ is given by $\R^{2n} \subset \R^{2n+1}$ and the restriction $T\Phi|_\xi$ is identified with a conformally symplectic linear map $A$. The full differential $T\Phi$ at $x$ in this chart can be computed as the block matrix
\[
T\Phi = \left[
\begin{array}{cc}
\on{CF}(A) & 0\\
* & A
\end{array}\right] \quad\text{in}\quad \on{End}(\R \oplus \R^{2n})
\]
Thus by Lemma \ref{lem:CSp_differential}, it suffices to show that we can choose the chart so that $T\Phi$ is diagonal at $0$. To prove this, note that $\on{CF}(A)$ is an eigenvalue of $T\Phi$. If $A$ is conformally non-singular, then the eigenvalue $\on{CF}(A)$ is simple and has an eigenvector $R$ that is positively transverse to $\xi = \R^{2n}$ at $x$. Choose any contact form $\alpha$ and let $\beta = f\alpha $ where $f$ is a positive function such that at $x$ we have
\[
1 = \beta(R) = f \cdot \alpha(R) \qquad \text{and}\qquad 0 = \iota_Rd\beta|_\xi = \alpha(R) \cdot df + f \cdot \iota_Rd\beta
\]
Then $R$ is the Reeb vector field of $\beta$ at $x$. We can take a strict contact Darboux chart (cf. \cite[Thm 2.5.1]{geiges}) to get a chart where $R = \partial_z$. This concludes the proof.
\end{proof}

\subsection{Isotopy To Identity} \label{subsec:isotopy_to_Id} We next discuss a procedure for performing an isotopy of a contactomorphism with a fixed point to a contactomorphism that is the identity near that point.

\begin{proposition}[Isotopy To Identity] \label{prop:isotopy_to_Id} Let $\Phi:U \to Y$ be a contact embedding with a fixed point $x$. Then there is an open set $W \subset U$ containing $x$ and an isotopy of contact embeddings
\[
\Psi:[0,1] \times W \to Y \qquad\text{with}\qquad \Psi_0 = \on{Id} \quad \Psi_1 = \Phi|_W \quad\text{and}\quad \Psi_t(x) = x \text{ for $t \in [0,1]$}
\]
\end{proposition}

\begin{corollary}[Local Small Hamiltonian] \label{cor:local_small_Ham} Let $\Phi:U \to Y$ be a contact embedding of an open set $U \subset Y$ with a fixed point $x$. Then for any $\epsilon$ and any contact form $\alpha$ on $Y$, there is a contact Hamiltonian
\[
H:[0,1] \times U \to [-\epsilon,\epsilon]
\]
that is compactly supported in $(0,1) \times U$ and with time 1 contact Hamiltonian flow $\Phi_H = \Phi$ near $x$. \end{corollary}

\begin{proof} (Corollary \ref{cor:local_small_Ham}) Let $W \subset U$ and $\Psi$ be the open set and isotopy in Proposition \ref{prop:isotopy_to_Id}. After reparametrizing in time, we may assume that $\Psi_t$ is constant for $t$ near $0$ or $1$. Then the contact vector field $V$ generating $\Psi$ vanishes identically at $x$ since $\Psi_t(x) = x$ for all $t$, and for $t$ near $0$ and $1$.  We may then choose $H$ to agree with $\alpha(V)$ in a neighborood of $x$ and to vanish outside of $W$. The contact Hamiltonian flow of $H$ must then agree with $\Psi_t$ in a small neighborhood of $x$. 
\end{proof}

The strategy for Proposition \ref{prop:isotopy_to_Id} can be summarized as follows. First, we prove the result in the case where the contact embedding is small in the $C^1$-topology using a standard Legendrian graph trick (Lemma \ref{lem:small_contactomorphism}). We then use composition with a certain standard family of lifted conformally linear contactomorphisms to reduce to the $C^1$-small case.

\vspace{3pt}

\begin{lemma}[$C^1$-Small Case] \label{lem:small_contactomorphism} Fix a closed contact manifold  $(Y,\xi)$ with a Riemannian metric. Then there is a constant $\epsilon$ such that, for any domain $U \subset Y$ and any contact embedding $\Phi:U\to Y$ with
\[
\|\Phi - \on{Id}\|_{C^1} \le \epsilon
\]
there is an isotopy of contact embeddings $\Phi_t$ from $\on{Id}$ to $\Phi$ with the property that any fixed point $x$ of $\Phi$ is a fixed point of $\Phi_t$ for all $t$. \end{lemma}

\begin{proof} This is a standard argument using the Legendrian graph (cf. \cite{ms2017} for the analogous argment in the symplectic setting). We consider two cases.

\vspace{3pt}

\noindent {\bf Case 1: Contactomorphisms.} First assume that $U = Y$ so that $\Phi$ is a contactomorphism. Fix a contact form $\alpha$ on $Y$ and consider the contact manifold
\[Y \times \R_s \times Y \qquad\text{with contact form}\qquad \pi_1^*\alpha - e^s\pi_2^*\alpha\]
Here $\pi_i$ denotes the projection map from $Y \times \R_s \times Y$ to the $i$th factor of $Y$. Any contactomorphism $\Phi:Y \to Y$ determines a Legendrian graph $\Gamma_\Phi \subset Y \times \R_s \times Y$ given by
\[\Gamma_\Phi = \{(x,F(x),\Phi(x) \; : \; x \in Y\} \qquad\text{where $F:Y \to \R_+$ is defined by $\Phi^*\alpha = F\alpha$}\]
Let $\Gamma_{\on{Id}} \subset Y \times \R_s \times Y$ denote the Legendrian diagonal given by the graph of $\on{Id}$. We can take a Legendrian Weinstein neighborhood around $\Gamma_{\on{Id}}$ identified with an open neighborhood
\[
W \subset \R \times T^*Y \qquad\text{of the diagonal $\Gamma_{\on{Id}} = Y$}
\]
Let $\pi:\R \times T^*Y \to Y$ denote the projection to the base in the jet bundle $\R \times T^*Y$. There is an $\epsilon$ such that if $\|\Phi - \on{Id}\|_{C^1} \le \epsilon$ then $\Gamma_\Phi$ is identified with the 1-jet of a smooth function
\[G_\Phi:Y \to \R \] Conversely, any sufficiently $C^2$-small function $G$ determines a Legendrian $\Gamma_G$ contained in $W$ that is $C^1$-close to the zero section via the 1-jet, and thus a contactomorphism $\Psi_G$ defined by
\[
\Psi_G(x) = \pi_2(\pi_1^{-1}(x) \cap \Gamma_G)
\]
In particular, we can write a contact isotopy $\Phi:[0,1] \times Y \to Y$ with $\Phi_0 = \on{Id}$ and $\Phi_1 = \Phi$ by taking $\Phi_t = \Psi_{tG}$ where $G$ is the function corresponding to $\Phi$. The fixed points of $\Psi_G$ are precisely the zeros of $G$, and so it is simple to see that $\Phi_t = \Psi_{tG}$ has the same fixed points for all $t$. 

\vspace{3pt}

\noindent {\bf Case 2: Embeddings.} Now assume more generally that $U$ is any domain in $Y$. Choose a constant $\epsilon$ as in Case 1. We can apply the construction of Case 1 to acquire a Legendrian with boundary
\[
\Gamma_\Phi \subset W \subset \R \times T^*Y\qquad\text{that is $C^1$-close to $U \subset Y$}
\]
This Legendrian $\Gamma_\Phi$ is the 1-jet of a smooth function $G$ defined over the domain $U' = \pi(\Gamma_\Phi)$ where $\pi$ is projection from $\R \times T^*Y$ to $Y$. We can extend $G$ to a smooth function over all of $Y$ and this defines a contactomorphism
\[
\Psi_G:Y \to Y \qquad\text{such that}\qquad \Psi_G|_U = \Phi
\]
Then we take our isotopy of embeddings $\Phi_t$ to be the restriction of the contact isotopy $\Psi_{tG}$ for $t \in [0,1]$ to the domain $U$.  \end{proof}

We can now prove the full version of Proposition \ref{prop:isotopy_to_Id} by using lifted conformally symplectic linear maps to reduce to the case of $C^1$-small contact embeddings.

\begin{proof} (Proposition \ref{prop:isotopy_to_Id}) By possibly shrinking $U$ and passing to a Darboux chart, we may assume that $Y = D$ is the standard contact disk in $\R^{2n+1}$ and that $x = 0$ is the origin. Let $A \in \CSp(2n)$ be the conformally symplectic linear map given by $T\Phi|_\xi$ at $0$. We consider two cases.

\vspace{3pt}

\noindent {\bf Case 1: Non-Singular Differential.} First suppose that $A$ is conformally non-singular. Then by Proposition \ref{prop:linearization}, we can pass to a Darboux chart $W$ where $T\Phi = \widetilde{A}$ at $x = 0$. Since the conformally symplectic group $\CSp(2n)$ is connected, we can choose a path
\[
B:[0,1] \to \on{CSp}(2n) \qquad\text{with}\qquad B_0 = \on{Id} \text{ and }B_1 = A^{-1}
\]
We can then lift $B$ to a path of lifted contactomorphisms $\widetilde{B}_t$. Note that $\widetilde{B}_t$ fixes the origin for all $t$ by Lemma \ref{lem:CSp_differential}. Thus the path of contact embeddings $\Phi \circ \widetilde{B}_t$ is then an isotopy of contact embeddings from a small neighborhood $W$ of $0$ to $D$ such that
\[
T\Phi_1 = \on{Id} \text{ at the origin $x = 0$}\qquad\text{and}\qquad \Phi_t(0) = 0 \text{ for all $x$}
\]
By further shrinking the chart $W$, we may assume that $\|\Phi_1 - \on{Id}\|$ is arbitrarily small on $W$. We may also assume that $W$ is a domain (i.e. a compact  codimension zero sub-manifold with boundary). Now Lemma \ref{lem:small_contactomorphism} provides an isotopy $\Psi_t$ from $\Phi_1$ to $\on{Id}$ such that $\Psi_t(0) = 0$ for all $t$. The concatenation of $\Phi_t$ and $\Psi_t$ therefore gives the desired isotopy from $\on{Id}$ to $\Phi$ on $W \subset U$.

\vspace{3pt}

\noindent {\bf Case 2: General Differential.} In the general case, choose a path $B_t$ in $\on{CSp}(2n)$ so that $B_0 = \on{Id}$ and $AB_1$ is conformally non-singular. This path exists by Lemma \ref{lem:CSp_nonsing}. Then the isotopy $\Psi_t = \Phi \circ \widetilde{B}_t$ restricted to a small open $W \subset U$ is an isotopy from $\Phi$ to a map satisfying the assumptions of Case 1. This reduces to that case.
\end{proof}

\subsection{Cycle Creation For Maps} \label{subsec:cycle_creation_for_maps} Here we construct contact vector fields (and thus arbitrarily small contactomorphisms) that contain heteroclinic cycles between two hyperbolic fixed points. 

\begin{proposition}[Contact Vector Field With Cycle] \label{prop:contact_vector_field_with_cycle} For any $n \ge 2$, there is a contact vector field
\[
V \qquad\text{on the standard contact disk $D_{\on{std}}$ of dimension $2n-1$}
\]
that has hyperbolic critical points $O$ of index $n-1$ and $P$ of index $n$ such that
\[
W^s(V,O) \cap W^u(V,P) \neq \emptyset \qquad\text{and}\qquad W^u(V,O) \cap W^s(V,P) \neq \emptyset 
\]\end{proposition}

\begin{corollary}[Cycle Creation For Identity] \label{cor:cycle_creation_for_Id} For any $\epsilon$ and any $n \ge 2$, there is a contact Hamiltonian
\[
H:[0,1] \times D_{\on{std}} \to [-\epsilon,\epsilon] \qquad\text{on the standard disk $D_{\on{std}}$ of dimension $2n-1$}
\]
that is compactly supported in $(0,1) \times \on{int}(D_{\on{std}})$ and such that $\Phi_H$ contains a simple heteroclinic cycle consisting of two hyperbolic fixed points $O$ and $P$, where $O$ is the origin.
\end{corollary}

\begin{proof} Let $V$ be the vector field in Proposition \ref{prop:contact_vector_field_with_cycle} with critical points $O$ and $P$. We may assume after applying a contactomorphism that $O$ is the origin. Let $\Phi:[0,1] \times D \to D$ be the flow of $\delta V$ for small $\delta$. Evidently $\Phi$ is identity near $\partial D$. Moreover, we can reparametrize in time to a contact isotopy $\Phi'$ generated by a vector field $V'$ with $V'_t = 0$ for $t$ near $0$ and $1$, and such that $|V'_t| \le 2\delta |V|$. If $\delta$ is chosen small, the Hamiltonian $H$ for $V'$ will satisfy $|H| \le \epsilon$ and $\Phi'_1$ will be the time $\delta$ flow of $V$, which will have a heteroclinic cycle consisting of the origin $O$ and $P$.
\end{proof}

For the proof of Proposition \ref{prop:contact_vector_field_with_cycle} , we first construct a certain standard local model for one of the heteroclinics. Denote the standard contact Euclidean space by
\[
\R^{2n-1} = \R \times T^*\R^{n-1}  \qquad\text{with the standard contact form}\qquad dz + \lambda_{\on{std}}
\]
Let $O_{\on{std}}$ denote the origin of $\R^{2n-1}$ and let $P_{\on{std}}$ denote the point in the zero section defined by
\[P_{\on{std}} = (1,0 \dots 0) \in \R^{n-1} \subset \R \times T^*\R^{n-1}\]
Let $\Xi_{\on{std}}$ denote open line segment in $\R^{n-1}$ connecting $O_{\on{std}}$ and $P_{\on{std}}$ given by
\[
\Xi_{\on{std}} = \{(x_1,0 \dots 0) \; : \; 0 < x_1 < 1\} \subset \R^{n-1}
\]
Finally, let $\Gamma^u_{\on{std}}$ and $\Gamma^s_{\on{std}}$ denote the positive half of the Reeb trajectory starting at $O_{\on{std}}$ and the negative half of the Reeb trajectory ending on $P_{\on{std}}$ respectively.
\[
\Gamma^u_{\on{std}} = [0,\infty)_z \times O_{\on{std}} \subset \R^{2n-1} \qquad\text{and}\qquad \Gamma^s_{\on{std}} = (-\infty,0]_z \times P_{\on{std}} \subset \R^{2n-1}
\]

\begin{lemma}[Single Heteroclinic] \label{lem:single_het} For any $\epsilon$ and any $n \ge 2$, there is a contact vector field $V$ on the standard contact manifold $\R^{2n-1}$
with the following properties.
\begin{itemize}
    \item[(a)] There are hyperbolic critical points of $V$ at $O_{\on{std}}$ of index $n-1$ and $P_{\on{std}}$ of index $n$.\vspace{2pt}
    \item[(b)] The segment $\Xi_{\on{std}}$ from $P_{\on{std}}$ to $O_{\on{std}}$ is a heteroclinic, or equivalently
    \[\Xi_{\on{std}} \subset W^s(V,O_{\on{std}}) \cap W^u(V,P_{\on{std}})\]
    \item[(c)] The positive Reeb trajectory $\Gamma^u_{\on{std}}$ starting at $O_{\on{std}}$ is contained in $W^u(V,O_{\on{std}})$ and the negative Reeb trajectory $\Gamma^s_{\on{std}}$ ending on $P_{\on{std}}$ is contained in $W^s(V,P_{\on{std}})$.\vspace{2pt}
    \item[(d)] The vector field $V$ satisfies $V = \partial_z$ outside of the set
    \[
    U_\epsilon = \on{Nbhd}_\epsilon(\R^{n-1}) \cap \on{Nbhd}_\epsilon(\Gamma^u_{\on{std}} \cup \Xi_{\on{std}} \cup \Gamma^s_{\on{std}})
    \]
\end{itemize}
\end{lemma}

\begin{proof} We start by fixing some choices and notation. First, let $V_+$ denote the contact vector field
\[
V_+ = 2z\partial_z + 3\sum_i y_i \partial_{y_i} - x_i \partial_{x_i} \qquad\text{with Hamiltonian}\qquad H_+ = 2z - \sum_i x_i y_i
\]
Note that this vector field has a single hyperbolic critical point of index $n-1$ at $O_{\on{std}}$. Second, let $V_-$ be the translation of $-V_{\on{std}}$ so that the single critical point lies on $P_{\on{std}}$ and let $H_-$ be the corresponding Hamiltonian. This critical point of $V_-$ has index $n$. Finally, pick smooth functions
\[
\phi:\R^{2n-1} \to [0,1] \qquad\text{and}\qquad \psi:\R^{2n-1} \to [0,1] 
\]
satisfying the following properties: $\phi$ depends only on the $x$-variables, $\phi = 1$ near $O_{\on{std}}$ and $\phi = 0$ near $P_{\on{std}}$; and $\psi$ depends only on the $z$-variable in a neighborhood of $\Gamma^u_{\on{std}}$ and $\Gamma^s_{\on{std}}$, $\psi = 1$ near $\Xi_{\on{std}}$ and $\psi = 0$ outside of the set $U_\epsilon$. Define $V = V_H$ to be the contact Hamiltonian vector field of
\[
H = (1 - \psi) + \psi(\phi H_+ + (1-\phi)H_-)
\]
Note that properties (a) and (d) are immediate since $H = H_+$ near $P_{\on{std}}, H = H_-$ near $q_{\on{std}}$ and $H = 1$ outside of $U_\epsilon$ by construction. Let us now verify (b) and (c). 

\vspace{3pt}

First, we verify property (b). Note that $H_+$ and $H_-$ vanish along the zero section $\R^{n-1}$. The contact Hamiltonian equations then imply that $V$ is given by the combination
\[
V = \phi V_+ + (1-\phi) V_- \qquad\text{in a neighborhood of the segment $\Xi_{\on{std}}$}
\]
Directly along the line segment $\Xi_{\on{std}}$ connecting $O_{\on{std}}$ to $P_{\on{std}}$ in $\R^n$, this further simplifies to
\[
V = -\phi x_1 \partial_{x_1} + (1-\phi)(x_1 - 1) \partial_{x_1} \qquad\text{on the line segment } \Xi_{\on{std}}
\]
It follows from this expression that the segment $\Xi_{\on{std}}$ is a heteroclinic from $y_{\on{std}}$ to $x_{\on{std}}$. Second we verify property (c). It follows from the choices of $H_\pm$ and $\psi$ that
\begin{equation} \label{eq:single_het_lemma_1} dH_+ = 2 dz \text{ along $\Gamma^u_{\on{std}}$}\qquad H_- = -2 dz \text{ along $\Gamma^s_{\on{std}}$}\qquad d\psi = \partial_z\psi \cdot dz \qquad\text{along $\Gamma^u_{\on{std}}$ and $\Gamma^s_{\on{std}}$}\end{equation}
Note that the contact Hamiltonian $H$ has contact Hamiltonian vector field $V$ parallel to the Reeb vector field $R$ at a point $x$ if and only if $dH = dH(R) \cdot \alpha$ at $x$. From (\ref{eq:single_het_lemma_1}), we compute that
\begin{equation} \label{eq:single_het_lemma_2} 
H = (1 - \psi) + \psi H_+ \quad\text{and thus}\quad dH = (H_+ - 1) \cdot d\psi + \psi\cdot dH_+ \quad\text{ near $\Gamma^u_{\on{std}}$}
\end{equation}
An analogous formula holds with $H_-$ near $\Gamma^s_{\on{std}}$. It follows that $V$ is proportional to $R$ and of the form $V = H \cdot R$ along $\Gamma^u_{\on{std}}$ and $\Gamma^s_{\on{std}}$. From (\ref{eq:single_het_lemma_2}), we can compute that
\begin{equation} \label{eq:single_het_lemma_3} 
V = ((1 - \psi) + 2z\psi)\partial_z \qquad\text{along $\Gamma^u_{\on{std}}$}\qquad\text{and}\qquad V = ((1 - \psi) - 2z\psi)\partial_z \qquad\text{along $\Gamma^s_{\on{std}}$}
\end{equation}
It is immediate from (\ref{eq:single_het_lemma_3}) and the construction of $\psi$ that this is a positive multiple of $\partial_z$ along $\Gamma^u_{\on{std}}$ and $\Gamma^s_{\on{std}}$. It follows that $\Gamma^u_{\on{std}}$ is a trajectory limiting to $O_{\on{std}}$ backward in time and $\Gamma^s_{\on{std}}$ is a trajectory limiting to $P_{\on{std}}$ forward in time. In other words, $\Gamma^u_{\on{std}} \subset W^u(V,O_{\on{std}})$ and $\Gamma^s_{\on{std}} \subset W^s(V,P_{\on{std}})$.  \end{proof}

\begin{proof} (Proposition \ref{prop:contact_vector_field_with_cycle}) Consider the standard contact disk $D \subset \R \times T^*\R^{n-1}$ with the standard contact form $dz + \lambda_{\on{std}}$ and Reeb vector field $R = \partial_z$. Let $\Lambda \subset D$ be the standard Legendrian unknot, which must necessarily have a Reeb chord $\Gamma$ from $P$ to $O$. By the (strict) Legendrian Weinstein neighborhood theorem, there is a strict contact embedding
\[
U' \to D \qquad\text{from a neighborhood $U' \subset \R \times T^*\Lambda$ of the zero section $\Lambda$}
\]
that sends the zero section to $\Lambda$. Take an chart $W \subset \Lambda$ along with a diffeomorphism $\R^{n-1} \to W$ sending $O_{\on{std}}$ to $O$ and $P_{\on{std}}$ to $P$. This induces an embedding of jet spaces $\R^{2n-1} \to \R \times T^*\Lambda$ and thus a strict contact embedding
\[
U \to D \qquad\text{from a neighborhood $U \subset \R^{2n-1}$ of $\R^{n-1}$}
\]
This neighborhood must contain the region $U_\epsilon$ in Lemma \ref{lem:single_het}(d) for sufficiently small $\epsilon$. Let $V_\epsilon$ be the corresponding vector field in Lemma \ref{lem:single_het} for this choice of $\epsilon$ and define a contact vector field $V$ on $D$ by
\[
V = V_\epsilon \text{ in $U_\epsilon$} \qquad\text{and}\qquad V = R \text{ outside of $U_\epsilon$}
\]
Lemma \ref{lem:single_het}(d) implies that $V$ is smooth. Lemma \ref{lem:single_het}(a-b) implies that $O$ and $P$ are hyperbolic critical points of $V$ of index $n-1$ and $n$ respectively, with a connecting heteroclinic $P$ to $O$ contained in $U_\epsilon$. Finally, Lemma \ref{lem:single_het}(c-d) implies the Reeb chord $\Gamma$ is tangent to $V$ and intersects the stable manifold of $P$ and the unstable manifold of $O$. It follows that there is a heteroclinic from $O$ to $P$. \end{proof}

\subsection{Proof Of Theorem \ref{thm:local_cycle_creation}} \label{subsec:proof_of_local_cycle_creation} We now assemble the results of this section to prove Theorem \ref{thm:local_cycle_creation}.        

\begin{proof} (Theorem \ref{thm:cycle_creation}) Fix a constant $\delta > 0$. Lemma \ref{lem:orbit_creation} states that there is a plug $\jmath$ of size $\delta$ such that the deformed characteristic foliation of $\eta_\jmath$ has a closed orbit $\Gamma$. Consider the return map
\[
\Phi = \on{Ret}\Gamma:D \to D \qquad\text{for some local section $D$ intersecting $\Gamma$ at $O \in D$}
\]
Fix a neighborhood $W \subset D$ of $O$ such that the inverse of $\on{Ret}\Gamma$ is well-defined on $W$. We can assume that $W$ is contactomorphic to $D_{\on{std}}$ and that $O$ is identified with the origin. Then the inverse $\Phi^{-1}:W \to D$ is a contact embedding fixing the origin, and we can apply Corollary \ref{cor:local_small_Ham} to get a contact Hamiltonian
\[
F:[0,1] \times D \to [-\delta, \delta]
\]
compactly supported in $[0,1] \times W$ such that $\Phi_F = \Phi^{-1}$ in a neighborhood $U$ of $O$. Similarly, Corollary \ref{cor:cycle_creation_for_Id} implies that there is a contact Hamiltonian
\[G: [0,1] \times D \to [-\delta, \delta]\] 
supported on $(0,1) \times W$ such that $\Phi_G$ has a heterodimensional cycle of the desired index in an arbitrarily small neighborhood of $O$ consisting of hyperbolic fixed points $O$ and $P$. By concatenating the Hamiltonians $F$ and $G$, we get a contact Hamiltonian
\[
H:[0,1] \times D \to [-\delta,\delta] \qquad\text{such that $\Phi \circ \Phi_H = \Phi \circ \Phi^{-1} \circ \Phi_G = \Phi_G$}
\]

We can next choose a small plugging domain $U \simeq U_{\on{std}}$ that is well-positioned along $\Gamma$ (Definition \ref{def:well_positioned}) such that $D = D_{\on{in}}$. By inserting the plug associated to $H$ along $\Gamma$ along this domain (Definition \ref{def:insertion_along_trajectory}), we acquire an ambient deformation
\[
\iota_{U,H} \qquad\text{of the contact Hamiltonian manifold $(\Sigma,\eta_\iota)$}
\]
By Lemma \ref{lem:transition_map_deformation}, after inserting this plug, the continuation $\Gamma_{U,H}$ of $\Gamma$ has return map
\[
\on{Ret}\Gamma_{U,H}:D \to D \qquad\text{given by $\Phi \circ \Phi_H = \Phi_G$ on $W$}
\]
In particular, the deformed contact Hamiltonian structure $\eta_{U,H}$ has a heteroclinic cycle $C$ of the desired index, given by the suspension of the corresponding cycle $C$ of $\Phi \circ \Phi_H$. We now let
\[\iota:[0,1] \times U_{\on{std}} \to CU_{\on{std}}\]
be the concatenation of $\jmath$ and $\iota_{U,H}$ as in Remark \ref{rmk:concatenation}. Then the characteristic foliation of $\eta_\iota = \eta_{U,H}$ has a heteroclinic cycle $C$ of the desired index. Moreover, by Remark \ref{rmk:concatenation} and Remark \ref{rmk:insertion_along_domain_of_H}, there is a constant $A$ independent of $U$ and $H$ such that
\[
\|\iota\|_{C^0} \le \|\iota_{U,H} \star \jmath\|_{C^0} \le A\|\iota_{U,H}\|_{C^0} + \|\jmath\|_{C^0} \le (A+1)\delta
\]
By choosing $\delta$ sufficiently small, we assure that the size of our deformation is less than $\epsilon$. \end{proof}

\section{Non-Degenerate Heterodimensional Cycles And Proper Unfoldings} \label{sec:non_degenerate_cycles_and_unfoldings}

In this section, we review the notions of non-degenerate heterodimensional cycles and proper unfoldings as formulated by Li-Turaev \cite{lt2024}. In the process, we introduce key tools that will be used in the construction of unfoldings in the setting of characteristic foliations. 

\subsection{Multipliers And Invariant Manifolds} \label{subsec:hyperbolic_fixed_points} We begin by reviewing key terminology and structures associated to hyperbolic fixed points of diffeomorphisms. Fix a smooth embedding
\[
\Phi:U \to M\qquad \text{to a manifold $M$ from an open set $U \subset M$}
\]

\begin{remark} In later sections, this map will be the Poincare return map of a periodic orbit of a flow. In particular, it will only be well-defined up to smooth conjugacy as a germ of a diffeomorphism.  
\end{remark}

\begin{definition}[Multipliers] A \emph{multiplier} of a fixed point $O$ of the embedding $\Phi$ is an eigenvalue of the differential of $\Phi$ at $O$. Given a fixed point $O$, we use the notation
\[
\lambda_1(O) \dots \lambda_n(O) \qquad\text{and}\qquad S(O) 
\]
for the sequence of multipliers (with repetition, in order of increasing absolute value) and the set of absolute values of the multipliers, respectively. The \emph{invariant subspace}
\[
\text{$E^s(O,\sigma) \subset T_OM$}\qquad\text{and}\qquad \text{$E^u(O,\sigma) \subset T_OM$} \qquad\text{associated to a real number $\sigma \in \R_+$}
\]
is the direct sum of the real invariant subspaces of the differential of $\Phi$ at $O$ corresponding to multipliers of absolute value less than or equal to $\sigma$ and greater than or equal to $\sigma$, respectively. \end{definition}

We next recall the notion of local invariant manifolds and foliations for a fixed point of a local diffeomorphism or embedding. 

\begin{construction}[Local Invariant Manifolds] \label{con:local_invariant_manifolds} Let $O$ be a fixed point of $\Phi$ and let $P$ be a point in a sufficiently small neighborhood of $\Phi$. The 
\emph{local invariant manifolds}
\[
W^s(P,\sigma) \subset U \qquad\text{and}\qquad W^u(P,\sigma) \subset U \qquad\text{for any real number $\sigma \in \R$}
\]
are constructed as follows. Take an open chart $V \subset U$ containing $P$ and $O$ with an chart map $V \to \R^n$ identifying $O$ with the origin. Extend $\Phi|_V$ to a smooth map $\bar{\Phi}:\R^n \to \R^n$ of the form
\[
\bar{\Phi}(x) = A(x) + F(x)
\]
where $A$ is a linear isomorphism and $F$ is a smooth function that vanishes at the origin with small $C^1$-norm. Then $W^s(P,\sigma)$ and $W^u(P,\sigma)$ are  respectively the components of the subsets
\[
\bar{W}^s(P,\sigma) = \big\{Q \in V : \on{dist}(\bar{\Phi}^i(Q),\bar{\Phi}^i(P)) \le Ce^{\sigma i} \text{ for some $C > 0$ and all $i \ge 0$}\big\}\]
\[
\bar{W}^u(P,\sigma) = \big\{Q \in V:\on{dist}(\bar{\Phi}^i(Q),\bar{\Phi}^i(P)) \le Ce^{\sigma i} \text{ for some $C > 0$ and all $i \le 0$}\big\}
\]
that contain the point $P$. Note that these sets depend on the chart $V$ and extension $\bar{\Phi}$. \end{construction}

\begin{definition}[Unstable And Stable Manifolds] The \emph{local stable manifold} and \emph{local unstable manifold} of a hyperbolic fixed point $O$ of the embedding $\Phi$ are the local invariant manifolds
\[
W^s(O) = W^s(O,1) \qquad\text{and}\qquad W^u(O) = W^u(O,1) 
\]
Note that since $O$ is hyperbolic, we have $W^s(O) = W^s(O,-\epsilon)$  and $W^u(O) = W^u(O,\epsilon)$ for $\epsilon$ positive and sufficiently close to zero. We denote the corresponding tangent spaces at $O$ by
\[
E^s(O) \qquad\text{and}\qquad E^u(O)
\]\end{definition}

The following result summarizes the key properties of local invariant manifolds. See Hirsch-Pugh-Shub \cite[Thm 5.5, p. 61]{hps1977} or alternatively Shilnikov-Shilnikov-Turaev-Chua \cite[\S 5.2]{shilnikov2001methods}.

\begin{theorem}[Local Invariant Manifolds] \label{thm:local_invariant_manifolds} Let $O$ be a fixed point of the smooth embedding $\Phi$ and fix real numbers $\sigma$ and $\tau$ that are not in $S(O)$. Then the local invariant manifolds
\[W^s(P,\sigma) \subset U\qquad\text{and}\qquad W^u(P,\tau) \subset U \qquad\text{for $P$ near the fixed point $O$}\]
are locally $\Phi$-invariant and $C^1$-embedded sub-manifolds that form the leaves of continuous foliations
\[\mathcal{F}^s(\sigma) \quad\text{and}\quad \mathcal{F}^u(\tau) \qquad \text{of a small neighborhood of $O$ in $U$}\]
These manifolds and foliations have the following properties.
\begin{itemize} 
    \item[(a)] (Tangent Space) The tangent spaces of the local invariant manifolds at $O$ are given by
\[
TW^s(O,\sigma)|_O = E^s(O,\sigma) \qquad\text{and}\qquad TW^u(O,\tau)|_O = E^u(O,\tau) 
\]
    \item[(b)] (Inclusions) Given real numbers $\sigma < \sigma'$ and $\tau' < \tau$, there are the following inclusions with equality if and only if the intervals $[\sigma,\sigma']$ and $[\tau',\tau]$ are disjoint from $S(O)$.
    \[
W^s(P,\sigma) \subseteq W^s(P,\sigma') \qquad\text{and}\qquad  W^u(P,\tau) \subseteq W^u(P,\tau')
\]
\item[(c)] (Strong Foliations) The invariant manifolds $W^s(P,\sigma)$ and $W^u(P,\tau)$ and the corresponding foliations $\mathcal{F}^s(\sigma)$ are smooth and independent of the choices in Construction \ref{con:local_invariant_manifolds}, for $\sigma < 1 < \tau$.
\vspace{1pt}
\item[(d)] (Extended Manifolds) The local invariant manifolds $W^s(O,\sigma)$ and $W^u(O,\tau)$ may depend on the choices in Construction \ref{con:local_invariant_manifolds} for $\sigma > 1 > \tau$ but the restricted tangent bundles
\[
TW^s(O,\sigma)|_{W^s(O)} \qquad\text{and}\qquad TW^s(O,\tau)|_{W^u(O)} 
\]
are independent of the choices in Construction \ref{con:local_invariant_manifolds} for any hyperbolic fixed point $O$. \end{itemize}
\end{theorem}

\begin{remark}[Alternate Notation] In settings where we must consider the local invariant manifolds of multiple embeddings, we will use the notation
\[
W^s(\Phi;O,\sigma) \qquad\text{and}\qquad W^u(\Phi;O,\tau)
\]
for the stable and unstable invariant manifolds of the embedding $\Phi$. \end{remark}

\begin{remark}[Continuity] \label{rmk:continuity_of_invariants} The local invariant sub-manifolds $W^s(\Phi;O,\sigma)$ and $W^u(\Phi;O,\tau)$ depend continuously in the $C^1$-topology among the embeddings $\Phi$ such that $\sigma$ and $\tau$ are not in $S(O)$. Moreover, if $\sigma < 0 < \tau$ and $\Phi_\epsilon$ is a family of embeddings depending smoothly on a parameter $\epsilon$, then $W^s(\Phi_\epsilon;O,\sigma)$ and $W^u(\Phi_\epsilon;O,\tau)$ depend smoothly on $\epsilon$. See \cite{hps1977} and \cite{shilnikov2001methods} for details.\end{remark}

We will require the following lemma that dictates the behavior of the invariant manifolds with respect to invariant distributions.

 \begin{lemma}[Invariant Distributions] \label{lem:invariant_bundle} Let $O$ be a hyperbolic fixed point of $\Phi$ and fix a real number $\sigma$ not in $S(O)$. Then for any $\Phi$-invariant distribution $D$, we have
 \[
 E^s(O,\sigma) \subset D \text{ at $O$} \qquad\text{if and only if}\qquad TW^s(O,\sigma) \subset D \text{ along the stable manifold $W^s(O)$}
 \]
 \[
 E^u(O,\sigma) \subset D \text{ at $O$} \qquad\text{if and only if}\qquad TW^u(O,\sigma) \subset D \text{ along the unstable manifold $W^u(O)$}
 \] 
 \end{lemma}

\begin{proof} We prove the result for the stable manifolds, since the other case follows by considering the local inverse. Moreover, note that $TW^s(O,\sigma) = E^s(O,\sigma)$ at the fixed point $O$, and so one direction of the result is obvious.
\vspace{3pt}

In the other direction, suppose that $E^s(O,\sigma) \subset D$ at $O$. Consider the continuous distributions tangent to the invariant foliations in Theorem \ref{thm:local_invariant_manifolds}. 
\[
F^s = T\mathcal{F}^s(\sigma) \qquad\text{and}\qquad F^u = T\mathcal{F}^u(\sigma)
\]
These distributions are locally invariant and coincide with $E^s(O,\sigma)$ and $E^u(O,\sigma)$ at $O$, respectively. Since $T_OM = E^s(O,\sigma) \oplus E^u(O,\sigma)$ if $\sigma$ is not in $S(O)$, it follows that
\[
TM = F^s \oplus F^u \qquad\text{in a sufficiently small neighborhood $U'$ of $O$}\]
Given a vector $v$ in $TM$ in $U'$, we write $v^s$ and $v^u$ for its components in $F^s$ and $F^u$. Finally, fix a metric on $TM$. By construction of $E^s(O,\sigma)$ and $E^u(O,\sigma)$ there are constants $\sigma_{cs} < \sigma < \sigma_{cu}$ so that
\begin{equation} \label{eq:lem:invariant_bundle_1}
|T\Phi(u)| < \sigma_{cs}|u| \text{ for any $u \in E^s(O,\sigma)$} \quad\text{and}\quad \sigma_{cu}|v|  < |T\Phi(v)| \text{ for any $v \in E^u(O,\sigma)$}
\end{equation}
Since $F^s$ and $F^u$ are continuous distributions with fibers $E^s(O,\sigma)$ and $E^u(O,\sigma)$ at $O$, it follows that the same inequalities hold for $u \in F^s$ and $v \in F^u$ in a small neighborhood $U'$ of $U$.

\vspace{3pt} 

Next, fix a parameter $\epsilon > 0$ and consider the cone field $C^s_\epsilon$ in the tangent bundle to $U'$ defined by the following inequalities.
\[
C^s_\epsilon = \{v\in TU' \; : \; |v^u| \le \epsilon |v^s|\}
\]
It follows from (\ref{eq:lem:invariant_bundle_1}) that $\Phi^*(C^s_\epsilon) \subset C^s_{c\epsilon}$ where $c =\sigma_{cs}/\sigma_{cu}$. To prove the desired result, it suffices to show that for every $\epsilon$, $x$ in $W^s(O) \cap U'$ and $u$ in $F^s_x = T_xW^s(O,\sigma)$, there is a $v$ in $T_xU'$ such that
\[
v^s = u \qquad\text{and}\qquad v \in D \cap C_\epsilon^s
\]
To prove this claim, note that since $F^s \subset D$ at $O$, there is a neighborhood $N$ of $O$ such that for every $u \in F^s|_N$, there is a $v \in D \cap C_1^s$ with $v^s = u$. On the other hand, for any $x \in W^s(O)$ we know that $\Phi^k(x) \to O$ and thus $\Phi^k(x) \in N$ for large $k$. Thus
\[
\Phi_*^k(u) \in C_1^s \qquad\text{for any $x \in W^s(O) \cap U', u \in F^s_x$ and $k$ large.}
\]
This implies that $u$ is in $C^s_\epsilon$ where $\epsilon = c^k$. Since $c < 1$ and $k$ can be arbitrarily large, this proves the desired claim.\end{proof}

\subsection{Center Multipliers} \label{subsec:center_multipliers} We will be primarily interested in the local invariant manifolds associated to the central multipliers of hyperbolic fixed points.

\begin{definition}[Center Multipliers] Let $O$ be a hyperbolic fixed point of the embedding $\Phi$. The \emph{center-stable norm} and \emph{center-unstable norm} are the elements of $S(O)$ given by
\[
\sigma_{cs}(O) = \max{\sigma \in S(O) \; : \; \sigma < 1}\qquad\text{and}\qquad \sigma_{cu}(O) = \min{\sigma \in S(O) \; : \; \sigma > 1}
\]
The \emph{center-stable} and \emph{center-unstable} multipliers are those with respective norm $\sigma_{cs}(O)$ and $\sigma_{cu}(O)$. 
\end{definition}

\begin{definition}[Center Subspaces] The \emph{center-stable} and \emph{center-unstable} subspaces are the sums of the real invariant subspaces corresponding to the center-stable and center-unstable eigenvalues.
\[
E^{cs}(O) \quad\text{and}\quad E^{cu}(O) 
\]
The \emph{strong stable} and \emph{strong unstable} spaces are the real invariant subspaces corresponding to the eigenvalues of norm less than $\lambda_{cs}(O)$ and greater than $\lambda_{cu}(O)$ respectively, and are denoted by
\[
E^{ss}(O) \quad\text{and}\quad E^{uu}(O) 
\]
Note that the stable and unstable eigenspaces decompose as direct sums
\[
E^s(O) = E^{cs}(O) \oplus E^{ss}(O) \qquad\text{and}\qquad E^u(O) = E^{cu}(O) \oplus E^{uu}(O)
\]
\end{definition}

\begin{definition}[Extended (Un)stable Manifolds] Let $O$ be a hyperbolic fixed point with simple multipliers. The \emph{local extended stable} and \emph{local extended unstable} manifolds
\[
W^{se}(O) \subset U \qquad\text{and}\qquad W^{ue}(O) \subset U
\]
are the unique local invariant manifolds with tangent spaces at $O$ given by the subspaces
\[
E^{se}(O) = E^s(O) \oplus E^{cs}(O)  \quad\text{and}\quad E^{ue}(O) = E^u(O) \oplus E^{cu}(O) 
\]
\end{definition}

\begin{definition}[Strong Foliations] Let $O$ be a hyperbolic orbit with simple multipliers. The \emph{strong stable} and \emph{strong unstable} foliations
\[
\mathcal{F}^{ss} \text{ on } W^s(O) \qquad\text{and}\qquad \mathcal{F}^{uu} \text{ on } W^u(O)
\]
that are the local invariant foliations tangent at $O$ to the subspaces
\[
E^{ss}(O) \quad\text{and}\quad E^{uu}(O)
\]
\end{definition}

\noindent Each of the stable and unstable manifolds are equipped with strong stable and strong unstable foliations, tangent to the corresponding subspaces of the tangent space at the fixed point.

\begin{remark}[Strong Stable/Unstable Manifolds] The leaves of the strong-stable foliation and the strong-unstable foliation that pass through $O$ are denoted by $W^{ss}(O)$ and $W^{uu}(O)$.
\end{remark}

\noindent There is a useful choice of coordinates that straightens out the stable and unstable manifolds near a fixed point. The existence of these coordinates is established in \cite{shilnikov2001methods} and discussed in \cite{lt2024}. 

\begin{theorem}[(Un)stable Linear Coordinates] \label{thm:linear_coordinates} Let $O$ be a hyperbolic orbit. Then there exist two sets of coordinates in a neighborhood of $O$ of the form 
\[
(x_s,x_{cu},x_{uu}) \in \R^a_s \times \R^b_{cu} \times \R^c_{uu} \qquad\text{and}\qquad (y_{ss},y_{cs},y_u) \in \R^d_{ss} \times \R^e_{cs} \times \R^f_u 
\]
called stable linear coordinates $x$ and unstable linear coordinates $y$ with the following properties.
\begin{itemize}
    \item The stable and unstable manifolds $W^s(O)$ and $W^u(O)$ are given in the coordinates by
    \[
    W^s(O) = \{x_{cu} = x_{uu} = 0\} = \{y_u = 0\} \quad\text{and}\quad W^u(O) = \{x_s = 0\} = \{y_{cs} = y_{ss} = 0\}
    \]
    \item The strong stable and strong unstable foliations are given by
    \[
    \mathcal{F}^{ss} = \{x_s = x_{cu} = 0\} \qquad\text{and}\qquad \mathcal{F}^{uu} = \{y_u = y_{cs} = 0\}
    \]
    \item The map $\Phi$ acts linearly on the center-unstable and center-stable components in the $x$-coordinates and $y$-coordinates respectively. That is, there are linear maps
    \[
    A_{cu}:\R^b_{cu} \to \R^b_{cu} \qquad\text{and}\qquad A_{cs}:\R^e_{cs} \to \R^e_{cs} 
    \]
    such that $\Phi$ restricts to $A_{cu}$ on $\{x_s = x_{uu} = 0\}$ and to $A_{cs}$ on $\{y_{ss} = y_u = 0\}$. 
\end{itemize}
\end{theorem}

\noindent Finally, there is an important simplicity condition on the central multipliers that will be needed below. We will say that a multiplier norm $\sigma$ in $S(O)$ is \emph{simple} if
\[
\text{the set of multipliers $\lambda$ with $|\lambda| = \sigma$ is a singleton or a conjugate pair}
\]

\begin{definition}[Simple Center Multipliers] A hyperbolic fixed point $O$ of an embedding $\Phi$ is \emph{simple center-stable multipliers} if $\sigma_{cs}(O)$ is simple and \emph{simple center-unstable multipliers} if $\sigma_{cu}(O)$ is simple. In these respective settings, we let 
\[
\lambda_{cs}(O) \qquad\text{and}\qquad \lambda_{cu}(O)
\]
be the unique center-stable and center-unstable multipliers with non-negative imaginary part.
\end{definition}

\begin{figure}[h]
    \centering
    \includegraphics[width=.6\linewidth]{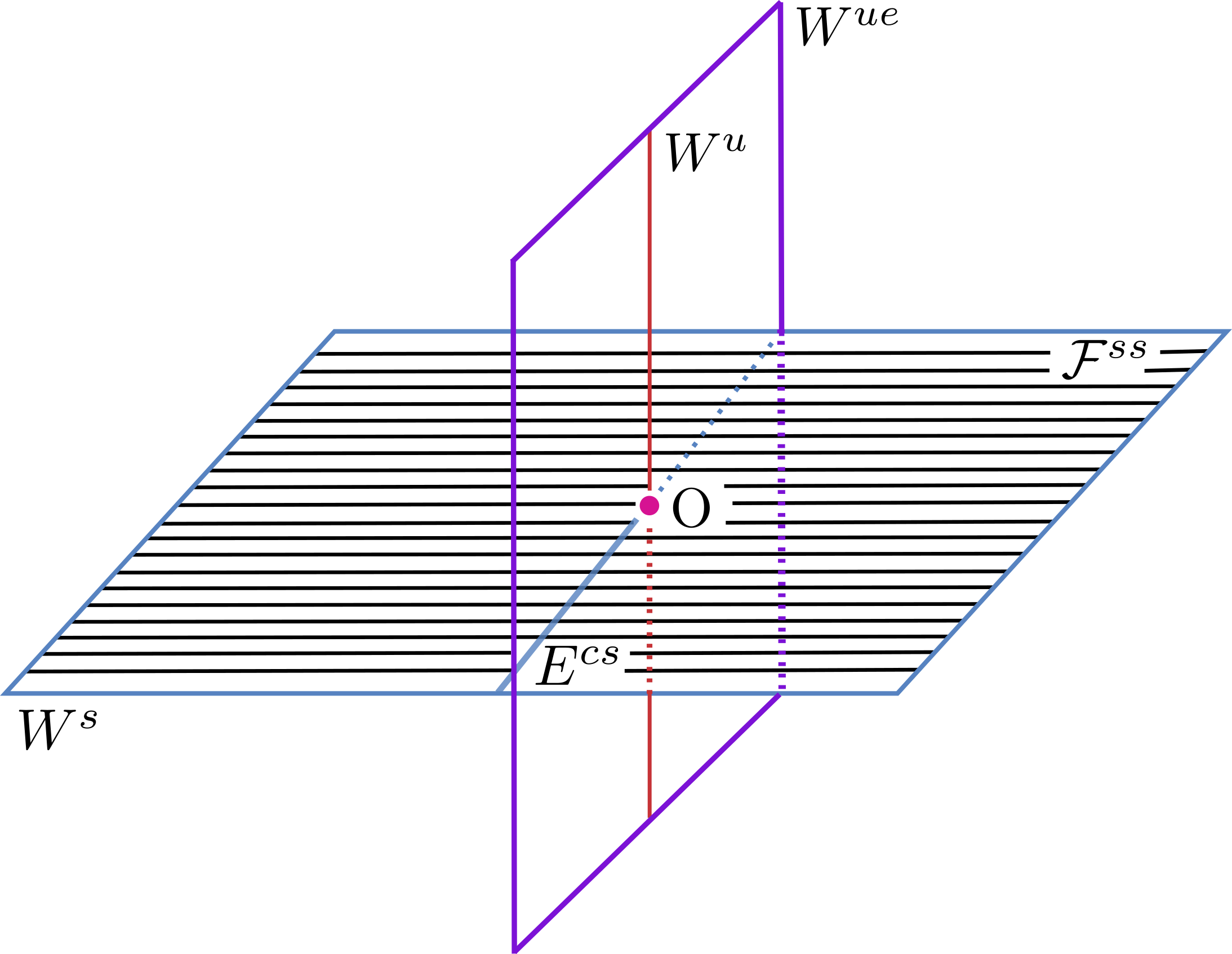}
    \caption{The various invariant manifolds associated to a hyperbolic fixed point $O$ with real and simple center-unstable  multipliers.}
    \label{fig:local_manifolds}
\end{figure}

\subsection{Non-Degeneracy Of Heteroclinic Cycles} We next introduce non-degeneracy for heteroclinic cycles of coindex one following Li-Turaev \cite{lt2024}. We require the following definitions.

\begin{definition}[Coindex] The \emph{coindex} of a heteroclinic cycle $C$ is the difference of the indices of the constituent hyperbolic orbits $C_+$ and $C_-$.
\[
\on{coind}(C) = \on{ind}(C_-) - \on{ind}(C_+)
\]
\end{definition}

\begin{definition}[Simple Center Multipliers] \label{def:simple_center_multipliers_for_cycle} A heterodimensional cycle $C$ has \emph{simple center multipliers} if $C_-$ has simple center-stable multipliers and $C_+$ has simple center-unstable multipliers.
\end{definition}

We also require the following preliminary setup and notation in order to formuate non-degeneracy. Here we will use the notation for transition and return maps from Section \ref{subsec:transition_and_return_maps}. 

\begin{setup}[Non-Degeneracy] \label{set:non_degeneracy} Fix a singular line field $L$ containing a heterodimensional cycle $C$ with coindex one and simple center multipliers. Fix a local sections of $L$ intersecting $C_+$ and $C_-$ at single points denoted as follows.
\begin{equation} \label{eq:choice_of_local_sections}
\text{$D_+$ and $D_-$ with $O_+ = D_+ \cap C_+$ and $O_- = D_- \cap C_-$}
\end{equation}
This determines a pair of Poincare return maps associated to $C_+$ and $C_-$.
\[
\on{Ret}C_+:D_+ \to D_+ \quad\text{and}\quad \on{Ret}C_-:D_- \to D_-
\]
These return maps have hyperbolic a fixed point with simple center-unstable multipliers at $O_+$ and a fixed point with simple center-stable multipliers at $O_-$. Since $C$ forms a heteroclinic cycle, we may fix points along such heteroclinics of the form
\begin{equation} \label{eq:choice_of_heteroclinic_points}
P^u_\pm \in W^u(O_\pm) \subset D_\pm \qquad\text{and}\qquad P^s_\pm \in W^s(O_\pm) \subset D_\pm
\end{equation}
such that $P^u_+$ is connected to $P^s_-$ by a segment of a heteroclinic  from $C_+$ to $C_-$ and $P^s_+$ is connected to $P^u_-$ by a segment of a heteroclinic from $C_-$ to $C_+$. We denote these heteroclinic segments by
\begin{equation} \label{eq:choice_of_heteroclinic_segments}
\Gamma_+ \quad\text{and}\quad \Gamma_-
\end{equation}
We assume that these two segments are disjoint from the sections $D_+$ and $D_-$ other than at their endpoints. In other words, we assume that
\begin{equation} \label{eq:choice_of_heteroclinic_segments_property}
\Gamma_+ \cap D_+ = P^u_+ \qquad \Gamma_+ \cap D_- = P^s_+ \qquad \Gamma_- \cap D_+ = P^s_+ \qquad \Gamma_- \cap D_- = P^u_-
\end{equation}
We will refer to the two heteroclinics as the \emph{robust heteroclinic} $\Gamma_+$ and the \emph{fragile heteroclinic} $\Gamma_-$, respectively. We denote the transition maps along the segments between $P^u_\pm$ and $P^s_\pm$ by
\[
\on{Tr}\Gamma_+:D_+ \to D_- \qquad\text{and}\qquad \on{Tr}\Gamma_-:D_- \to D_+
\]
Note that the points $P^u_\pm$ and $P^s_\pm$ are mapped to each other by these transition maps.
\[
\on{Tr}\Gamma_+(P^u_+) = P^s_- \qquad\text{and}\qquad \on{Tr}\Gamma_-(P^u_-) = P^s_+
\] \end{setup}

\begin{figure}[h]
    \centering
    \includegraphics[width=\linewidth]{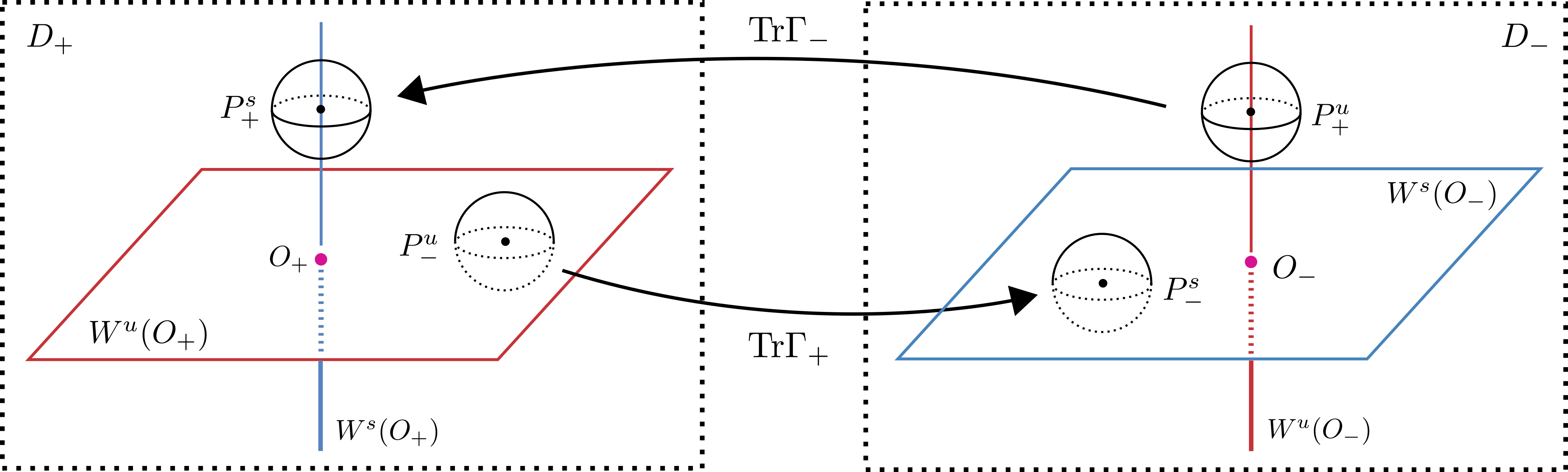}
    \caption{A depiction of the various data fixed in Setup \ref{set:non_degeneracy}. Here we only depict the data on the local sections $D_+$ and $D_-$, and we do not depict the flow itself.}.
    \label{fig:setup}
\end{figure}

We now formulate the non-degeneracy conditions assuming Setup \ref{set:non_degeneracy} as follows. The first three may be viewed as basic transversality and genericity assumptions for the transition maps.

\begin{condition}[Simple Fragile Heteroclinic] \label{cond:simple_fragile_heteroclinic} The transition map of the fragile heteroclinic $\Gamma_-$ satisfies the following transversality properties at the point $P^s_+$.
\[
\on{Tr}\Gamma_-\big(W^u(O_-)\big) \text{ is tranverse to }W^{se}(O_+) 
\quad\text{and}\quad \on{Tr}\Gamma_-\big(W^{ue}(O_-)\big) \text{ is tranverse to }W^s(O_+)
\]
\end{condition}

\begin{condition}[Simple Robust Heteroclinic] \label{cond:simple_robust_heteroclinic} The transition map of the robust heteroclinic $\Gamma_+$ satisfies the following transversality property at the point $P^s_-$.
\[
\on{Tr}\Gamma_+\big(W^u(O_+)) \text{ is transverse to }W^s(O_-)
\]
Moreover, the strong unstable foliation $\mathcal{F}^{uu}$ of the return map $\on{Ret} C_+$ and the strong stable foliation $\mathcal{F}^{ss}$ of the return map $\on{Ret} C_-$ must satisfy the following transversality conditions at $P^s_-$.
\[
\on{Tr}\Gamma_+\big(\mathcal{F}^{uu}\big) \text{ is not tangent to }W^s(O_-) \quad\text{and}\quad \mathcal{F}^{ss} \text{ is not tangent to }\on{Tr}\Gamma_+\big(W^u(O_+)\big)
\]
 \end{condition}

\begin{condition} \label{cond:extra_condition} The points $P^u_+$ and $P^s_-$ must satisfy $P^u_+ \not\in W^{uu}(O_+)$ and $P^s_- \not\in W^{ss}(O_-)$.
\end{condition}
 
\begin{remark}[Intersection Curves] \label{rmk:intersection_curves} Condition \ref{cond:simple_robust_heteroclinic} implies that the unstable manifold $W^u(C_+)$ and the stable manifold $W^s(C_-)$ intersect transversely along the heteroclinic $\Gamma_+$. Therefore
\[
I_- = \on{Tr}\Gamma_+\big(W^u(O_+)) \cap W^s(O_-) \subset D_- \qquad\text{and}\qquad I_+ = W^u(O_+) \cap \on{Tr}\Gamma_+^{-1}\big(W^s(O_-)\big) \subset D_+
\]
are embedded 1-manifolds in a sufficiently small neighborhood of $P^s_-$ and $P^u_+$, respectively. Note that, by construction, the transition map induced by $\Gamma_+$ restricts to a diffeomorphism
\begin{equation}  \label{eq:diffeo_of_intervals}
\on{Tr}\Gamma_+:I_+ \to I_- \qquad\text{such that}\qquad \on{Tr}\Gamma_+(P^u_+) = P^s_-
\end{equation}
Moreover, the two intersection curves satisfy the following trsansversality properties.
\begin{equation} \label{eq:transversality_of_intervals}
I_- \text{ is transverse to }\mathcal{F}^{ss}\text{ in $W^s(O_-)$} \qquad\text{and}\qquad I_+ \text{ is transverse to }\mathcal{F}^{uu} 
\text{ in $W^u(O_+)$}\end{equation}\end{remark}

\begin{remark}[Third Condition] \label{rmk:third_condition} Note that, when Conditions \ref{cond:simple_fragile_heteroclinic} and \ref{cond:simple_robust_heteroclinic} are satisfied, then the intersections $W^{ss}(O_-) \cap I_-$ and $W^{uu}(O_+) \cap I_+$ are isolated. It follows that Condition \ref{cond:extra_condition} can be achieved without perturbation of the line field by simply changing the choice of $P^u_+$ and $P^s_-$ to a different pair of points in $I_+$ and $I_-$. 
\end{remark}

\begin{figure}[h]
    \centering
    \includegraphics[width=\linewidth]{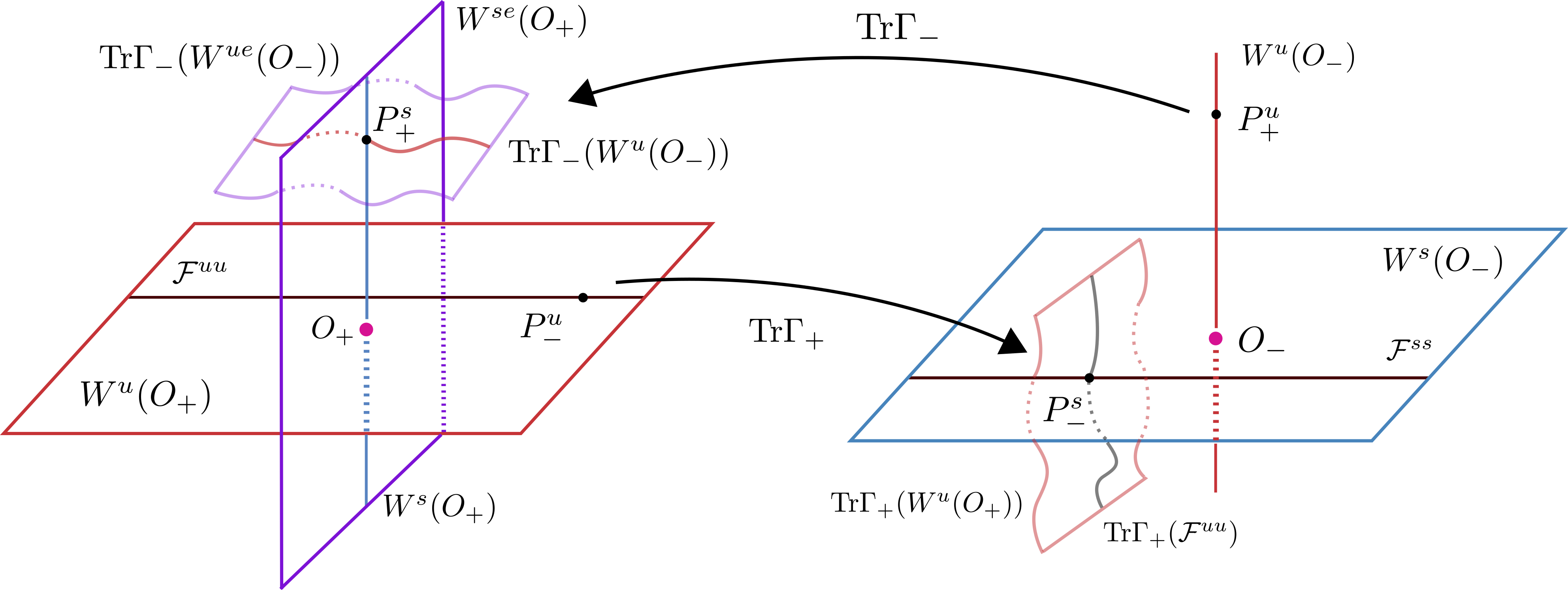}
    \caption{A depiction of data in Setup \ref{set:non_degeneracy} that satisfies the first three non-degeneracy conditions, Conditions \ref{cond:simple_fragile_heteroclinic}, \ref{cond:simple_robust_heteroclinic} and \ref{cond:extra_condition}.}
    \label{fig:nondegeneracy_cond_combined}
\end{figure}

The fourth and final condition depends on the central multipliers of the hyperbolic orbits. Assume Conditions \ref{cond:simple_fragile_heteroclinic}-\ref{cond:extra_condition} and choose linear coordinates as in Theorem \ref{thm:linear_coordinates}.
\begin{equation}\label{eq:linear_coordinates_for_condition}
(x_s,x_{cu},x_{uu}) \text{ for $\on{Ret} C_+$ near $O_+$} \qquad\text{and}\qquad (y_{ss},y_{cs},y_u) \text{ for $\on{Ret} C_-$ near $O_-$}
\end{equation}
Also fix an arbitrary parametrization of the intersection curve $I_-$ that maps the origin to $P^u_+$.
\begin{equation} \label{eq:parametrization_of_I}
\psi:[-1,1]_r \to I_- \subset D_- \qquad\text{such that}\qquad \psi(0) = P^u_+
\end{equation}
The center-unstable coordinates $x_{cu}$ and the center-stable coordinates $y_{cs}$ respectively determine two smooth maps $u_{cu}$ and $v_{cs}$ on the interval $[-1,1]_r$ defined by
\begin{equation}
u_{cu}:[-1,1]_r \to \R^b_{cu} \qquad\text{given by}\qquad u_{cu} = x_{cu} \circ \psi\end{equation}
\begin{equation} v_{cs}:[-1,1]_r \to \R^e_{cs} \qquad\text{given by}\qquad v_{cs} = y_{cs} \circ \on{Tr}\Gamma_+ \circ \psi
\end{equation}
Here $\R^b_{cu} \simeq E^{cu}(O_+)$ and $\R^e_{cs} \simeq E^{cs}(O_-)$ are linearly identified with the center-unstable and center-unstable invariant subspaces of $O_+$ and $O_-$, respectively. Moreover, the properties (\ref{eq:diffeo_of_intervals}) and (\ref{eq:transversality_of_intervals}) and Condition \ref{cond:extra_condition} imply that the following quantities are non-zero.
\[
u_{cu}(0) \neq 0 \qquad v_{cs}(0) \neq 0 \qquad \partial_ru_{cu}(0) \neq 0 \qquad \partial_r v_{cs}(0) \neq 0
\]
We can now formulate the final non-degeneracy conditions, which is broken into cases depending on the center-stable and center-unstable multipliers.

\begin{condition}[Saddle] \label{cond:saddle} If the stable multiplier $\lambda_{cs}(O_-)$ and the unstable multiplier $\lambda_{cu}(O_+)$ are real, then the center-unstable and center-stable coordinates satisfy
\begin{equation} \label{eq:cond:saddle}
\Big|\frac{\partial_r\log |v_{cs}|}{\partial_r\log |u_{cu}|}\Big| \neq 1 \qquad\text{at $r = 0$}
\end{equation} \end{condition}
\begin{condition}[Focus]  \label{cond:focus} If the center-stable multiplier $\lambda_{cs}(O_-)$ is complex, then the vectors
\[
\text{$v_{cs}(0)$ and $\partial_r v_{cs}(0)$ are not proportional as vectors in $\R^2_{cs}$}
\]
Similarly, if the center-unstable multiplier  $\lambda_{cu}(O_+)$ is complex, then the vectors
\[
\text{$u_{cu}(0)$ and $\partial_r u_{cu}(0)$ are not proportional as vectors in $\R^2_{cu}$}
\]\end{condition}

\begin{definition}[Non-Degeneracy] \label{def:non_degeneracy} A coindex one heterodimensional cycle $C$ with simple center multipliers is \emph{non-degenerate} if there exist local sections and heteroclinic points as in Setup \ref{set:non_degeneracy}
\[
D_\pm \text{ as in (\ref{eq:choice_of_local_sections})} \qquad\text{and}\qquad  P^u_\pm \text{ and }P^s_\pm \text{ as in (\ref{eq:choice_of_heteroclinic_points})}
\]
such that the non-degeneracy Conditions \ref{cond:simple_fragile_heteroclinic}, \ref{cond:simple_fragile_heteroclinic}, \ref{cond:extra_condition}, and either \ref{cond:saddle} or \ref{cond:focus} are satisfied. \end{definition}

\subsection{Proper Unfoldings} We next introduce proper unfoldings for heteroclinic cycles following Li-Turaev \cite[\S 2.4]{lt2024}. We continue to assume Setup \ref{set:non_degeneracy} so that we have a heteroclinic cycle
\[
C = (C_+,C_-) \qquad\text{of coindex one for a singular line field $L$}
\]

We must introduce some continuous parameters associated to line fields that are close to $L$. Given a line field $K$ contained in the neighborhood  in the $C^1$-topology, there are continuations
\[
C_K^\pm \qquad\text{of the hyperbolic orbits }C_\pm
\]
The chosen local sections $D_\pm$ remain local sections of $C_K^\pm$ and thus there are continuations
\[
O^\pm_K \qquad\text{that are fixed points of the return maps $\on{Ret} \Gamma^\pm_K:D_\pm \to D_\pm$}
\]
Note that the orbits, return maps and fixed points all vary continuously in $K$. The first set of parameters that we will need are associated to the multipliers of these fixed points. 

\begin{definition}[Ratio/Arguments] Let $K$ be a singular line field that is sufficiently $C^1$-close to $L$. The \emph{multiplier ratio} is the quantity
\[
\theta(K) = -\frac{\log |\lambda_{cs}(O^-_K)|}{\log |\lambda_{cu}(O^+_K)|}
\]
Similarly, the \emph{multiplier arguments} $\omega_s(K)$ and $\omega_u(K)$ are given by
\[
\omega_s(K) = \on{arg}(\lambda_{cs}(O^-_K)) \qquad\text{and}\qquad \omega_u(K) = \on{arg}(\lambda_{cu}(O^+_K))
\]
where $\on{arg}(z) = b$ for a complex number $z = ae^{2\pi i b}$ where $a$ is real and positive, and $0 \le b \le 1/2$. \end{definition}

The additional parameter that we need is more complicated and requires the notion of a splitting function for embeddings of disks discussed in the following remark.

\begin{remark}[Splitting For Disks] \label{rmk:splitting_for_disks} Let $D$ denote a disk equipped with a Riemannian metric $g$ and a pair of properly embedded disks $A \subset D$ and $B \subset D$ with
\[
\on{dim}(A) + \on{dim}(B) = \on{dim}(D) - 1 \qquad\text{and}\qquad \text{$A$ intersects $B$ cleanly at one point $O \in D$}
\]
Fix small $C^1$-neighborhoods $\mathcal{U}$ and $\mathcal{V}$ of $A$ and $B$ respectively in the spaces of embedded disks in $D$. For sufficiently small choices of neighborhoods, the set $\mathcal{U} \times \mathcal{V}$ splits as a disjoint union
\[
\mathcal{U} \times \mathcal{V} = \mathcal{W}_+ \sqcup \mathcal{Y} \sqcup \mathcal{W}_-
\]
where $\mathcal{Y}$ consists of pairs of embedded hypersurfaces that intersect and $\mathcal{W}_\pm$ are the two connected components of $(\mathcal{U} \times \mathcal{V}) \setminus \mathcal{Y}$. Moreover, there is a signed refinement of the distance function
\[
\on{sdist}:\mathcal{U} \times \mathcal{V} \to \R
\]
that is negative on $\mathcal{W}_-$, positive on $\mathcal{W}_+$, and whose absolute value coincides with the distance.
\[
|\on{sdist}(A',B')| = \on{dist}_g(A',B') \qquad\text{for any $A' \in \mathcal{U}$ and $B' \in \mathcal{V}$}
\]
Note that this is unique up to an overall choice of sign, corresponding to the choice of signs on the components $\mathcal{W}_\pm$. \end{remark}

There is a corresponding function on singular line fields $K$ near $L$ that measures the failure of the line field to possess a fragile heteroclinic. Precisely, the (un)stable manifolds
\[
W^s(O^+_K) \qquad\text{and}\qquad W^u(O^-_K)
\]
vary continuously with $K$ in the $C^1$-topology. Moreover, the segments $\Gamma_\pm$ connecting $D_+$ and $D_-$ have continuations
\[
\Gamma^\pm_K \text{ with transition maps $\on{Tr}\Gamma^\pm _K$}
\]
By choosing the disks $D_+$ and $D_-$ sufficiently small in Setup \ref{set:non_degeneracy}, we may assume that
\[
A_K = W^u(O^+_K) \qquad\text{and}\qquad B_K = \on{Tr}\Gamma^-_K(W^s(O^-_K))
\]
are properly embedded disks in $D_+$ of dimension $\on{ind}(C_+)$ and codimension $\on{ind}(C_-)$ respectively. Since the heteroclinic cycle $C$ is codimension one, we therefore have
\[
\on{dim}(A_K) + \on{dim}(B_K) = \on{dim}(D_+) - 1
\]

\begin{definition}[Splitting Function] \label{def:splitting_function} Let $C$ be a heteroclinic cycle equipped with choices of local sections $D_\pm$ and heteroclinic points $P^u_\pm$ and $P^s_\pm$ as in Setup \ref{set:non_degeneracy}. The \emph{splitting function}
\[
\sigma_g:\mathcal{U} \to \R \qquad\text{for a choice of Riemannian metric $g$}
\]
is the continuous function on an open neighborhood $\mathcal{U}$ of $L$ in the space of singular line fields with the $C^1$-topology, defined by
\[
\sigma_g(K) = \on{sdist}_g(A_K,B_K)
\]
where $\on{sdist}_g$ is the signed refinement of the distance function for $g$ described in Remark \ref{rmk:splitting_for_disks}. \end{definition}

We are now ready to precisely define the notion of a proper unfolding and state the main theorem relating unfoldings to robust heteroclinic cycles \cite[Theorem B]{lt2024}.

\begin{definition}[Proper Unfolding] \label{def:proper_unfolding} Let $L$ be a singular line field with a non-degenerate heteroclinic cycle $C$ and let $U$ be an open set of $\R^m$ containing $0$. A \emph{$C^k$ proper unfolding}
\[
L_\epsilon \qquad\text{parametrized by $\epsilon \in U$ and based at $L = L_0$}
\]
is a $C^k$-continuous family of singular line fields such that there is a choice of Riemannian metric $g$ such that
\[
\sigma_g \circ L:U \to \R \qquad\text{is differentiable with $0$ as a regular value}
\]
Moreover, let $\Sigma \subset U$ be the sub-manifold of parameters given by $\Sigma = (\sigma \circ L)^{-1}(0)$. Then the family $L_\epsilon$ must satisfy the following properties depending on the multipliers.
\begin{itemize}
    \item[(SS)] If $\lambda_{cs}(O_-)$ and $\lambda_{cu}(O_+)$ are both real, then $\theta \circ L_\epsilon$ is differentiable and has nowhere zero derivative as a function of $\epsilon \in \Sigma$.
    \vspace{2pt}
    \item[(FS)] If $\lambda_{cs}(O_-)$ is not real and $\lambda_{cu}(O_+)$ is real, then the continuous functions $\theta \circ L_\epsilon$, $\omega_s \circ L_\epsilon$ and $1$ are linearly independent as functions of $\epsilon \in \Sigma$.
    \vspace{2pt}
    \item[(SF)] If $\lambda_{cs}(O_-)$ is real and  $\lambda_{cu}(O_+)$ is not real, then the continuous functions $\theta \circ L_\epsilon$, $\omega_u \circ L_\epsilon$ and 1 are linearly independent as functions of $\epsilon \in \Sigma$.
    \vspace{2pt}
    \item[(FF)] If $\lambda_{cs}(O_-)$ and  $\lambda_{cu}(O_+)$ are both not real, the the continuous functions $\theta \circ L_\epsilon$, $\omega_s \circ L_\epsilon$, $\theta\omega_u \circ L_\epsilon$ and $1$ are linearly independent as functions of $\epsilon \in \Sigma$.
\end{itemize}\end{definition}

\begin{theorem}[Unfoldings To Robust Cycles] \label{thm:unfoldings_to_robust} \cite{lt2024} Let $L_\epsilon$ be a proper unfolding parametrized by $U$ based at a non-degenerate singular line field $L$ with a heteroclinic cycle $C$ of coindex one. Then arbitrarily close to $0 \in U$, there exist parameters $\epsilon \in U$ such that $L_\epsilon$ has a robust heteroclinic cycle of index $\on{ind}(C)$.
\end{theorem}

\section{Contact Unfoldings And Robustification} \label{sec:contact_unfoldings_and_robustifications}

In this section, we show that any contact Hamiltonian manifold whose characteristic foliation contains a simple heteroclinic cycle can be ambiently perturbed to contain a robust heteroclinic cycle (Theorem \ref{thm:robustification}). More precisely, we prove that any simple heteroclinic cycle of index $(n-1,n)$ can be perturbed to be non-degenerate (Theorem \ref{thm:nondeg_approx}) and that any non-degenerate simple heteroclinic cycle can be extended to a proper unfolding (Theorem \ref{thm:characteristic_unfoldings}). These two results, combined with Theorem \ref{thm:unfoldings_to_robust} will immediately imply Theorem \ref{thm:robustification}.

\subsection{Invariant Manifolds Of Contactomorphisms} \label{subsec:invariant_manifolds_of_contactomorphisms} We start by showing that the local invariant manifolds of middle-index hyperbolic fixed points of contactomorphisms are variously Legendrian and pre-Lagrangian. Fix a contact manifold $Y$ of dimension $2n-1$ and a contact embedding
\[
\Phi:U \to Y \qquad\text{from an open set $U \subset Y$}
\]
We will require the following normal form for the differential of a contactomorphism will be repeatedly used below. The proof is an elementary computation and left as an exercise. 

\begin{lemma}[Normal Form] \label{lem:normal_form} Let $O$ be a fixed point of a contact embedding $\Phi$. Then in a standard contact Darboux chart centered at $O$, we have
\begin{equation} 
T_O\Phi = \left[\begin{array}{cc}
C & 0\\
* & \sqrt{C} \cdot B\end{array}\right] 
\end{equation}
where $C > 0$ is a constant and $B$ is a symplectic matrix acting on $\R^{2n-2}$ at $O$. \end{lemma}

We now prove the various properties of the invariant manifolds that we will need. First, the stable and unstable manifolds are isotropic under natural index assumptions.

\begin{lemma}[Isotropic Stables] \label{lem:isotropric_stable_unstable} Let $O$ be a hyperbolic fixed point of a contact embedding $\Phi$. Then
\[
W^s(O) \text{ is isotropic if $\on{ind}(O) \le n-1$} \qquad\text{and}\qquad W^u(O) \text{ is isotropic if $\on{ind}(O) \ge n$} 
\]
%Moreover, $T\Phi|_\xi$ is a hyperbolic linear map on $\xi$ at $O$ if the index $\on{ind}(O)$ is either $n-1$ or $n$ .
\end{lemma}

\begin{proof} The unstable case follows from the stable case by considering the local inverse. Therefore assume that $O$ has index $n-1$ or less. By Lemma \ref{lem:invariant_bundle}, it suffices to show that
\[
E^s(O) \subset \xi \qquad\text{at the fixed point $O$}
\]
where $E^s(O)$ is the stable eigenspaces of $T_O\Phi$ consisting of the sum of real invariant subspaces corresponding to eigenvalues of norm less than one. We may pass to a standard Darboux chart so that the differential $T_O\Phi$ the normal form in Lemma \ref{lem:normal_form}. Then the eigenvalues of $T_O\Phi$ are $C$ and $\sqrt{C} \cdot \lambda$ where $\lambda$ is an eigenvalue of $B$. Every eigenvalue of $B$ comes in a pair with its inverse. Thus if $C < 1$, then $\sqrt{C} B$ has at least $n-1$ eigenvalues of norm less than $1$ and thus $T_O\Phi$ is index at least $n$. If $C > 1$, then the stable eigenspace is contained in the subspace $\xi= \R^{2n-2}$ of $T_OY = \R^{2n-1}$ at $O$. This proves the desired result.\end{proof}

\begin{remark} Lemma \ref{lem:isotropric_stable_unstable} is also proven by Breen \cite{b2021} with a proof that is seemingly very different.
\end{remark}

\noindent A similar argument can be used to prove a lower bound on determinant, which was invoked in the proof of Lemma \ref{lem:det_bound_flows} in Section \ref{sec:liouville_and_hyperbolic_invariant_sets}.

\begin{lemma}[Determinant Bound] \label{lem:det_bound_map} Let $O$ be a hyperbolic fixed point of a contact embedding $\Phi$ with $\on{ind}(O) \le n-1$. Then
\[
\on{det}(T_O\Phi) \ge \sigma_{cu}(T_O\Phi)^n
\]
where $\sigma_{cu}(T_O\Phi)$ is the center-unstable norm, or equivalently smallest absolute value of an eigenvalue of the differential of $\Phi$ at $O$ that is greater than one.\end{lemma}

\begin{proof} We pass to a Darboux chart so that the differential $T_O\Phi$ takes the normal form in Lemma \ref{lem:normal_form}. In the normal form we can compute the determinant as
\[
\on{det}(T_O\Phi) = C \cdot \on{det}(C\Phi) = C^n
\]
Moreover, note that $C$ must be an eigenvalue which satisfies $C \neq 1$ since $T_O\Phi$ is hyperbolic. As in the proof of Lemma \ref{lem:isotropric_stable_unstable}, we have
\[
C > 1 \qquad\text{if the index satisfies}\qquad \on{ind}(O) \le n -1 
\]
Therefore we may compute that
\[
\on{det}(T_O\Phi) = C^n \ge \sigma_{cu}(M)^n = \sigma_{cu}(T_O\Phi)^n \qedhere
\]\end{proof}

A similar result holds for the complementary stable and unstable manifolds under a more restrictive index assumption. We fix the following terminology. 

\begin{definition}[Pre-Lagrangian] \label{def:pre_lagrangian} A sub-manifold $W \subset Y$ in a contact manifold $(Y,\xi)$ is \emph{pre-Lagrangian along a subset $S \subset W$} if it satisfies
\[TW \text{ is transverse to $\xi$} \qquad\text{and}\qquad TW \cap \xi\text{ is a Legendrian sub-bundle of }\xi|_S \]
A sub-manifold $W$ that is pre-Lagrangian along $W \subset W$ will simply be called \emph{pre-Lagrangian}.\end{definition}

\begin{remark} It is a simple exercise to show that if a compact sub-manifold $W$ with boundary is pre-Lagrangian, then there exists a contact form whose Reeb vector field is tangent to $W$.
\end{remark}

\begin{lemma}[Pre-Lagrangian Unstables] \label{lem:pre_lag_unstable} Let $O$ be a hyperbolic fixed point of a contact embedding $\Phi$ with index $n-1$. Then $O$ has simple center-unstable multipliers and
\[
W^u(O) \text{ is pre-Lagrangian with Legendrian foliation $\mathcal{F}^{uu}$ if $\on{ind}(O) = n-1$}\]
Similarly, if $O$ is a hyperbolic fixed point with index $n$, then $O$ has simple center-stable multipliers and
\[
W^s(O) \text{ is pre-Lagrangian with Legendrian foliation $\mathcal{F}^{ss}$ if $\on{ind}(O) = n$}
\]
\end{lemma}

\begin{proof} The stable case follows from the unstable case by considering the local inverse. Therefore assume that $O$ has index $n-1$. By Lemma \ref{lem:invariant_bundle}, it suffices to show that $O$ has simple center-unstable multipliers and that at the fixed point $O$, we have
\[
E^{uu}(O) \text{ is a subspace of $\xi$ of dimension $n-1$} \qquad\text{and}\qquad E^u(O) \text{ transverse to $\xi$}
\]
Note that $E^u(O)$ is transverse to $\xi$ at $O$ since $E^s(O) \subset \xi$ by Lemma \ref{lem:isotropric_stable_unstable} and $T_OY$ is a direct sum of $E^s(O)$ and $E^u(O)$. Thus we check that $O$ has simple center-unstable multipliers and $E^{uu}(O) \subset \xi$.

\vspace{3pt}

For these claims, we pass to a contact Darboux chart to adopt the normal form in Lemma \ref{lem:normal_form} for $T_O\Phi$ with constant $C > 0$ and symplectic matrix $B$. By Lemma \ref{lem:isotropric_stable_unstable}, $E^s(O) \subset \xi$ is a Lagrangian subspace and the constant $C$ in Lemma \ref{lem:normal_form} satisfies $C > 1$. Moreover, $E^s(O)$ must consist of generalized eigenvectors of $\sqrt{C}B$ with eigenvalue $\sqrt{C}\lambda$ for an eigenvalue $\lambda$ of $B$. In order for $\on{dim}(E^s(O)) = n-1$, it must be the case that $B$ has $n-1$ eigenvalues $\lambda$ such that $\sqrt{C}\lambda < 1$ and thus $\lambda < 1$. Since the inverse eigenvalues will have $\lambda^{-1} > 1$, we conclude that $B$ is hyperbolic and that $E^s(O) \subset \xi_O = \R^{2n-2}$ is the stable invariant subspace $V^s$ of $B$.

\vspace{3pt}

Next, note that the unstable invariant subspace $V^u$ for $B$ is spanned by generalized eigenvectors of $B$ of eigenvalue $\lambda$ with $\sqrt{C}\lambda^{-1} < 1$. It follows that, as eigenvalues of $T_O\Phi$, these eigenvectors have eigenvalue $\sqrt{C}\lambda > C$. Thus the strong stable subspace $E^{uu}(O)$ is precisely $V^u$, which is a subspace of $\xi = \R^{2n-2}$ at $O$. The dimension is the codimension of $V^s$ in $\R^{2n-2}$, which is $n-1$. This concludes the proof. 
\end{proof}

\begin{corollary}[Center Multipliers] \label{cor:real_center_multipliers} Let $O$ be a hyperbolic fixed point of a contact embedding $\Phi$. Then
\[
\text{$\lambda_{cu}(O)$ is real and $E^{cu}(O)$ is transverse to $\xi$ at $O$ if $\on{ind}(O) = n-1$}\]
\[\lambda_{cs}(O) \text{ is real and $E^{cs}(O)$ is transverse to $\xi$ at $O$ if $\on{ind}(O) = n$}
\]
\end{corollary}

\begin{proof} It suffices to prove the first case by passing to the inverse. If $\on{ind}(O) = n-1$, then the center-unstable bundle $E^{cu}(O)$ is 1-dimensional by Lemma \ref{lem:pre_lag_unstable}. It follows that $\lambda_{cu}(O)$ must be real since otherwise $E^{cu}(O)$ must be 2-dimensional.
\end{proof}

\noindent Next, the extended invariant manifolds of a hyperbolic fixed point with simple multipliers are also pre-Lagrangian in the appropriate sense.

\begin{lemma}[Pre-Lagrangian Extended Manifolds] \label{lem:pre_lag_extended} Let $O$ be a hyperbolic fixed point of a contact embedding $\Phi$. Then
\[
W^{se}(O) \text{ is pre-Lagrangian along $W^s(O)$ if $\on{ind}(O) = n-1$}\]
\[
W^{ue}(O) \text{ is pre-Lagrangian along $W^u(O)$ if $\on{ind}(O) = n$}
\]
\end{lemma}

\begin{proof} The stable case follows from the unstable case by considering the local inverse. Therefore assume that $O$ has index $n-1$. It follows from Lemma \ref{lem:pre_lag_unstable} that
\[
E^{se}(O) = E^{cu}(O) \oplus E^s(O) \not\subset \xi
\]
Therefore by Lemma \ref{lem:invariant_bundle}, the tangent space $TW^{se}(O)$ is transverse to $\xi$ along all of $W^s(O)$. On the otherhand, $TW^s(O) \subset TW^{se}(O)$ is a Lagrangian subspace of $\xi$ along $W^s(O)$ by Lemma \ref{lem:isotropric_stable_unstable}. It follows from dimension considerations that
\[
TW^{se}(O) \cap \xi = TW^s(O) \quad\text{along $W^s(O)$}
\]
Thus $W^{se}(O)$ is pre-Lagrangian along $W^s(O)$.\end{proof}

\subsection{Contact Preliminaries} Here we record several results in contact geometry that will be useful in the constructions of unfoldings in the following sections. 

\vspace{3pt}

We first need the following result, stating that intersections between Legendrians and pre-Lagrangians can be made transverse by an arbitrarily small perturbation.

\begin{lemma}[Transversality Contactomorphism] \label{lem:transversality_lemma} Let $L$ and $M$ be sub-manifolds of the contact manifold $Y$ intersecting at a point $P$ and let $U$ be a neighborhood of $P$. Suppose that
\[
\text{$T_PL$ is transverse to $\xi$ at $P$} \qquad\text{and}\qquad \text{$T_PL \cap \xi$ and $T_PM \cap \xi$ are Lagrangian subspaces of $\xi$ at $P$}
\]
Then there is an arbitrarily $C^\infty$-small Hamiltonian $H:[0,1] \times U \to \R$ supported in $(0,1) \times U$ such that $\Phi_H(L)$ intersects $M$ transversely at $P$. \end{lemma}

\begin{proof} By passing to a Darboux chart centered at $P$ and shrinking $U$, we can assume that $U = D_{\on{std}}$ is the standard contact disk and $P$ is the origin in $D_{\on{std}} \subset \R^{2n-1}$. We consider the subspaces
\[
V = T_PL \cap \R^{2n-2} \qquad\text{and}\qquad W = T_PM \cap \R^{2n-2}\qquad\text{in the contact structure $\xi_P = \R^{2n-2}$ at $P$}
\]
These are Lagrangian subspaces of $\R^{2n-2}$ and thus we can find a path 1-parameter subgroup of linear symplectomorphisms $A:\R \to \on{Sp}(2n)$ such that $A_s(V)$ is transverse to $W$ for all sufficiently small $s > 0$. We can then take the family of lifted contactomorphisms
\[
\Phi_s = \widetilde{A}_s:\R^{2n-1} \to \R^{2n-1} \qquad\text{as in Definition \ref{def:lifted_contactomorphisms}}
\]
By Lemma \ref{lem:CSp_differential}, we know that $\Phi_s$ fixes $P$ and the differential $T\Phi_s$ of $\Phi_s$ at $P$ is given by $1 \oplus A_s$ for all $s$. It follows that $T\Phi_s(V)$ is transverse to $W$ as subspaces of $\R^{2n-2} = \xi_P$. Since $L$ is transverse to $\xi$ at $P$, this furthermore implies that $T\Phi_s(T_PL)$ is transverse to $T_PM$ for all $s$. 

\vspace{3pt}

Now we simply note that $\Phi_s$ is a 1-parameter subgroup of contactomorphisms and is therefore generated by an autonomous contact Hamiltonian $G$. We define $H$ to be the cutoff of $sG$ by a bump function supported in $U$ that is $1$ near $P$. Then $\Phi_H = \Phi_s$ near $P$, so that $\Phi_H(L)$ intersects $M$ at $P$ transversely. By taking $s$ close to $0$, we can ensure that $H$ is arbitrarily $C^\infty$-small. By reparametrizing time, we can assume that $H$ is supported in the interval $(0,1)$ in time. \end{proof}

\begin{remark} \label{rmk:reverse_transversality} By taking the inverse contactomorphism, we can also find an arbitrarily $C^\infty$-small contactomorphism $H$ satisfying the hypotheses of Lemma \ref{lem:transversality_lemma} so that $\Phi_H(M)$ is transverse to $L$.
\end{remark}

Next, we require the following lemma stating that transverse curves can be stretched near a given point using contactomorphisms.

\begin{lemma}[Stretching Contactomorphism] \label{lem:transverse_stretching} Let $I$ be an embedded 1-manifold in a contact manifold $(Y,\xi)$ that is transverse to $\xi$ at a point $P$ in $I$ and let $U$ be a neighborhood of $P$. Then there is an arbitrarily $C^\infty$-small Hamiltonian $H:[0,1] \times U \to \R$ supported in $(0,1) \times U$ such that $\Phi_H$ restricts to a diffeomorphism $\psi:I \to I$ supported in $U \cap I$ such that
\[\psi(P) = P \qquad\text{and}\qquad T_P\psi \neq \pm \on{Id}\]
\end{lemma}

\begin{proof} By the standard neighborhood theorem for transverse sub-manifolds, we can pass to a Darboux chart to assume that
\[
Y = \R^{2n-1} = \R_z \times T^*\R^{n-1} \qquad I = [-1,1]_z \times 0 \subset \R^{2n-1} \qquad P = 0 \in [-1,1]
\]
Here $\R^{2n-1}$ is equipped with the standard contact form $\alpha_{\on{std}} = dz + \lambda_{\on{std}}$.  Let $G$ be the autonomous contact Hamiltonian $G = f(z)$ where $f$ vanishes outside of $U \cap I$ and $f(z) = z$ near $0$. Then the contact Hamiltonian vector field along $\R_z \times 0$ is precisely $f(z)\partial_z$ and the flow $\Phi^G_t$ of $G$ restricts to a flow $\phi_t$ on $I$ such that
\[
\phi_t(0) = 0 \qquad\text{and}\qquad T_0\phi_t = e^t \on{Id} \qquad\text{for all $t$}
\]
We can now take $H$ to be the cutoff of $sG$ by a cutoff function that is $1$ near $I$ and that vanishes outside of $U$. Then $H$ can be made arbitrarily $C^\infty$-small by taking $s$ small enough, and the restriction $\psi$ of $\Phi_H$ to $I$ will have the desired properties.\end{proof}

Finally, we will require a result from metric contact geometry. We formulate this result via the following definition.

\begin{definition}[Weakly Compatible] \label{def:weakly_compatible} A Riemannian metric $g$ on a contact manifold $(Y,\xi)$ with contact form $\alpha$ is \emph{weakly compatible} if the Reeb vector field $R$ satisfies
\[
g(R,R) = 1 \qquad \text{and}\qquad \text{$\xi$ is orthogonal to $R$} 
\]
\end{definition}

\noindent It is a standard fact that the flowlines of the Reeb vector field $R$ are geodesics with respect to any weakly compatible metric (cf. \cite{becker2023geodesible}). We will need the following result.

\begin{proposition}[Legendrian Distance] \label{prop:legendrian_dist} Let $g$ be a weakly Riemannian metric on a compact contact manifold with boundary $(Y,\xi)$ with contact form $\alpha$. Fix compact Legendrian sub-manifolds $\Lambda$ and $K$ intersecting at a single point
\begin{equation} \label{eq:leg_dist_transversality}
P = \Lambda \cap K \qquad\text{satisfying}\qquad \xi_P = T_P\Lambda \oplus T_PK
\end{equation}
Then there is a constant $\epsilon > 0$, a neighborhood $B$ of $P$, and $C^\infty$-neighborhoods $\mathcal{U}$ of $\Lambda$ and $\mathcal{V}$ of $K$ in the space of Legendrian embeddings such that
\begin{itemize}
    \item[(a)] Any pair of Legendrian $\Lambda'$ in $\mathcal{U}$ and $K'$ in $\mathcal{V}$ are connected by a unique Reeb chord 
    \[
    \Gamma(\Lambda',K') \subset B \qquad\text{of length}\qquad  T(\Lambda',K') < \epsilon
    \]
    \item[(b)] The distance $\on{dist}_g(\Lambda',K')$ between any pair of Legendrians $\Lambda'$ in $\mathcal{U}$ and $K'$ in $\mathcal{V}$ is given by
    \[
    \on{dist}_g(\Lambda',K') = T(\Lambda',K') 
    \]
\end{itemize}
Moreover, the Reeb chord $\Gamma(\Lambda',K')$ and its length $T(\Lambda',K')$ depend smoothly on $\Lambda'$ and $K'$. \end{proposition}

\begin{remark} Note that, for a weakly compatible metric, the length of a Reeb chord with respect to the parametrization induced by the Reeb vector field is equivalent to the metric length.
\end{remark}

\begin{proof} Let $R$ denote the Reeb vector field of the contact form $\alpha$ and let $\Psi:\R \times Y \to Y$ denote the Reeb flow. Also let $\iota:\Lambda \to Y$ and $\jmath:\Lambda \to K$ denote the inclusions of $\Lambda$ and $K$. Recall that Reeb chords are in bijection with points
\[
(T,Q) \in (\Psi \circ \iota)^{-1}(\jmath(K))
\]
The hypothesis (\ref{eq:leg_dist_transversality}) implies that $\Gamma(\Lambda,K) = (0,P)$ is a transverse intersection point. This implies that there are $C^\infty$-neighborhoods $\mathcal{U}$ of $\Lambda$, $\mathcal{V}$ of $K$, and $[-\epsilon,\epsilon] \times B$ of $(0,P)$ such that every inclusion $\iota' \in \mathcal{U}$ and $\jmath' \in \mathcal{V}$ has a unique transverse chord
\[
\Gamma(\Lambda',K') = (T(\Lambda',K'),P(\Lambda',K')) \in [-\epsilon,\epsilon] \times U
\]
that depends smoothly on the sub-manifolds $\Lambda$ and $K$. This is simply a rephrasing of (a). 

\vspace{3pt}

Next, let $\sigma:Y \times Y \to \R_+$ denote the square distance function $\sigma(x,y) = \on{dist}_g(x,y)^2$. Recall that the square distance function is smooth on a neighborhood of the diagonal $\Delta \subset Y \times Y$ and the Hessian along the diagonal is given by
\begin{equation} \label{eq:Hessian_of_square_dist}
\nabla^2 \sigma_{(x,x)}(u \oplus v, u \oplus v) = 2g_x(u-v,u-v)
\end{equation}
By choosing $B$ to be geodesically convex with diamater smaller than the injectivity radius of $g$, we can guarantee that $\sigma|_{B \times B}$ is differentiable and that any two points in $B$ are connected by a unique (and length minimizing) geodesic contained in $B$. Note that (\ref{eq:Hessian_of_square_dist}) implies that the restriction $\sigma|_{\Lambda \times K}$ has a non-degenerate minimum at $(P,P)$ since
\[
T_P\Lambda \cap T_PK = 0
\]Moreover, the restriction $\sigma|_{\Lambda \times K}$ varies smoothly with $\Lambda$ and $K$ since $\sigma$ is smooth. It follows that, by choosing the neighborhoods $B$, $\mathcal{U}$ and $\mathcal{V}$ sufficiently small, we can guarantee that the function
\[
\sigma|_{\Lambda' \times K'} \qquad\text{for any $\Lambda' \in \mathcal{U}$ and $K' \in \mathcal{V'}$}
\]
has a unique critical point $(P',Q')$ contained in the neighborhood $(B \cap \Lambda') \times (B \cap K')$ such that $P'$ and $Q'$ are connected by a unique geodesic in $B$ (which is length minimizing) and such that
\[\sigma(P',Q') = \on{dist}_g(\Lambda',K')^2\]
On the other hand, since $\Gamma(\Lambda',K') \subset B$ is contained in $B$, it is the unique length minimizing geodesic connecting the endpoints of $\Gamma(\Lambda',K')$. Moreover, since the endpoints of $\Gamma(\Lambda'K')$ are orthogonal to $\Lambda'$ and $K'$ at the endpoints, the endpoints are critical points of $\sigma|_{\Lambda' \times K'}$ and therefore must be $(P',Q')$. It follows that the distance between the endpoints is $\on{dist}_g(\Lambda',K')$ and that $\on{dist}_g(\Lambda',K') = T(\Lambda',K')$. This proves (b).\end{proof}

%\begin{lemma}[Legendrian-Lagrangian Standard Form] Let $(Y,\xi)$ be a contact manifold and consider a smooth family of embeddings\[\iota_s:L \to Y \qquad\text{and}\qquad \jmath_s:\Lambda \to Y \qquad\text{parametrized by an open set $U \subset \R^k$ containing $0$}\]Suppose that $\iota_s(L)$ is pre-Lagrangian for all $s \in U$, $\jmath_s(\Lambda)$ is Legendrian for all $s \in U$ and $\iota_0$ is transverse to $\jmath_0$ at a point $P$. Then there family of contactomorphisms\[\Phi_s:Y \to Y \qquad \text{parametrized a neighborhood $V \subset U$ of $0$}\]such that $\Phi_0 = \on{Id}$, $\Phi_s \circ \iota_0 = \iota_s$ and $\Phi_s \circ \jmath_0 = \jmath_s$ for $s \in V$ \end{lemma}

\subsection{Non-Degenerate Approximation} \label{subsec:nondegenerate_approximation} In this part, we prove that any contact Hamiltonian manifold with a heterodimensional cycle of the appropriate index can be ambiently perturbed so that the heterodimensional cycle is non-degenerate in the sense of Definition \ref{def:non_degeneracy}.

\begin{theorem}[Non-Degenerate Approximation] \label{thm:nondeg_approx} Let $(\Sigma,\eta)$ be a contact Hamiltonian manifold whose characteristic foliation contains a heterodimensional cycle
\[
C = (C_+,C_-) \qquad\text{with index $(n-1,n)$}
\]Then there is an ambient deformation $\iota$ of $(\Sigma,\eta)$, arbitrarily small in the $C^\infty$-topology, such that the continuations $C'_\pm$ of $C_\pm$ in the characteristic foliation of $\eta_\iota$ form a non-degenerate heteroclinic cycle $C'$. 
\end{theorem}

\begin{proof} The heterodimensional cycle $C$ in Theorem \ref{thm:nondeg_approx} is coindex one by hypothesis and $C$ automatically has simple center multipliers in the sense of Definition \ref{def:simple_center_multipliers_for_cycle} by Lemma \ref{lem:pre_lag_unstable}. Therefore, it is only necessary show that we can perform a $C^\infty$-small isotopy to achieve Conditions \ref{cond:simple_fragile_heteroclinic}-\ref{cond:focus}. Each non-degeneracy condition is open in the $C^1$-topology (cf. the discussion in \cite[\S 2.2]{lt2024}) and thus we may show that a perturbation exists that achieves each condition separately. 

\vspace{3pt}

Lemmas \ref{lem:fragile_het_condition} and \ref{lem:robust_het_condition} below establish this for Conditions \ref{cond:simple_fragile_heteroclinic} and \ref{cond:simple_robust_heteroclinic} respectively. Condition \ref{cond:extra_condition} can be achieved without perturbation by changing the heteroclinic segments as in Remark \ref{rmk:third_condition}. Finally, Corollary \ref{cor:real_center_multipliers} implies that the center-stable and center-unstable multipliers $\lambda_s$ and $\lambda_u$ are both real under our index assumptions on $C$. Thus the final condition that must be achieved is Condition \ref{cond:saddle}, and this is addressed in Lemma \ref{lem:saddle_saddle_condition} below.\end{proof}

In the remainder of this part, we work towards the proofs of Lemmas \ref{lem:fragile_het_condition}, \ref{lem:robust_het_condition} and \ref{lem:saddle_saddle_condition}. Fix a contact Hamiltonian manifold $(\Sigma,\eta)$ and a heterodimensional cycle $C$ as in Theorem \ref{thm:nondeg_approx}. Also fix local sections and heteroclinic points as in Setup \ref{set:non_degeneracy}.
\[
D_\pm \text{ as in (\ref{eq:choice_of_local_sections})} \qquad\text{and}\qquad  P^u_\pm \text{ and }P^s_\pm \text{ as in (\ref{eq:choice_of_heteroclinic_points})}
\]
Let $\Gamma_+$ and $\Gamma_-$ be the heteroclinic segments connecting $P^u_+$ to $P^s_-$ and $P^u_-$ to $P^s_+$, respectively. As in Setup \ref{set:non_degeneracy} and specifically (\ref{eq:choice_of_heteroclinic_segments_property}), we will assume that these heteroclinic segments are disjoint from the partial sections $D_+$ and $D_-$ away from their endpoints. We consider the corresponding transition maps
\[
\on{Tr}\Gamma_+:D_+ \to D_- \qquad\text{and}\qquad \on{Tr}\Gamma_-:D_- \to D_+
\]
Given an sufficiently $C^\infty$-small ambient deformation $\iota$ of $\eta$, the characteristic foliation of the deformed contact Hamiltonian structure $\eta_\iota$ is $C^\infty$-close to the characteristic foliation of $\eta$. In particular, it inherits continuations
\[
C^\iota_\pm \text{ of the hyperbolic orbits $C_\pm$}  \qquad\text{and}\qquad \Gamma^\iota_\pm \text{ of the heteroclinic segments $\Gamma_\pm$}
\]
In order to achieve the desired non-degeneracy conditions, we will plug the characteristic foliation along the heteroclinic segments $\Gamma_+$ and $\Gamma_-$. In particular, we will repeatedly use the following variant of Lemma \ref{lem:transition_map_deformation} from Section \ref{subsec:transition_and_return_maps}.

\begin{lemma}[Transition Deformation For Heteroclinics] \label{lem:transition_deformation_for_hets} There is an $\epsilon$ and a neighborhood $W$ of $P^s_-$ in $D_-$ such that, for any contact Hamiltonian
\[
H:[0,1] \times D_- \to \R \qquad\text{with}\qquad |H| \le \epsilon \quad\text{and}\quad \on{supp}(H) \subset (0,1) \times W 
\]
there is an ambient deformation $\iota = \iota_H$ of $\eta$ such that the return maps of the continuations $C^\iota_\pm$ and the transition maps of the continuations $\Gamma^\iota_\pm$ satisfy
\[
\on{Tr}\Gamma^\iota_+ = \Phi_H \circ \on{Tr}\Gamma_+ \qquad \on{Tr}\Gamma^\iota_- = \on{Tr}\Gamma_- \qquad \on{Ret} C^\iota_\pm = \on{Ret} C_\pm
\]
Moreover, the map sending $H$ to the ambient deformation $\iota_H$ is $C^\infty$-continuous. The analogous results for Hamiltonians on $D_+$ and the transition map of $\Gamma_-$ also holds.\end{lemma}

\begin{proof} Choose a plugging domain $U = [0,1] \times D_{\on{std}}$ embedded in $\Sigma$ that is well-positioned along $\Gamma_+$ (see Definition \ref{def:well_positioned}). Then we can factor the transition map of $\Gamma$ as a composition
\[
\on{Tr}\Gamma = \on{Tr}\Gamma_{\on{out}} \circ \on{Tr}\Gamma_{\on{in}}
\]of the segment $\Gamma_{\text{in}}$ of $\Gamma_+$ from $D_+$ to the inward boundary $D_{\on{in}} = 0 \times D_{\on{std}}$ of $U$ and the segment $\Gamma_{\text{out}}$ of $\Gamma_+$ from the outward boundary $D_{\on{out}} = 1 \times D_{\on{std}}$ of $U$ to $D_-$. Let $Q$ be the point $\on{Tr}\Gamma_{\on{in}}(P^u_+)$ so that $ \on{Tr}\Gamma_{\on{out}}(Q) = P^s_-$. Note that the transition map $\on{Tr}\Gamma_{\on{out}}$ is a contactomorphism from a small neighborhood of $Q$ in $D_{\on{std}}$ to a small neighborhood of $P^s_-$ in $D_-$. Therefore, there is a neighborhood $U$ of $P^s_-$ in $D_-$ so that, for any contact Hamiltonian $H$ supported in $[0,1] \times U$, we can write
\[
\Phi_H \circ \on{Tr}\Gamma_{\on{out}} = \on{Tr}\Gamma_{\on{in}} \circ \Phi_{H'} \qquad\text{for a contact Hamiltonian $H':[0,1] \times D_{\on{std}} \to \R$}
\]
Note that the map $H \mapsto H'$ is $C^k$-continuous for every $k$ and $H'$ is supported in the open set $U' = \on{Tr}\Gamma_{\on{out}}^{-1}(U)$. Let the ambient isotopy $\iota$ be the plugging $\iota_{U,H'}$ of $\eta$ along the plugging domain $U$ by the contact Hamiltonian $H'$ (Definition \ref{def:plug_insection} and Remark \ref{rmk:insertion_along_domain_of_H}). Then Lemma \ref{lem:transition_map_deformation} states that
\[
\on{Tr} \Gamma^\iota_+ = \on{Tr}\Gamma_{\on{out}} \circ \Phi_{H'} \circ \on{Tr}\Gamma_{\on{in}} = \Phi_H \circ \on{Tr}\Gamma_{\on{out}} \circ \on{Tr}\Gamma_{\on{in}} = \Phi_H \circ \on{Tr}\Gamma_+
\]
Since the map sending $H'$ to the ambient deformation $\iota_{U,H'}$ is $C^k$-continuous for each $k$, the map $H \mapsto \iota_{U,H'}$ is also $C^k$-continuous for all $k$. 

\vspace{3pt}

Finally, the deformation $\iota = \iota_{U,H'}$ is supported in the plugging domain $U$, which can be chosen to lie within an arbitrarily small neighborhood of the point $Q \in \Gamma_+$. Thus, by choosing $U$ sufficiently small, we can guarantee that the characteristic foliation is unchanged along the flow segment $\Xi$ connecting a point $P \in D_-$ in the domain of $\on{Tr}\Gamma_-$ to its image $\on{Tr}\Gamma_-(P)$ in $D_+$, since the union of such segments is disjoint from $\Gamma_+$. This implies that $\Gamma^\iota_- = \Gamma_-$ and $\on{Tr}\Gamma^\iota_- = \on{Tr}\Gamma_-$. A similar argument applies the return maps of $C_+$ and $C_-$.  \end{proof}

We can now present the proofs of Lemmas \ref{lem:fragile_het_condition} and \ref{lem:robust_het_condition} using Lemma \ref{lem:transversality_lemma} and Lemma \ref{lem:transition_deformation_for_hets}.

\begin{lemma}[Fragile Heteroclinic] \label{lem:fragile_het_condition} There is an arbitrarily $C^\infty$-small ambient deformation $\iota$ of $\eta$ such that $C$ has a continuation $C^\iota$ in the characteristic foliation of $\eta_\iota$ satisfying Condition \ref{cond:simple_fragile_heteroclinic}.
\end{lemma} 

\begin{proof} We adopt the following simplified notation for the sub-manifolds of $D_+$ in Condition \ref{cond:simple_fragile_heteroclinic}. 
\[
L_- = \on{Tr}\Gamma_-(W^u(O_-)) \qquad M_+ = W^{se}(O_+) \qquad M_- = \on{Tr}\Gamma_-(W^{ue}(O_-)) \qquad L_+ = W^s(O_+)
\]
We also denote the point $P^s_+$ in $D_+$ by $P$. Note that $P$ is contained in $L_+ \cap M_+$ and $L_- \cap M_-$. Note that $L_+$ is the stable manifold of the index $n-1$ hyperbolic fixed point $O_+$ of the contactomorphism $\on{Ret}C_+$, and $L_-$ is the image under a contactomorphism $\on{Tr}\Gamma_-$ of the unstable manifold of the index $n$ hyperbolic fixed point $O_-$ of $\on{Ret}C_-$. It follows from Lemma \ref{lem:isotropric_stable_unstable} that $L_+$ and $L_-$ are Legendrians in $D_+$. Similarly, Lemma \ref{lem:pre_lag_extended} implies that $M_+$ and $M_-$ are pre-Lagrangian along the intersections
\[
M_+ \cap W^s(O_+) \qquad\text{and}\qquad M_- \cap \on{Tr}\Gamma_-(W^u(O_-))
\]
Moreover, $P = P^s_+$ is contained in the intersection of $W^s(O_+)$ and $\on{Tr}\Gamma_-(W^u(O_-))$. It follows that $M_+$ and $M_-$ are pre-Lagrangian at $P$ in the sense of Definition \ref{def:pre_lagrangian}.

\vspace{3pt}

Lemma \ref{lem:transversality_lemma} and Remark \ref{rmk:reverse_transversality} now imply that there is an arbitrarily $C^\infty$-small contact Hamiltonian $H$ supported in an arbitrarily small neighborhood of $P$ such that
\[
\Phi_H(L_+) \text{ is transverse to $M_+$ at $P$}\qquad\text{and}\qquad \Phi_H(M_-) \text{ is transverse to $L_-$ at $P$}
\]
By Lemma \ref{lem:transition_deformation_for_hets}, there is a corresponding $C^\infty$-small ambient deformation $\iota$ of $\Sigma$ such that the continuation $\Gamma^\iota_+$ of $\Gamma_+$ has transition map $\on{Tr}\Gamma^\iota_- = \Phi_H \circ \on{Tr}\Gamma_-$, and such that the return maps of $C_\pm$ and the transition map of $\Gamma_+$ are unchanged. In particular, the (un)stable and extended (un)stable invariant manifolds are the same and
\[
 \text{$\on{Tr}\Gamma^\iota_+(W^u(O_-)) = \Phi_H(L_+)$ is transverse to $W^{se}(O_+) = M_+$ at $P^s_+ = P$}
\]
\[
 \text{$\on{Tr}\Gamma'_+(W^{ue}(O_-)) = \Phi_H(M_-)$ is transverse to $W^{se}(O_+) = L_-$ at $P^s_+ = P$}
\]
This is precisely Condition \ref{cond:simple_fragile_heteroclinic} for the deformed structure $\eta_\iota$ and this proves the lemma.
\end{proof}

\begin{lemma}[Robust Heteroclinic] \label{lem:robust_het_condition}  There is an arbitrarily $C^\infty$-small ambient deformation $\iota$ of $\eta$ such that $C$ has a continuation $C^\iota$ in the characteristic foliation of $\eta_\iota$ satisfying Condition \ref{cond:simple_robust_heteroclinic}.
\end{lemma}

\begin{proof} We again adopt simplified notation. Consider the strong-stable foliation $\mathcal{F}^{ss}$ on $W^s(O_-)$ and the strong-unstable foliation $\mathcal{F}^{uu}$ on $W^u(O_+)$. Let $\Lambda$ be the leaf of $\on{Tr}\Gamma_+(\mathcal{F}^{uu})$ passing through $P^s_-$ and $\Xi$ be the leaf of $\mathcal{F}^{ss}$ passing through $P^s_-$. Finally, let
\[
L = \on{Tr}\Gamma_+(W^u(O_+))\qquad M = W^s(O_-) \qquad Q = P^s_-
\]
Note that $M$ is the stable manifold of an index $n$ hyperbolic fixed point and $L$ is the image under a contactomorphism of the unstable manifold of an index $n-1$ hyperbolic fixed point. Thus $L$ and $M$ are pre-Lagrangian by Lemma \ref{lem:pre_lag_unstable}. Moreover, $\Lambda$ and $\Xi$ are Legendrian by Lemma \ref{lem:pre_lag_unstable}, since they are (images of) leaves of the strong stable and unstable foliations on $M$ and $L$. 

\vspace{3pt}

Lemma \ref{lem:transversality_lemma} and Remark \ref{rmk:reverse_transversality} now imply that there is an arbitrarily $C^\infty$-small contact Hamiltonian $H$ supported in an arbitrarily small neighborhood of $P$ such that
\begin{equation} \label{eq:condition_2_transversality}
\text{$\Phi_H(L)$ and $\Phi_H(\Lambda)$ are transverse to $M$ at $P$}\qquad \text{$\Phi_H(L)$ is transverse to $\Xi$ at $P$}
\end{equation}
As in Lemma \ref{lem:fragile_het_condition}, we can now invoke Lemma \ref{lem:transition_deformation_for_hets} to acquire a corresponding $C^\infty$-small ambient isotopy $\iota$ of $\Sigma$ such that the continuation $\Gamma^\iota_+$ has transition map $\Phi_H \circ \on{Tr}\Gamma_+$. An analogous argument to Lemma \ref{lem:fragile_het_condition} then shows that the transversality properties in (\ref{eq:condition_2_transversality}) are equivalent to Condition \ref{cond:simple_robust_heteroclinic} for the deformed characteristic foliation. 
\end{proof}

\begin{lemma}[Saddle-Saddle] \label{lem:saddle_saddle_condition} There is an arbitrarily $C^\infty$-small ambient deformation $\iota$ of $\eta$ such that $C$ has a continuation $C^\iota$ in the characteristic foliation of $\eta_\iota$ satisfying Condition \ref{cond:saddle}. 
\end{lemma}

\begin{proof} By Lemma \ref{lem:fragile_het_condition} and \ref{lem:robust_het_condition}, we may assume that Condition \ref{cond:simple_fragile_heteroclinic} and \ref{cond:simple_robust_heteroclinic}. As in the formulation of Condition \ref{cond:saddle}, we consider the curves of intersection between the stable and unstable manifolds defined as follows.
\[
I_- = \on{Tr}\Gamma_+\big(W^u(O_+)) \cap W^s(O_-) \subset D_- \qquad\text{and}\qquad I_+ = W^u(O_+) \cap \on{Tr}\Gamma_+^{-1}\big(W^s(O_-)\big) \subset D_+
\]
Fix linear coordinates $(x_s,x_{cu},x_{uu})$ and $(y_{ss},y_{cs},y_u)$ as in Theorem \ref{thm:linear_coordinates} and (\ref{eq:linear_coordinates_for_condition}). Also fix a parametrization $\psi$ of $I_-$ by $[-1,1]_r$ as in (\ref{eq:parametrization_of_I}). We consider the functions on $[-1,1]_r$ given by
\[
u_{cu} = x_{cu} \circ \psi \qquad\text{and}\qquad v_{cs} = y_{cs} \circ \on{Tr}\Gamma_+ \circ \psi
\]
Note that by construction, $\psi(0) = P^u_+$ and $\on{Tr}\Gamma_+(P^u_+) = P^s_-$. Suppose Condition \ref{cond:saddle} is not satisified, since we are finished otherwise. Then
\[
\big|\frac{\partial_rv_{cs}(0) \cdot u_{cu}(0)}{\partial_r u_{cu}(0) \cdot v_{cs}(0)}\big| = \big|\frac{\partial_r\log |v_{cs}|}{\partial_r\log |u_{cu}|}\big| = 1
\]
Note that the interval $I_-$ is transverse to the foliation $\mathcal{F}^{ss}$ in $W^s(O_-)$. Since $W^s(O_-)$ is pre-Lagrangian by Lemma \ref{lem:pre_lag_unstable}, this implies that $I_-$ is transverse to the contact structure of $D_-$. Thus, by Lemma \ref{lem:transverse_stretching}, we can construct an arbitrarily $C^\infty$-small contact Hamiltonian $H$ supported in an arbitrarily small neighborhood of $P^s_-$ such that $\Phi_H$ restricts to a diffeomorphsm
\[
\phi:I_- \to I_- \qquad\text{such that}\qquad \Phi_H(P^s_-) = P^s_- \qquad\text{and}\qquad T\phi \neq \pm \on{Id} \text{ at $P^s_-$}
\]
As in Lemmas \ref{lem:fragile_het_condition} and \ref{lem:robust_het_condition}, we can apply Lemma \ref{lem:transition_deformation_for_hets} to find a corresponding $C^\infty$-small ambient isotopy $\iota$ of $\Sigma$ so that the transition map of the continuation $\Gamma^\iota_+$ satisfies
\[
\on{Tr}\Gamma^\iota_+ = \Phi_H \circ \on{Tr} \Gamma_+ \qquad \on{Tr}\Gamma^\iota_- = \on{Tr}\Gamma_- \qquad \on{Ret}\Gamma^\iota_\pm = \on{Ret}\Gamma_\pm
\]
Let $I^\iota_-$ and $I^\iota_+$ denote the intervals of intersection corresponding to this characteristic foliation of $\eta_\iota$ and let $u^\iota_{cu}$ and $v^\iota_{cs}$ denote the corresponding functions on $[-1,1]_r$. Since $\Phi_H$ fixes $I_-$, it is simple to decud
\[
I_- = \on{Tr}\Gamma_+(W^u(O_+)) \cap W^s(O_-) \subseteq \Phi_H \circ \on{Tr}\Gamma_+(W^u(O_+)) \cap W^s(O_-) = I^\iota_-
\]
On the other hand, the intersections above are transverse and $\Phi_H$ is $C^\infty$-small, it follows that we must have $I^\iota_- = I_-$. A similar argument shows that $I^\iota_+ = I_+$. If we adopt $\psi$ as the parametrization of both $I_-$ and $I^\iota_-$, then
\[
u^\iota_{cu} = u_{cu} \qquad \text{and}\qquad v_{cs}^\iota = y_{cs} \circ \Phi_H \circ \on{Tr}\Gamma_+ \circ \psi = y_{cs} \circ \phi \circ \on{Tr}\Gamma_+ \circ \psi
\]
Note that $v_{cs}^\iota(0) = v_{cs}(0)$ since $\Phi_H \circ \on{Tr}\Gamma_+ \circ \psi(0) = \phi(P^s_+) = P^s_+$. However, $\partial_r v_{cs}^\iota(0) \neq \pm \partial_r v_{cs}(0)$ since $T\phi \neq \pm \on{Id}$ at $P^s_+$. It follows that
\[
\big|\frac{\partial_rv^\iota_{cs}(0) \cdot u_{cu}^\iota(0)}{\partial_r u_{cu}^\iota(0) \cdot v_{cs}^\iota(0)}\big| = \big|\frac{\partial_r v^\iota_{cs}(0)}{\partial_r v_{cs}(0)}\big| \cdot \big|\frac{\partial_rv_{cs}(0) \cdot u_{cu}(0)}{\partial_r u_{cu}(0) \cdot v_{cs}(0)}\big| \neq 1
\]
This is precisely Condition \ref{cond:saddle} for the characteristic foliation of $\eta_\iota$. This concludes the lemma.\end{proof}

\subsection{Characteristic Unfoldings} We conclude this section by proving the following existence theorem for unfoldings of characteristic foliations. This will finish the proof of Theorem \ref{thm:robustification}.

\begin{theorem}[Characterestic Unfoldings] \label{thm:characteristic_unfoldings} Let $(\Sigma,\eta)$ be a contact Hamiltonian manifold whose characteristic foliation contains a non-degenerate heterodimensional cycle $C$ with index $(n-1,n)$. Then there is a 2-parameter family of embeddings from $\Sigma$ into the contactization
\[
\iota_\epsilon:\Sigma \to C\Sigma \qquad\text{parametrized by an open set $U \subset \R^2$ containing $0$}
\]
such that $\iota_0$ is the inclusion $\Sigma = 0 \times \Sigma \to \Sigma$ and such that the characteristic foliations of the contact Hamiltonian structures $\eta_\epsilon = \iota^*_\epsilon\xi$ for $\epsilon \in U$ form a proper unfolding in the sense of Definition \ref{def:proper_unfolding}. \end{theorem}

We will deduce Theorem \ref{thm:characteristic_unfoldings}, via a plugging construction, from the following proposition on the existence of certain contact isotopies of contactomorphisms and Legendrians.

\begin{proposition} \label{prop:unfolding_proposition} Let $\Phi:U \to D$ be a contact embedding to a contact disk $D$ from an open set $U \subset D$ with a hyperbolic fixed point $O$ of index $n$. Let $\Lambda \subset U$ be a Legendrian with an intesection point
\[
P \in \Lambda \cap W^u(O) \qquad\text{such that}\qquad \xi_P = T_P\Lambda \oplus T_PW^u(O)
\]
Then for any neighborhoods $V$ of $O$ and $W$ of $P$, there is an $\epsilon > 0$ and two families of contact Hamiltonians
\[G_r:[0,1] \times D \to \R \qquad\text{and}\qquad H_{r,s}:[0,1] \times D \to \R\]
parametrized by $r \in [-\epsilon,\epsilon]$ and $(r,s) \in [-\epsilon,\epsilon]^2$ respectively, that satisfy the following properties.
\begin{itemize}
    \item[(a)] (Support) $G_r$ and $H_{r,s}$ are supported on $(0,1) \times V$ and $(0,1) \times W$, respectively, and
    \[G_0 = H_0 = 0\]
    \item[(b)] (Hyperbolicity) For each $r \in [-1,1]$, the point $O$ is a hyperbolic fixed point of the embedding
    \[
    \Phi_r:U \to D \qquad\text{given by} \qquad \Phi_{G_r} \circ \Phi
    \]
    \item[(c)] (Multiplier) Let $\lambda_{cs}(\Phi_r,O)$ denote the center-stable multiplier of $O$ as a fixed point of $\Phi_r$. Then
    \[
    \lambda_{cs}(\Phi_r,O) = e^r \cdot \lambda_{cs}(\Phi_0,O)
    \]
    \item[(d)] (Signed Distance) Let $\Lambda_{r,s} = \Phi_{H_{r,s}}(\Lambda)$. Then there is a closed ball $B \subset D$ centered at $P$ and a Riemannian metric $g$ on $D'$ such that
    \[
    \on{sdist}_g(\Lambda_{r,s} \cap B,W^u(\Phi_r,O) \cap B) = s
    \]\end{itemize}\end{proposition}

\begin{proof} We break the proof into two steps, namely the construction of $G$ and the construction of $H$. We verify the desired properties (a-d) as we go.

\vspace{3pt}

\noindent {\bf Step 1: Construction Of $G$.} We start by constructing $G$. Consider the splitting of $T_OD$ into strong stable, center stable and unstable spaces
\[
T_OD = E^s(O) \oplus E^u(O) = E^{ss}(O) \oplus E^{cs}(O) \oplus E^{u}(O)
\]
By Lemma \ref{lem:isotropric_stable_unstable} and Lemma \ref{lem:pre_lag_unstable}, the spaces $E^u(O)$ and $E^{ss}(O)$ are Lagrangian subspaces of the contact structure $\xi$ of $D$ at $O$, and $E^{cs}(O)$ is transverse to $\xi$ at $O$. Therefore, by passing to a Darboux chart we may assume that $D = D_{\on{std}}$ is the standard contact disk in $\R^{2n-1} = \R_z \times T^*\R^{n-1}$, $O$ is the origin and
\[
E^{cs}(O) = \text{span}(\partial_z)  \qquad\text{at the point $O$}
\]
We construct $G_r$ as follows. Let $G':\R^{2n-1} \to \R$ be the autonomous contact Hamiltonian generating the contact flow $\Psi_t$ given by
\[\Psi_t(z,x,y) = (e^tz,e^{t/2}x,e^{t/2}y)\]
We can construct $G_r$ by taking $rG$, multiplying by a function that is $1$ near $O$ and $0$ outside of a neighborhood of $O$, and then reparametrizing time so that the resulting contact Hamiltonian is compactly supported in the time parameter $(0,1)$. We thus acquire a 1-parameter family of contact Hamiltonians
\[G_r:[0,1] \times D \to \R \qquad\text{with}\qquad G_0 = 0 \text{ and }\Phi_{G_r} = \Psi \text{ near }O\]
such that $G_r$ is supported on $(0,1) \times D$ and $\Phi_{G_r} = \Psi_r$ in a neighborhood of $O$. Note that property (a) is immediate from the construction. Property (b) holds for $r \in [-\epsilon,\epsilon]$ for some small $\epsilon$, since the origin $O$ is fixed by $\Psi_r$ and hyperbolicity of a fixed point is an open condition. Finally, we argue that property (c) holds. To see this, note that for any vector $v \in E_{cs}(O)$, we have
\[
T_O\Phi_r(v) = T_O\Phi_{G_r} \circ T_O\Phi(v) = e^r \cdot \lambda_{cs}(O) \cdot v
\]
In particular, $e^r \cdot \lambda_{cs}(O)$ is a multiplier of $O$ as a fixed point of $\Phi_r$. Since the center-stable multiplier varies continuously and the spectrum is finite, it follows that $\lambda_{cs}(\Phi_r,O) = e^r \cdot \lambda_{cs}(O)$ for $r \in [-\epsilon,\epsilon]$ and sufficiently small $\epsilon$. 

\vspace{3pt}

\noindent {\bf Step 2: Construction Of $H$.} Next, we construct the family $H$. Fix a small, closed contact Darboux ball $B$ centered at $P$ and contained in the neighborhood $W$ of $P$. Let $\alpha$ be a contact form that is standard in this ball and let $\Psi$ denote the (partially defined) Reeb flow of this contact form. Finally, for convenience, let
\[K_r = W^u(\Phi_r,O) \qquad\text{and}\qquad K = W^u(O) = K_0\]
Note that the family $K_r$ is family of Legendrian sub-manifolds of $D$ by Lemma \ref{lem:isotropric_stable_unstable} that depends smoothly on $r$ (see Remark \ref{rmk:continuity_of_invariants}). Moreover, by the hypotheses of the proposition
\[
\xi_P = T_P\Lambda \oplus T_PK
\]
This implies that for small $\epsilon$, the map
\[
F:[-\epsilon,\epsilon]_\tau \times K \to Y \qquad\text{given by the Reeb flow $F(t,x) = \Psi_t(x)$}
\]
is an embedding transverse to $\Lambda$ at $P$. In particular, there is a smooth function $\tau:[-\epsilon,\epsilon]_r \to \R$ such that $\Lambda_r = \Psi_{\tau(r)}(\Lambda)$ intersects $K_r$ at $P$ and $\xi_P = T_P\Lambda_r \oplus T_PK_r$. 

\vspace{3pt}

Next, choose a weakly compatible metric $g$ on $\R^{2n-1}$. We apply Proposition \ref{prop:legendrian_dist} to find a smooth function $\epsilon)(r) > 0$ such that, for any $r$ and any $s$ with $|s| < \epsilon(r)$, we have
\[
\on{dist}(\Psi_{\tau(r) + s}(\Lambda) \cap B,K \cap B) = \on{dist}_g(\Psi_s(\Lambda_r) \cap B, K_r \cap B) = \on{dist}_g(\Psi_s(\Lambda_r \cap B), K_r \cap B) = |s|\]
By choosing $\epsilon$ small enough, we may assume that $\epsilon(r) > \epsilon$ and thus that the above inequality holds for all $(r,s) \in [-\epsilon,\epsilon]^2$. We now define
\[
H_{r,s}:[0,1] \times D \to D
\]
to be a smooth family of Hamiltonians supported in $(0,1) \times V$ such that the family of contactomorphisms generated by $H_{r,s}$ satisfies
\[
\Phi_{H_{r,s}} = \Psi_{\tau(r) + s} \qquad\text{in a neighborhood of $B$}
\]
Then for sufficiently small $\epsilon$, the family of Legendrians $\Lambda_{r,s} = \Phi_{H_{r,s}}(\Lambda)$ will satisfy $\on{dist}(\Lambda_{r,s} \cap B, W^u(\Phi_r,O) \cap B) = |s|$. Moreover, since the isotopy of Legendrians $s \mapsto \Lambda_{r,s}$ intersects $K_r$ tranversely at one point at $s = 0$, the signed distance changes sign as $s$ changes sign. Therefore property (d) follows and this concludes the proof. \end{proof}

\begin{construction}[Unfolding] \label{const:unfoldings} Let $C$ be a non-degenerate heterodimensional cycle in the characteristic foliation of $(\Sigma,\eta)$, as in Theorem \ref{thm:characteristic_unfoldings}. We now construct a family of embeddings 
\begin{equation} \label{eq:unfolding_definition}
\iota_{r,s}:\Sigma \to C\Sigma  \qquad\text{paramfetrized by $(r,s) \in \R^2$ sufficiently small}
\end{equation}
By Definition \ref{def:non_degeneracy} for non-degeneracy, there exists a choice of local sections, heteroclinic points and heteroclinic segments
\[
D_\pm \text{ as in (\ref{eq:choice_of_local_sections})} \qquad P^u_\pm \text{ and }P^s_\pm \text{ as in (\ref{eq:choice_of_heteroclinic_points})} \qquad\text{and}\qquad \Gamma_\pm \text{ as in (\ref{eq:choice_of_heteroclinic_segments})}
\]
such that Conditions \ref{cond:simple_fragile_heteroclinic}, \ref{cond:simple_robust_heteroclinic}, \ref{cond:extra_condition} and \ref{cond:saddle} all hold. We consider the contact embedding given by the Poincare return map of $C_-$
\[
\Phi = \on{Ret} C_-:U \to D_- \qquad\text{with hyperbolic fixed point}\qquad O = O_- = \Gamma_- \cap D_-
\]
and Legendrian given by the image of the stable manifold of $O_+$ under the transfer map of $\Gamma_-$.
\[
\Lambda = \on{Tr}\Gamma_-(W^s(O_+)) \qquad\text{with intersection point}\qquad P = P^s_-
\]
Note that Condition \ref{cond:simple_fragile_heteroclinic} implies that $\xi_P = T_P\Lambda \oplus T_PW^u(O)$. We can thus apply Proposition \ref{prop:unfolding_proposition} to the embedding $\Phi$ and Legendrian $\Lambda$ to acquire families of contact Hamiltonians
\[G_r:[0,1] \times D_- \to \R \qquad\text{and}\qquad H_{r,s}:[0,1] \times D_- \to \R\]
satisfying Proposition \ref{prop:unfolding_proposition}(a-d) supported in arbitrary neighborhoods $V$ of $O$ and $W$ of $P$. Next, choose a plugging domain $U_G$ that is well-positioned along the orbit $C_-$ and such that the outward boundary is $D_-$. Also choose a plugging domain $U_H$ that is well-positioned along $\Gamma_-$ and that is disjoint from $U_G$. By Lemma \ref{lem:transition_map_deformation} and Lemma \ref{lem:transition_deformation_for_hets}, we can insert the plug corresponding to $G_r$ along $U_G$ and a plug corresponding to $H_{r,s}$ along $U_H$ to get an family of embeddings
\begin{equation} 
\iota_{r,s}:\Sigma \to C\Sigma  \qquad\text{for $(r,s)$ sufficiently close to $(0,0)$}
\end{equation}
such that the continuation $C^{r,s}_-$ of $C_-$ and the continuation $\Gamma^{r,s}_-$ of $\Gamma_-$ have the following return map and transition map, respectively.
\begin{equation} \label{eq:return_and_transition_map_of_unfolding}
\on{Ret}C^{r,s}_- = \Phi_{G_r} \circ \on{Ret}C_- \qquad\text{and}\qquad \on{Tr} \Gamma^{r,s}_- = \Phi_{H_{r,s}} \circ \on{Tr}\Gamma_-
\end{equation}
Note that for (\ref{eq:return_and_transition_map_of_unfolding}) to hold, we must choose the open sets $U$ and $V$ to be sufficiently small neighborhoods of $O = O_-$ and $P = P^s_-$. This guarantees that the trajectory from any point $Q$ in $D_+$ to its image $\on{Tr}\Gamma_-$ in $D_-$ cannot pass through the support of the plug by $G_r$. Also note that the return maps of $C_+$ and $\Gamma_+$ are unchanged as long as $U$ and $V$ are small.
\begin{equation}
\on{Ret}C^{r,s}_+ = \on{Ret}C_+ \qquad\text{and}\qquad \on{Tr}\Gamma_+^{r,s} = \on{Tr}\Gamma_+
\end{equation}
\end{construction}

\begin{lemma} Let $\iota_{r,s}$ be the embeddings in Construction \ref{const:unfoldings}. Then the characteristic foliations $L_{r,s}$ of the contact Hamiltonian structures $\eta_{r,s} = \iota_{r,s}^*\xi$ form a proper unfolding in the sense of Definition \ref{def:proper_unfolding}.
\end{lemma}

\begin{proof} As in Proposition \ref{prop:unfolding_proposition}, we adopt the following notation.
\[
\Phi_r = \Phi_{G_r} \circ \Phi = \on{Ret}C^{r,s}_- \qquad\text{and}\qquad \Lambda_{r,s} = \Phi_{H_{r,s}}(\Lambda) = \on{Tr}\Gamma^{r,s}_-(W^s(O_+))
\]
First, we check that the function $\sigma(L_{r,s})$ is smooth in $r$ and $s$, with a regular value at $0$, where $\sigma$ is the splitting function in Definition \ref{def:splitting_function}. Let $B$ be the ball in Proposition \ref{prop:unfolding_proposition}(d). Then
\[
\sigma_g(L_{r,s}) = \on{sdist}\big(W^u(O_-^{r,s}) \cap B, \on{Tr}\Gamma^{r,s}_-(W^s(O_+^{r,s})) \cap B\big)
\]
Here $O^{r,s}_\pm$ is the continuation fixed point of $O_\pm$ for the line field $L_{r,s}$. Note that $W^u(O_-^{r,s})$ is precisely the unstable manifold $W^u(\Phi_r,O_-)$ of $O_-$ as a fixed point of $\Phi_r$, while $W^s(O_+^{r,s}) = W^s(O_+)$ since the return map of $C_+^{r,s}$ is simply $\on{Ret}C_+$. Therefore
\[
\sigma_g(L_{r,s}) = \on{sdist}\big(W^u(\Phi_r,O_-) \cap B, \Lambda_{r,s} \cap B) = s
\]
This shows that $\sigma_g(L_{r,s})$ is a smooth function in the parameters $r$ and $s$ with a regular value at $0$. Moreover, the sub-manifold $S$ of parameters with $\sigma_g(L_{r,s}) = 0$ is precisely the set where $s = 0$. 

\vspace{3pt}

Second, we compute the derivative of the multiplier ratio along $S$. We are interested in the function on $S = \{s = 0\}$ given by
\[
\theta(r) = -\frac{\on{log}|\lambda_{cs}(O_-^{r,s})|}{\on{log}|\lambda_{cu}(O_+^{r,s})|}|_{s = 0} 
\]
Again, since $\on{Ret}C_+^{r,s} = \on{Ret}C_+$ for all $r$ and $s$, we know that $\lambda_{cu}(O_+^{r,s}) = \lambda_{cs}(O_+)$ is independnt of $r$. On the other hand, $\lambda_{cs}(O^{r,s}_-)$ is simply the center-stable multipler $\lambda_{cs}(\Phi_r,O_-)$ of $O_-$ with respect to the return map $\Phi_r = \on{Ret}C_-^{r,s}$. Therefore, Proposition \ref{prop:unfolding_proposition}(c) implies that
\[
\theta(r) = -\frac{\on{log}|\lambda_{cs}(\Phi_r,O_-)|}{\on{log}|\lambda_{cu}(O_+)|} = -r \cdot \frac{\on{log}|\lambda_{cs}(O_-)|}{\on{log}|\lambda_{cu}(O_+)|}
\]
In particular, $\theta(r)$ is differentiable with a nowhere zero derivative. This verifies the second property of a proper unfolding and finishes the proof.\end{proof}

\bibliographystyle{hplain}
\bibliography{standard_bib}

\end{document}